\documentclass[reqno]{amsart} 
\usepackage{amsmath,amsthm,amssymb,amsfonts}
\usepackage{mathrsfs}
\usepackage{longtable}
\numberwithin{equation}{section}
\usepackage{color}
\usepackage{enumerate}
\usepackage{multirow} 
\newtheoremstyle{mystyle}
{}
{}
{\normalfont}
{ }
{\bfseries}
{}
{10pt}
{ }
\theoremstyle{mystyle}
\newtheorem{theorem}{Theorem}
\newtheorem{proposition}{Proposition}
\newtheorem{lemma}{Lemma} 
\newtheorem{remark}{Remark}

\usepackage{mathtools}
\usepackage[pdftex]{graphicx}

\def\Diag{\mathop{\rm Diag}\nolimits}
\def\tr{\mathop{\rm tr}\nolimits}
\def\vec{\mathop{\rm vec}\nolimits}
\def\vech{\mathop{\rm vech}\nolimits}
\def\det{\mathop{\rm det}\nolimits}
\def\rank{\mathop{\rm rank}\nolimits}

\allowdisplaybreaks[4]
\usepackage[foot]{amsaddr}
\usepackage[top=3cm, bottom=3cm, left=2.5cm, right=2.5cm]{geometry}
\usepackage[hang,small,bf]{caption}
\usepackage[subrefformat=parens]{subcaption}
\usepackage{comment}
\large
\title[Akaike-type information criterion of SEM for diffusion processes with jumps]{Akaike-type information criterion of SEM for jump-diffusion processes based on high-frequency data}
\author[S. Kusano]{Shogo Kusano $^{1}$}
\author[M. Uchida]{Masayuki Uchida $^{2,3,4}$}
\address{$^{1}$ Faculty of Advanced Science and Technology, Kumamoto University, Kumamoto, Japan}
\address{$^{2}$ Graduate School of Engineering Science, The University of Osaka, Toyonaka, Japan}
\address{$^{3}$ Center for Mathematical Modeling and Data Science (MMDS), The University of Osaka, Toyonaka, Japan}
\address{$^{4}$ CREST, Japan Science and Technology Agency, Tokyo, Japan}
\begin{document}
\begin{abstract}
\fontsize{8pt}{10pt}\selectfont
Structural equation modeling (SEM) is a statistical method used to investigate relationships among latent variables. In SEM, the model must be specified in advance. However, in practice, statisticians often have several candidate models and need to select the most appropriate one. Consequently, model selection is a key issue in SEM, and information criteria are commonly used to address this issue. In this study, we develop an Akaike-type information criterion of SEM for jump-diffusion processes, which enables model selection for SEM based on high-frequency data with jumps. Simulation studies are conducted to illustrate the finite-sample performance of the proposed method.
\end{abstract}
\maketitle

\section{Introduction}
\fontsize{10pt}{16pt}\selectfont
\setlength{\abovedisplayskip}{8pt}
\setlength{\belowdisplayskip}{8pt}
We consider a model selection problem of structural equation modeling (SEM) for jump-diffusion processes. Let $(\Omega, \mathcal{F},(\mathcal{F}_t)_{t\geq 0}, {\bf{P}})$ be a stochastic basis. First, we set the true model. The $p_1$ and $p_2$-dimensional observable processes $\{X_{1,0,t}\}_{t\geq 0}$ and $\{X_{2,0,t}\}_{t\geq 0}$ are defined by the following factor models:
\begin{align*}
    X_{1,0,t}&={\bf{\Lambda}}_{1,0}\xi_{0,t}+\delta_{0,t},\\
    X_{2,0,t}&={\bf{\Lambda}}_{2,0}\eta_{0,t}+\varepsilon_{0,t},
\end{align*}
where $\{\xi_{0,t}\}_{t\geq 0}$, $\{\delta_{0,t}\}_{t\geq 0}$, $\{\eta_{0,t}\}_{t\geq 0}$ and $\{\varepsilon_{0,t}\}_{t\geq 0}$ are $k_1$, $p_1$, $k_2$ and $p_2$-dimensional c{\`a}dl{\`a}g $(\mathcal{F}_t)$-adapted latent processes on the stochastic basis, $k_1\leq p_1$, $k_2\leq p_2$, and ${\bf{\Lambda}}_{1,0}\in\mathbb{R}^{p_1\times k_1}$ and 
${\bf{\Lambda}}_{2,0}\in\mathbb{R}^{p_2\times k_{2}}$ are constant loading matrices. 
We assume that the relationship between $\{\xi_{0,t}\}_{t\geq 0}$ and $\{\eta_{0,t}\}_{t\geq 0}$ 
is as follows:
\begin{align*}
    \eta_{0,t}={\bf{B}}_0\eta_{0,t}+{\bf{\Gamma}}_0\xi_{0,t}+\zeta_{0,t},
\end{align*}
where $\{\zeta_{0,t}\}_{t\geq 0}$ is a $k_2$-dimensional c{\`a}dl{\`a}g $(\mathcal{F}_t)$-adapted latent process on the stochastic basis, ${\bf{B}}_0\in\mathbb{R}^{k_2\times k_2}$ is a constant loading matrix whose diagonal elements are zero, and ${\bf{\Gamma}}_0\in\mathbb{R}^{k_2\times k_1}$ is a constant loading matrix.
It is supposed that ${\bf{\Lambda}}_{1,0}$ and ${\bf{\Lambda}}_{2,0}$ are of full column rank, and that ${\bf{\Psi}}_0=\mathbb{I}_{k_2}-{\bf{B}}_0$ is non-singular, where $\mathbb{I}_{k_2}$ is the identity matrix of size $k_2$. The stochastic processes $\{\xi_{0,t}\}_{t\geq 0}$, 
$\{\delta_{0,t}\}_{t\geq 0}$, $\{\varepsilon_{0,t}\}_{t\geq 0}$ and $\{\zeta_{0,t}\}_{t\geq 0}$ are assumed to satisfy the following stochastic differential equations:
\begin{align*}
    d\xi_{0,t}&=a_{1}(\xi_{0,t-})dt+{\bf{S}}_{1,0}d W_{1,t}+\int_{E_1}c_1(\xi_{0,t-},z_1)p_1(dt,dz_1),\quad \xi_{0,0}=x_{1,0},\\
    d\delta_{0,t}&=a_{2}(\delta_{0,t-})dt+{\bf{S}}_{2,0}d W_{2,t}+\int_{E_2}c_2(\delta_{0,t-},z_2)p_2(dt,dz_2),\quad \delta_{0,0}=x_{2,0},\\
    d\varepsilon_{0,t}&=a_{3}(\varepsilon_{0,t-})dt+{\bf{S}}_{3,0}d W_{3,t}+\int_{E_3}c_3(\varepsilon_{0,t-},z_3)p_3(dt,dz_3),\quad \varepsilon_{0,0}=x_{3,0},\\
    d\zeta_{0,t}&=a_{4}(\zeta_{0,t-})dt+{\bf{S}}_{4,0}d W_{4,t}+\int_{E_4}c_4(\zeta_{0,t-},z_4)p_4(dt,dz_4),\quad \zeta_{0,0}=x_{4,0}, 
\end{align*}
where $E_1=\mathbb{R}^{k_1} \backslash \{0\}$, $E_2=\mathbb{R}^{p_1} \backslash \{0\}$, 
$E_3=\mathbb{R}^{p_2} \backslash \{0\}$ and $E_4=\mathbb{R}^{k_2} \backslash \{0\}$. The initial values $x_{1,0}\in\mathbb{R}^{k_1}$, $x_{2,0}\in\mathbb{R}^{p_1}$, $x_{3,0}\in\mathbb{R}^{p_2}$ and $x_{4,0}\in\mathbb{R}^{k_2}$ are 
non-random vectors. For $i=1,2,3,4$, $\{W_{i,t}\}_{t\geq 0}$ is an $r_i$-dimensional standard $(\mathcal{F}_t)$-Wiener process, and $p_i(dt,dz_i)$ is a Poisson random measure on $\mathbb{R}_{+}\times E_i$ with the compensator $q_i(dt,dz_i)={\bf E_P}\bigl[p_i(dt,dz_i)\bigr]$.
The functions $a_1:\mathbb{R}^{k_1}\longrightarrow\mathbb{R}^{k_1}$, $a_2:\mathbb{R}^{p_1}\longrightarrow\mathbb{R}^{p_1}$, 
$a_3:\mathbb{R}^{p_2}\longrightarrow\mathbb{R}^{p_2}$, $a_4:\mathbb{R}^{k_2}\longrightarrow\mathbb{R}^{k_2}$, 
$c_1:\mathbb{R}^{k_1}\times E_1\longrightarrow\mathbb{R}^{k_1}$,
$c_2:\mathbb{R}^{p_1}\times E_2\longrightarrow\mathbb{R}^{p_1}$,
$c_3:\mathbb{R}^{p_2}\times E_3\longrightarrow\mathbb{R}^{p_2}$ and
$c_4:\mathbb{R}^{k_2}\times E_4\longrightarrow\mathbb{R}^{k_2}$
are unknown Borel functions, and ${\bf{S}}_{1,0}\in\mathbb{R}^{k_1\times r_1}$, ${\bf{S}}_{2,0}\in\mathbb{R}^{p_1\times r_2}$, ${\bf{S}}_{3,0}\in\mathbb{R}^{p_2\times r_3}$ and ${\bf{S}}_{4,0}\in\mathbb{R}^{k_2\times r_4}$ are constant coefficient matrices. Set the volatility matrices of $\{\xi_{0,t}\}_{t\geq 0}$, $\{\delta_{0,t}\}_{t\geq 0}$, $\{\varepsilon_{0,t}\}_{t\geq 0}$ and $\{\zeta_{0,t}\}_{t\geq 0}$ as ${\bf{\Sigma}}_{\xi\xi,0}={\bf{S}}_{1,0}{\bf{S}}_{1,0}^{\top}$, 
${\bf{\Sigma}}_{\delta\delta,0}={\bf{S}}_{2,0}{\bf{S}}_{2,0}^{\top}$, ${\bf{\Sigma}}_{\varepsilon\varepsilon,0}={\bf{S}}_{3,0}{\bf{S}}_{3,0}^{\top}$ and 
${\bf{\Sigma}}_{\zeta\zeta,0}={\bf{S}}_{4,0}{\bf{S}}_{4,0}^{\top}$, respectively. Here, $\top$ stands for the transpose of a matrix. Assume that for any $t\geq 0$, $\mathcal{F}_t$, $\sigma(W_{i,u}-W_{i,t};\ u\geq t) \ (i=1,2,3,4)$, and
\begin{align*}
    \sigma\bigl(p_i(A_i\cap ((t,\infty)\times E_i)); A_i\subset\mathbb{R}_{+}\times E_i \mbox{\ is a Borel set}\bigr)\quad (i=1,2,3,4)
\end{align*}
are independent. For $i=1,2,3,4$, we assume that $q_i(dt,dz_i)$ has a representation such that $q_i(dt,dz_i)=f_i(z_i)dz_idt$, and $f_i(z_i)=\lambda_{i,0}F_i(z_i)$, where $\lambda_{i,0}>0$ is an unknown value, and $F_i(z_i)$ is an unknown probability density. Let $p=p_1+p_2$ and $X_{0,t}=(X_{1,0,t}^{\top},X_{2,0,t}^{\top})^{\top}$. For simplicity, we write $X_{0,t}$ as $X_t$. $\mathbb{X}_n=\{X_{t_i^n}\}_{i=0}^n$ are discrete observations, where $t_i^n=ih_n$. We suppose that $T=nh_n$ is fixed. Define
\begin{align*}
    {\bf{\Sigma}}_0=\begin{pmatrix}
    {\bf{\Sigma}}_0^{11} & {\bf{\Sigma}}_0^{12}\\
    {\bf{\Sigma}}_0^{12\top} & {\bf{\Sigma}}_0^{22}
    \end{pmatrix},
\end{align*}
where
\begin{align*}
    {\bf{\Sigma}}_0^{11}&={\bf{\Lambda}}_{1,0}{\bf{\Sigma}}_{\xi\xi,0}{\bf{\Lambda}}_{1,0}^{\top}+{\bf{\Sigma}}_{\delta\delta,0},\\
    {\bf{\Sigma}}_0^{12}&={\bf{\Lambda}}_{1,0}{\bf{\Sigma}}_{\xi\xi,0}{\bf{\Gamma}}_0^{\top}{\bf{\Psi}}_0^{-1\top}{\bf{\Lambda}}_{2,0}^{\top},\\
    {\bf{\Sigma}}_0^{22}&={\bf{\Lambda}}_{2,0}{\bf{\Psi}}_0^{-1}({\bf{\Gamma}}_0{\bf{\Sigma}}_{\xi\xi,0}{\bf{\Gamma}}_0^{\top}+{\bf{\Sigma}}_{\zeta\zeta,0}){\bf{\Psi}}_0^{-1\top}{\bf{\Lambda}}_{2,0}^{\top}+{\bf{\Sigma}}_{\varepsilon\varepsilon,0}.
\end{align*}
Suppose that ${\bf{\Sigma}}_0$ is positive definite. Next, we define a parametric model of Model $m\in\{1,\ldots,M\}$. Set $\theta_m\in\Theta_m\subset\mathbb{R}^{q_m}$ as the parameter of Model $m$, where $\Theta_m$ is a bounded open subset of $\mathbb{R}^{q_m}$. 
Note that the parameter $\theta_m$ consists only of the unknown elements of the loading matrices and volatility matrices of the latent processes in Model $m$. See Kusano and Uchida \cite{Kusano(AIC), Kusano(BIC)} for further details. The stochastic processes $\{X^{\theta}_{1,m,t}\}_{t\geq 0}$ and $\{X^{\theta}_{2,m,t}\}_{t\geq 0}$ are characterized by the following factor models:
\begin{align*}
    X^{\theta}_{1,m,t}&={\bf{\Lambda}}^{\theta}_{1,m}\xi^{\theta}_{m,t}+\delta^{\theta}_{m,t}, \\
    X^{\theta}_{2,m,t}&={\bf{\Lambda}}^{\theta}_{2,m}\eta^{\theta}_{m,t}+\varepsilon^{\theta}_{m,t},
\end{align*}
where $\{\xi^{\theta}_{m,t}\}_{t\geq 0}$, $\{\delta^{\theta}_{m,t}\}_{t\geq 0}$, $\{\eta^{\theta}_{m,t}\}_{t\geq 0}$ and $\{\varepsilon^{\theta}_{m,t}\}_{t\geq 0}$ are $k_1$, $p_1$, $k_2$ and $p_2$-dimensional c{\`a}dl{\`a}g $(\mathcal{F}_t)$-adapted latent processes, and ${\bf{\Lambda}}^{\theta}_{1,m}\in\mathbb{R}^{p_1\times k_1}$ and 
${\bf{\Lambda}}^{\theta}_{2,m}\in\mathbb{R}^{p_2\times k_{2}}$ are constant loading matrices. Moreover, the relationship between $\{\eta^{\theta}_{m,t}\}_{t\geq 0}$ and $\{\xi^{\theta}_{m,t}\}_{t\geq 0}$ is given by
\begin{align*}
    \eta^{\theta}_{m,t}={\bf{B}}_m^{\theta}\eta^{\theta}_{m,t}+{\bf{\Gamma}}_m^{\theta}\xi^{\theta}_{m,t}+\zeta^{\theta}_{m,t},
\end{align*}
where $\{\zeta^{\theta}_{m,t}\}_{t\geq 0}$ is a $k_2$-dimensional c{\`a}dl{\`a}g $(\mathcal{F}_t)$-adapted latent process, ${\bf{B}}_m^{\theta}\in\mathbb{R}^{k_2\times k_2}$ is a constant loading matrix whose diagonal elements are zero, and ${\bf{\Gamma}}_m^{\theta}\in\mathbb{R}^{k_2\times k_1}$ is a constant loading matrix.
We assume that ${\bf{\Lambda}}^{\theta}_{1,m}$ and ${\bf{\Lambda}}^{\theta}_{2,m}$ are of full column rank, and that ${\bf{\Psi}}^{\theta}_m=\mathbb{I}_{k_2}-{\bf{B}}_m^{\theta}$ is non-singular. Suppose that the stochastic processes $\{\xi^{\theta}_{m,t}\}_{t\geq 0}$, 
$\{\delta^{\theta}_{m,t}\}_{t\geq 0}$, $\{\varepsilon^{\theta}_{m,t}\}_{t\geq 0}$ and $\{\zeta^{\theta}_{m,t}\}_{t\geq 0}$ are defined as the following stochastic differential equations:
\begin{align*}
    d\xi_{m,t}^{\theta}&=a_{1}(\xi^{\theta}_{m,t-})dt+{\bf{S}}^{\theta}_{1,m}d W_{1,t}+\int_{E_1}c_1(\xi^{\theta}_{m,t-},z_1)p_1(dt,dz_1),\quad \xi^{\theta}_{m,0}=x_{1,0},\\
    d\delta_{m,t}^{\theta}&=a_{2}(\delta^{\theta}_{m,t-})dt+{\bf{S}}^{\theta}_{2,m}d W_{2,t}+\int_{E_2}c_2(\delta^{\theta}_{m,t-},z_2)p_2(dt,dz_2),\quad \delta^{\theta}_{m,0}=x_{2,0},\\
    d\varepsilon_{m,t}^{\theta}&=a_{3}(\varepsilon^{\theta}_{m,t-})dt+{\bf{S}}^{\theta}_{3,m}d W_{3,t}+\int_{E_3}c_3(\varepsilon^{\theta}_{m,t-},z_3)p_3(dt,dz_3),\quad \varepsilon^{\theta}_{m,0}=x_{3,0},\\
    d\zeta^{\theta}_{m,t}&=a_{4}(\zeta^{\theta}_{m,t-})dt+{\bf{S}}^{\theta}_{4,m}d W_{4,t}+\int_{E_4}c_4(\zeta^{\theta}_{m,t-},z_4)p_4(dt,dz_4),\quad \zeta^{\theta}_{m,0}=x_{4,0},
\end{align*}
where ${\bf{S}}^{\theta}_{1,m}\in\mathbb{R}^{k_1\times r_1}$, ${\bf{S}}^{\theta}_{2,m}\in\mathbb{R}^{p_1\times r_2}$, ${\bf{S}}^{\theta}_{3,m}\in\mathbb{R}^{p_2\times r_3}$ and ${\bf{S}}^{\theta}_{4,m}\in\mathbb{R}^{k_2\times r_4}$ are constant coefficient matrices.
The volatility matrices of $\{\xi^{\theta}_{m,t}\}_{t\geq 0}$, $\{\delta^{\theta}_{m,t}\}_{t\geq 0}$, $\{\varepsilon^{\theta}_{m,t}\}_{t\geq 0}$ and $\{\zeta^{\theta}_{m,t}\}_{t\geq 0}$ are defined by ${\bf{\Sigma}}^{\theta}_{\xi\xi,m}={\bf{S}}^{\theta}_{1,m}{\bf{S}}_{1,m}^{\theta\top}$, 
${\bf{\Sigma}}^{\theta}_{\delta\delta,m}={\bf{S}}^{\theta}_{2,m}{\bf{S}}_{2,m}^{\theta\top}$, ${\bf{\Sigma}}^{\theta}_{\varepsilon\varepsilon,m}={\bf{S}}^{\theta}_{3,m}{\bf{S}}_{3,m}^{\theta\top}$ and 
${\bf{\Sigma}}^{\theta}_{\zeta\zeta,m}={\bf{S}}^{\theta}_{4,m}{\bf{S}}_{4,m}^{\theta\top}$, respectively. Set $X_{m,t}^{\theta}=(X_{1,m,t}^{\theta\top},X_{2,m,t}^{\theta\top})^{\top}$. Let
\begin{align*}
    {\bf{\Sigma}}_m(\theta_m)=\begin{pmatrix}
    {\bf{\Sigma}}_m^{11}(\theta_m) & {\bf{\Sigma}}_m^{12}(\theta_m)\\
    {\bf{\Sigma}}_m^{12}(\theta_m)^{\top} & {\bf{\Sigma}}_m^{22}(\theta_m)
    \end{pmatrix},
\end{align*}
where 
\begin{align*}                      
    \qquad\qquad{\bf{\Sigma}}^{11}_m(\theta_m)&
    ={\bf{\Lambda}}^{\theta}_{1,m}{\bf{\Sigma}}^{\theta}_{\xi\xi,m}{\bf{\Lambda}}_{1,m}^{\theta\top}
    +{\bf{\Sigma}}^{\theta}_{\delta\delta,m},\\  {\bf{\Sigma}}^{12}_m(\theta_m)&={\bf{\Lambda}}^{\theta}_{1,m}{\bf{\Sigma}}^{\theta}_{\xi\xi,m}{\bf{\Gamma}}_m^{\theta\top}{\bf{\Psi}}_m^{\theta-1\top}{\bf{\Lambda}}_{2,m}^{\theta\top},\\
    {\bf{\Sigma}}^{22}_m(\theta_m)&={\bf{\Lambda}}^{\theta}_{2,m}{\bf{\Psi}}_m^{\theta-1}({\bf{\Gamma}}_m^{\theta}{\bf{\Sigma}}^{\theta}_{\xi\xi,m}{\bf{\Gamma}}_m^{\theta\top}+{\bf{\Sigma}}^{\theta}_{\zeta\zeta,m}){\bf{\Psi}}_m^{\theta-1\top}{\bf{\Lambda}}_{2,m}^{\theta\top}+{\bf{\Sigma}}^{\theta}_{\varepsilon\varepsilon,m}.
\end{align*}
It is supposed that there exists $\theta_{m,0}\in \Theta_m$ such that ${\bf{\Sigma}}_0={\bf{\Sigma}}_m(\theta_{m,0})$. In addition, we assume that ${\bf{\Sigma}}_m(\theta_m)$ for any $\theta_m\in \bar{\Theta}_m$ is positive definite.

SEM is a statistical method for analyzing relationships among latent variables, and 
it has been widely applied in various fields including psychology, biology, and economics.
SEM has been extensively studied in the context of independent and identically distributed (i.i.d.) models, and extensions to time-series data have also been explored. For example, Czir{\'a}ky \cite{Cziraky(2004)} studied SEM for ARMA models, and Kusano and Uchida \cite{Kusano(JJSD)} studied SEM for diffusion processes based on high-frequency data. It should be noted that high-frequency data are commonly observed in finance, and that statistical analyses based on continuous-time stochastic process models are widely used. For further details on SEM, see, e.g., Everitt \cite{Everitt(1984)}, Mueller \cite{Mueller(1999)} and Huang et al. \cite{Huang(2017)}.

In SEM, statisticians must specify the model in advance based on the theoretical framework of the relevant research field. Although, in practice, they often have several candidate models, they must select the most appropriate one. To address this problem, 
various information criteria for SEM have been extensively studied. In the context of SEM, the Akaike information criterion (AIC) and the Bayesian information criterion (BIC) are widely used for model selection; see, e.g., Huang \cite{Huang AIC(2017)}.

In recent years, Akaike-type information criteria for continuous-time stochastic processes based on high-frequency data have been developed. Since the transition density of a diffusion process is generally unavailable in explicit form, it is difficult to construct AIC 
for discretely observed diffusion processes. Using a contrast function based on a locally Gaussian approximation of the transition density, Uchida \cite{Uchida(2010)} proposed a contrast-based information criterion, and established the asymptotic properties of the difference between the contrast-based information criteria. Uehara \cite{Uehara(2025)} studied an Akaike-type information criterion for ergodic jump-diffusion processes based on a threshold-based quasi-likelihood. Eguchi and Masuda \cite{Eguchi(2024)} proposed a quasi-AIC for semi-parametric L$\acute{e}$vy driven stochastic differential equations. Kusano and Uchida \cite{Kusano(AIC)} studied a quasi-AIC of SEM for diffusion processes. For further developments of information criteria for
continuous-time stochastic processes, see also  Fujii and Uchida \cite{Fuji(2014)}, Eguchi and Masuda \cite{Eguchi(2018)}, Eguchi and Uehara \cite{Eguchi(2021)}, Kusano and Uchida \cite{Kusano(BIC)} and the references therein. 

Since discontinuous sample paths are frequently observed in practice, statistical inference for jump-diffusion processes has also been studied extensively, and jump-diffusion models have been employed in various fields, including financial econometrics, physics, and hydrology; see, e.g., Shimizu and Yoshida \cite{Shimizu-Yoshida_JP}, Mancini \cite{Mancini(2004)}, Shimizu and Yoshida \cite{Shimizu(2006)}, Ogihara and Yoshida \cite{Ogihara(2011)}, Inatsugu and Yoshida \cite{Inatsugu(2021)} and Amorino and Gloter \cite{Amorino(2021)}. Recently, Kusano and Uchida \cite{Kusano(jump)} studied SEM for jump-diffusion processes, enabling the analysis of relationships among latent processes based on high-frequency data with jumps. However, to the best of our knowledge, model selection of SEM for jump-diffusion processes has not yet been investigated. As we mentioned above, information criteria play an important role in SEM, and jump-diffusion models provide powerful tools for capturing a wide range of stochastic phenomena. Hence, in this paper, we propose an Akaike-type information criterion of SEM for jump-diffusion processes and establish its asymptotic properties.

The rest of this paper is organized as follows. In Section \ref{notation}, we provide notation and assumptions. In Section \ref{main}, we propose an Akaike-type information criterion of SEM for jump-diffusion processes and provide its theoretical justification. We also consider the case in which the set of candidate models includes some (but not all) misspecified parametric models, and establish the asymptotic properties.
In Section \ref{simulation}, an example and simulation studies are presented. Section \ref{proofs} is devoted to the proofs of the results stated in Section \ref{main}.
\section{Notation and Assumptions}\label{notation}
First, we introduce the notation used throughout this paper. For a vector $v$, $v^{(i)}$ denotes its $i$-th element, $\Diag v$ represents the diagonal matrix whose 
$i$-th diagonal entry is $v^{(i)}$, and we define $|v|=\sqrt{\tr(vv^\top)}$. For a matrix $M$, $M_{ij}$ denotes the $(i,j)$-th entry of $M$, and we set $|M|=\sqrt{\tr(MM^\top)}$. $\otimes$ denotes the Kronecker product. For a symmetric matrix $M$, $\vec M$ and $\vech M$ indicate the vectorization and half-vectorization of $M$, respectively.  For a $p$-dimensional symmetric matrix $M$, the matrix $\mathbb{D}_p$ is the $p^2\times \bar{p}$ duplication matrix satisfying $\vec{M}=\mathbb{D}_{p}\vech{M}$, where $\bar{p}=p(p+1)/2$. Its Moore-Penrose inverse is denoted by $\mathbb{D}_p^{+}$. We write $\mathcal{M}_{p}^{+}$ for the set of all $p\times p$ real positive definite matrices. Let $ C^{k}_{\uparrow}(\mathbb R^{d})$ denote the space of all functions $f$ satisfying the following conditions:
\begin{itemize}
    \item[(i)] $f$ is continuously differentiable in $x\in \mathbb{R}^d$ up to order $k$. 
    \item[(ii)] $f$ and all its derivatives are of polynomial growth in $x\in \mathbb{R}^d$.
\end{itemize}
The symbols $\stackrel{p}{\longrightarrow}$ and $\stackrel{d}{\longrightarrow}$ indicate
convergence in probability and in distribution, respectively. A $d$-dimensional normal random variable with mean $\mu\in\mathbb{R}^d$ and covariance matrix $\Sigma\in\mathcal{M}_d^{+}$ is written as $N_{d}(\mu,\Sigma)$. A  $d$-dimensional standard normal random variable is denoted by $Z_d$, and $\mathbb{E}$ is the expectation under the distribution of $Z_d$. For a stochastic process $\{S_t\}_{t\geq 0}$,  we define $\Delta_i^n S=S_{t_i^n}-S_{t_{i-1}^n}$. Let $(d_1,d_2,d_3,d_4)=(k_1,p_1,p_2,k_2)$. Next, we introduce the following assumptions.
\begin{enumerate}
    \vspace{2mm}
    \item[\bf{[A1]}]
    For $i=1,2,3,4$, there exist $L_i>0$ and $\zeta_i(z_i)>0$ of at most polynomial growth in $z_i$ such that
    \begin{align*}
        |a_i(x_i)-a_i(y_i)|\leq L_i|x_i-y_i|,\ |c_i(x_i,z_i)-c_i(y_i,z_i)|\leq \zeta_i(z_i)|x_i-y_i|
    \end{align*}
    and
    \begin{align*}
        |c_i(x_i,z_i)|\leq \zeta_i(z_i)(1+|x_i|)
    \end{align*}
    for any $x_i,y_i\in \mathbb{R}^{d_i}$ and $z_i\in E_i$.
    \vspace{2mm}
    \item[\bf{[A2]}]
     For $i=1,2,3,4$, $a_i\in C_{\uparrow}^4(\mathbb{R}^{d_i})$.
     \vspace{2mm}
    \item[\bf{[A3]}] For $i=1,2,3,4$, there exist $r_i>0$ and $K_i>0$ such that
    \begin{align*}
        f_i(z_i){\bf{1}}_{\{|z_i|\leq r_i\}}\leq K_i|z_i|^{1-d_i}
    \end{align*}
    and
    \begin{align*}
        \int |z_i|^{L}f_i(z_i)dz_i<\infty
    \end{align*}
    for any $L\geq 1$.
    \vspace{2mm}
    \item[\bf{[A4]}]
    For $i=1,2,3,4$, it follows that
    \begin{align*}
        \inf_{x_i}|c_i(x_i,z_i)|\geq c_{i,0}|z_i|
    \end{align*}
    for some $c_{i,0}>0$ near the origin.
\end{enumerate}
\begin{remark}
{\bf{[A1]}}-{\bf{[A4]}} are the standard assumptions for jump-diffusion processes; see Kusano and Uchida \cite{Kusano(jump)}
for further details.
\end{remark}
\section{Main results}\label{main}
In this paper, we determine the presence of jumps in the interval $(t_{i-1}^n,t_i^n]$
based on the increment $|\Delta_i^n X|$; see, e.g., Shimizu and Yoshida \cite{Shimizu(2006)} and Ogihara and Yoshida \cite{Ogihara(2011)}. That is, we consider that a jump occurs in the interval when $|\Delta_i^n X|>Dh_n^{\rho}$ and no jump occurs when $|\Delta_i^n X|\leq Dh_n^{\rho}$, where $\rho\in[0,1/2)$ and $D>0$. We define the quasi-likelihood as
\begin{align*}
    {\bf{L}}_{m,n}(\mathbb{X}_n,\theta_m)=\exp\bigl\{{\bf{H}}_{m,n}(\mathbb{X}_n,\theta_m)\bigr\},
\end{align*}
where 
\begin{align*}
    {\bf{H}}_{m,n}(\mathbb{X}_n,\theta_m)&={\bf{H}}_{m,n}(\theta_m)\\
    &=-\frac{1}{2h_n}\sum_{i=1}^n(\Delta_{i}^n X)^{\top}{\bf{\Sigma}}_m(\theta_m)^{-1}(\Delta_{i}^n X){\bf{1}}_{\{|\Delta_{i}^n X|\leq Dh_n^{\rho}\}}\\
    &\qquad\qquad\qquad\qquad\qquad-\frac{1}{2}\sum_{i=1}^n\log\det {\bf{\Sigma}}_m(\theta_m){\bf{1}}_{\{|\Delta_i^n X|\leq Dh_n^{\rho}\}}.
\end{align*}
For details on the quasi-likelihood, see Kusano and Uchida \cite{Kusano(jump)}. Moreover, we assume that $\rho\in[3/8,1/2)$. The quasi-likelihood estimator $\hat{\theta}_{m,n}$ is defined by
\begin{align*}
    {\bf{H}}_{m,n}(\hat{\theta}_{m,n})=\sup_{\theta_m\in{\bar{\Theta}}_m}{\bf{H}}_{m,n}(\theta_m).
\end{align*}
Note that ${\bf{H}}_{m,n}(\theta_m)$ is three-times continuously differentiable with respect to $\theta_m$, and that ${\bf{H}}_{m,n}(\theta_m)$, $\partial_{\theta_m}{\bf{H}}_{m,n}(\theta_m)$, $\partial^2_{\theta_m}{\bf{H}}_{m,n}(\theta_m)$, and $\partial^3_{\theta_m}{\bf{H}}_{m,n}(\theta_m)$ can be continuously extended to the boundary $\partial\Theta$. 
The aim of an Akaike-type information criterion is to select the model that maximizes
\begin{align}
    {\bf{E}}_{\mathbb{X}_n}\Bigl[{\bf{E}}_{\mathbb{Z}_n}\Bigl[ {\bf{H}}_{m,n}\bigl(\mathbb{Z}_n,\hat{\theta}_{m,n}({\mathbb{X}_{n}})\bigr)\Bigr]\Bigr], \label{EZ}
\end{align}
where $\mathbb{Z}_n$ is an i.i.d. copy of $\mathbb{X}_n$, and ${\bf{E}}_{\mathbb{X}_n}$ and
${\bf{E}}_{\mathbb{Z}_n}$ denote expectations under the law of $\mathbb{X}_n$ and $\mathbb{Z}_n$, respectively. To construct an asymptotically unbiased estimator of (\ref{EZ}), it is necessary to establish the moment convergence of the quasi-maximum likelihood estimator. Set
\begin{align*}
    {\bf{Z}}_{m,n}(u_m;\theta_{m,0})=\exp\Biggl\{{\bf{H}}_{m,n}\biggl(\theta_{m,0}+\frac{1}{\sqrt{n}}u_m\biggr)-{\bf{H}}_{m,n}(\theta_{m,0})\Biggr\}
\end{align*}
for $u_m\in\mathbb{U}_{m,n}$, where 
\begin{align*}
    {\mathbb{U}}_{m,n}=\left\{u_m\in\mathbb{R}^{q_m}:\  \theta_{m,0}+\frac{1}{\sqrt{n}}u_m\in\Theta_m\right\}.
\end{align*}
For $r>0$, we define
\begin{align*}
    {\bf{V}}_{m,n}(r)=\Bigl\{u_m\in{\mathbb{U}}_{m,n}: r\leq |u_m|\Bigr\}
\end{align*}
and
\begin{align*}
    {\bf{Y}}_m(\theta_m)=-\frac{1}{2}\tr\Bigl({\bf{\Sigma}}_m(\theta_m)^{-1}{\bf{\Sigma}}_m(\theta_{m,0})-\mathbb{I}_{p}\Bigr)-\frac{1}{2}
    \log\frac{\det{\bf{\Sigma}}_m(\theta_m)}{\det{\bf{\Sigma}}_m(\theta_{m,0})}.
\end{align*}
Let
\begin{align*}
    {\bf{I}}_m(\theta_{m,0})=\Delta_{m,0}^{\top}{\bf{W}}_m(\theta_{m,0})^{-1}\Delta_{m,0},\ \ 
    {\bf{W}}_m(\theta_{m,0})=2\mathbb{D}^{+}_{p}\bigl({\bf{\Sigma}}_m(\theta_{m,0})
    \otimes{\bf{\Sigma}}_m(\theta_{m,0})\bigr)\mathbb{D}^{+\top}_{p},
\end{align*}
where
\begin{align*}
    \Delta_{m,0}=\left.\frac{\partial}{\partial\theta_m^{\top}}\vech {\bf{\Sigma}}_m(\theta_m)\right|_{\theta_m=\theta_{m,0}}.
\end{align*}
Moreover, we make the following assumption.
\begin{enumerate}
    \item[\bf{[B1]}] There exists a constant $\chi>0$ such that
        \begin{align*}
         {\bf{Y}}_m(\theta_m)\leq -\chi\bigl|\theta_m-\theta_{m,0}\bigr|^2
        \end{align*}
        for all $\theta_m\in\Theta_m$.
\end{enumerate}
\begin{remark}
If the usual identifiability condition 
\begin{align*}
    {\bf{\Sigma}}_m(\theta_{m})={\bf{\Sigma}}_m(\theta_{m,0})\Longrightarrow\theta_m=\theta_{m,0}
\end{align*}
and 
\begin{align*}
    \rank\Delta_{m,0}=q_m
\end{align*}
are satisfied, then {\bf{[B1]}} holds; see Kusano and Uchida \cite{Kusano(AIC)} for further details.
\end{remark}
The following theorem holds for the quasi-likelihood random field ${\bf{Z}}_{m,n}$.
\begin{theorem}\label{Zine}
Suppose that {\bf{[A1]}}-{\bf{[A4]}} and {\bf{[B1]}}. Then, for all $L>0$, there exists $C_L>0$ such that
\begin{align*}
    {\bf{P}}\left(\sup_{u_m\in {\bf{V}}_{m,n}(r)}{\bf{Z}}_{m,n}(u_m;\theta_{m,0})\geq e^{-r}\right)\leq\frac{C_L}{r^L}
\end{align*}
for all $r>0$ and $n\in\mathbb{N}$.
\end{theorem}
\noindent
Theorem \ref{Zine} together with Theorem 2 in Kusano and Uchida \cite{Kusano(jump)} implies the following result.
\begin{proposition}\label{moment}
Suppose that {\bf{[A1]}}-{\bf{[A4]}} and {\bf{[B1]}}. Then, for all $L>0$,
\begin{align*}
    \sup_{n\in\mathbb{N}}{\bf{E}}_{\mathbb{X}_n}\Biggl[\Bigl|\sqrt{n}(\hat{\theta}_{m,n}-\theta_{m,0})\Bigr|^L\Biggr]<\infty
\end{align*}
and for $f_m\in C_{\uparrow}(\mathbb{R}^{q_m})$,
\begin{align*}
   {\bf{E}}_{\mathbb{X}_n}\biggl[f_m\Bigl(\sqrt{n}(\hat{\theta}_{m,n}-\theta_{m,0})\Bigr)\biggr]\longrightarrow\mathbb{E}\biggl[f_m\Bigl({\bf{I}}_m(\theta_{m,0})^{-1/2}Z_{q_m}\Bigr)\biggr]
\end{align*}
as $n\longrightarrow\infty$.
\end{proposition}
\noindent
Using Proposition \ref{moment}, we have the following theorem.
\begin{theorem}\label{QAICtheorem}
Suppose that {\bf{[A1]}}-{\bf{[A4]}} and {\bf{[B1]}}. Then, as $n\longrightarrow\infty$,
\begin{align*}
    {\bf{E}}_{\mathbb{X}_n}\Bigl[{\bf{H}}_{m,n}\bigl(\mathbb{X}_{n},\hat{\theta}_{m,n}({\mathbb{X}_{n}})\bigr)\Bigr]-{\bf{E}}_{\mathbb{X}_n}\Bigl[{\bf{E}}_{\mathbb{Z}_n}\Bigl[{\bf{H}}_{m,n}\bigl(\mathbb{Z}_n,\hat{\theta}_{m,n}({\mathbb{X}_{n}})\bigr)\Bigr]\Bigr]=q_m+o(1).
\end{align*}
\end{theorem}
\noindent
Theorem \ref{QAICtheorem} implies that
\begin{align*}
    {\bf{H}}_{m,n}\bigl(\mathbb{X}_{n},\hat{\theta}_{m,n}({\mathbb{X}_{n}})\bigr)-q_m
\end{align*}
is an asymptotically unbiased estimator of (\ref{EZ}). Therefore, to estimate the model that minimizes
\begin{align*}
    -2{\bf{E}}_{\mathbb{X}_n}\Bigl[{\bf{E}}_{\mathbb{Z}_n}
    \Bigl[{\bf{H}}_{m,n}\bigl(\mathbb{Z}_n,\hat{\theta}_{m,n}({\mathbb{X}_{n}})\bigr)\Bigr]\Bigr],
\end{align*}
we define the Akaike-type information criterion as 
\begin{align*}
    {\bf{QAIC}}_n(m)=-2{\bf{H}}_{m,n}\bigl(\mathbb{X}_n,\hat{\theta}_{m,n}({\mathbb{X}_{n}})\bigr)+2q_m
\end{align*}
and select the optimal model $\hat{m}_n$ by 
\begin{align}
    {\bf{QAIC}}_n(\hat{m}_n)=\min_{m=1,\ldots,M}{\bf{QAIC}}_n(m). \label{hatm}
\end{align}
Next, we consider the setting in which the set of candidate models includes some (but not all) misspecified parametric models; that is, there exists $m\in\{1,\ldots,M\}$ such that for any $\theta_m\in\Theta_m$,
\begin{align*}
    {\bf{\Sigma}}_0\neq{\bf{\Sigma}}_m(\theta_m). 
\end{align*}
The optimal parameter $\bar{\theta}_{m}$ is defined by
\begin{align*}
    {\bf{H}}_{m}(\bar{\theta}_{m})=\sup_{\theta_m\in\Theta_m}{\bf{H}}_{m}(\theta_{m}),
\end{align*}
where
\begin{align*}
    {\bf{H}}_{m}(\theta_m)
    &=-\frac{1}{2}\tr\Bigl({\bf{\Sigma}}_m(\theta_m)^{-1}{\bf{\Sigma}}_0\Bigr)-\frac{1}{2}
    \log\det{\bf{\Sigma}}_m(\theta_m).
\end{align*}
Define the set of correctly specified models as
\begin{align*}
    \mathfrak{M}=\biggl\{m\in\{1,\ldots,M\}\ \Big|\ \mbox{There exists}\ \theta_{m,0}\in\Theta_{m}\ \mbox{such that}\ {\bf{\Sigma}}_0={\bf{\Sigma}}_{m}(\theta_{m,0}).\biggr\}
\end{align*}
and the set of misspecified models as $\mathfrak{M}^{c}=\{1,\ldots,M\}\backslash \mathfrak{M}$, respectively. Furthermore, we impose the following assumption.
\begin{enumerate}
    \vspace{3mm}
    \item[{\bf{[B2]}}] ${\bf{H}}_{m}(\theta_m)={\bf{H}}_{m}(\bar{\theta}_{m})\Longrightarrow \ \theta_m=\bar{\theta}_{m}$.
    \vspace{3mm}
\end{enumerate}
The following asymptotic result holds for $\hat{m}_n$ defined in (\ref{hatm}).
\begin{proposition}\label{miss}
Suppose that {\bf{[A1]}}-{\bf{[A4]}} and {\bf{[B2]}}. Then, as $n\longrightarrow\infty$,
\begin{align*}
{\bf{P}}\Bigl(\hat{m}_n\in\mathfrak{M}^c\Bigr)
    &\longrightarrow 0.
\end{align*}
\end{proposition}
\section{Examples and Simulation results}\label{simulation}
\subsection{True model}
The $4$ and $8$-dimensional observable processes $\{X_{1,0,t}\}_{t\geq 0}$ and $\{X_{2,0,t}\}_{t\geq 0}$ are defined by the following factor models:
\begin{align*}
    X_{1,0,t}&={\bf{\Lambda}}_{1,0}\xi_{0,t}+\delta_{0,t},\\
    X_{2,0,t}&={\bf{\Lambda}}_{2,0}\eta_{0,t}+\varepsilon_{0,t},
\end{align*}
where  $\{\xi_{0,t}\}_{t\geq 0}$, $\{\delta_{0,t}\}_{t\geq 0}$, $\{\eta_{0,t}\}_{t\geq 0}$ and $\{\varepsilon_{0,t}\}_{t\geq 0}$ are $1$, $4$, $2$ and $8$-dimensional latent processes, respectively, 
\begin{align*}
    {\bf{\Lambda}}_{1,0}=\begin{pmatrix}
    1 & 0.5 & 0.8 & 0.3
    \end{pmatrix}^{\top}
\end{align*}
and
\begin{align*}
    {\bf{\Lambda}}_{2,0}=\begin{pmatrix}
    1 & 1.3 & 0.8 & 0.5 & 0 & 0 & 0 & 0\\
    0 & 0 & 0 & 0 & 1 & 0.9 & 0.7 & 1.1
    \end{pmatrix}^{\top}.
\end{align*}
The relationship between $\{\xi_{0,t}\}_{t\geq 0}$ and  $\{\eta_{0,t}\}_{t\geq 0}$ is given by
\begin{align*}
    \eta_{0,t}={\bf{\Gamma}}_0\xi_{0,t}+\zeta_{0,t},
\end{align*}
where $\{\zeta_{0,t}\}_{t\geq 0}$ is a $2$-dimensional latent process and
\begin{align*}
    {\bf{\Gamma}}_0=\begin{pmatrix}
    -0.6 & 0.9
    \end{pmatrix}^{\top}.
\end{align*}
The stochastic process $\{\xi_{0,t}\}_{t\geq 0}$ is defined by the following stochastic differential equation:
\begin{align*}
    d\xi_{0,t}&=-4\bigl(\xi_{0,t-}-3)dt+0.8dW_{1,t}+\int_{\mathbb{R}\backslash \{0\}}zp_1(dt,dz), \quad 
    \xi_{0,0}=3,
\end{align*}
where $\{W_{1,t}\}_{t\geq 0}$ is a $1$-dimensional standard Wiener process, and $p_1(dt,dz)$ is a Poisson random measure on $\mathbb{R}^+\times \mathbb{R}\backslash\{0\}$ with the jump density $f_{1}(z)=4g(z|0,2)$. $g(z|\mu,\sigma^2)$ is the probability density function of the normal distribution with mean $\mu\in\mathbb{R}$ and variance $\sigma^2>0$, i.e.,
\begin{align*}
    g(z|\mu,\sigma^2)=\frac{1}{\sqrt{2\pi\sigma^2}}\exp{\biggl\{-\frac{(z-\mu)^2}{2\sigma^2}\biggr\}}.
\end{align*}
The stochastic process $\{\delta_{0,t}\}_{t\geq 0}$ is assumed to satisfy the following
stochastic differential equation:
\begin{align*}
    d\delta^{(1)}_{0,t}&=-2\delta^{(1)}_{0,t-}dt+0.6 dW^{(1)}_{2,t}
    +\int_{\mathbb{R}\backslash \{0\}}z p_{2,1}(dt,dz), \quad 
    \delta^{(1)}_{0,0}=0,\\
    d\delta^{(2)}_{0,t}&=-5\delta^{(2)}_{0,t-}dt+1.2dW^{(2)}_{2,t}
    +\int_{\mathbb{R}\backslash \{0\}}zp_{2,2}(dt,dz), \quad 
    \delta^{(2)}_{0,0}=0,\\
    d\delta^{(3)}_{0,t}&=-3\delta^{(3)}_{0,t-}dt+0.8dW^{(3)}_{2,t}
    +\int_{\mathbb{R}\backslash \{0\}}zp_{2,3}(dt,dz), \quad 
    \delta^{(3)}_{0,0}=0
\end{align*}
and
\begin{align*}
    d\delta^{(4)}_{0,t}&=-4\delta^{(4)}_{0,t-}dt+0.7dW^{(4)}_{2,t}
    +\int_{\mathbb{R}\backslash \{0\}}zp_{2,4}(dt,dz), \quad 
    \delta^{(4)}_{0,0}=0,
\end{align*}
where $\{W_{2,t}\}_{t\geq 0}$ is a $4$-dimensional standard Wiener process, and $p_{2,i}(dt,dz)$ for $i=1,2,3,4$ are Poisson random measures on $\mathbb{R}^+\times \mathbb{R}\backslash\{0\}$ with the jump densities $f_{2,1}(z)=g(z|0,4)$, $f_{2,2}(z)=g(z|0,3)$, $f_{2,3}(z)=g(z|0,2)$ and $f_{2,4}(z)=g(z|0,3)$, respectively. The stochastic process $\{\varepsilon_{0,t}\}_{t\geq 0}$ is modeled by
the following stochastic differential equations:
\begin{align*}
    d\varepsilon^{(1)}_{0,t}&=-3\varepsilon^{(1)}_{0,t-}dt+1.3dW^{(1)}_{3,t}
    +\int_{\mathbb{R}\backslash \{0\}}z p_{3,1}(dt,dz),\quad 
    \varepsilon^{(1)}_{0,0}=0,\\
    d\varepsilon^{(2)}_{0,t}&=-2\varepsilon^{(2)}_{0,t-}dt+0.5dW^{(2)}_{3,t}
    +\int_{\mathbb{R}\backslash \{0\}}zp_{3,2}(dt,dz),\quad 
    \varepsilon^{(2)}_{0,0}=0,\\
    d\varepsilon^{(3)}_{0,t}&=-4\varepsilon^{(3)}_{0,t-}dt+0.7dW^{(3)}_{3,t}
    +\int_{\mathbb{R}\backslash \{0\}}zp_{3,3}(dt,dz),\quad 
    \varepsilon^{(3)}_{0,0}=0,\\
    d\varepsilon^{(4)}_{0,t}&=-2\varepsilon^{(4)}_{0,t-}dt+0.6
    dW^{(4)}_{3,t}+\int_{\mathbb{R}\backslash \{0\}}zp_{3,4}(dt,dz),\quad 
    \varepsilon^{(4)}_{0,0}=0,\\
    d\varepsilon^{(5)}_{0,t}&=-3\varepsilon^{(5)}_{0,t-}dt
    +0.9dW^{(5)}_{3,t}+\int_{\mathbb{R}\backslash \{0\}}zp_{3,5}(dt,dz),\quad 
    \varepsilon^{(5)}_{0,0}=0,\\
    d\varepsilon^{(6)}_{0,t}&=-5\varepsilon^{(6)}_{0,t-}dt
    +0.8dW^{(6)}_{3,t}+\int_{\mathbb{R}\backslash \{0\}}zp_{3,6}(dt,dz),\quad 
    \varepsilon^{(6)}_{0,0}=0,\\
    d\varepsilon^{(7)}_{0,t}&=-2\varepsilon^{(7)}_{0,t-}dt
    +1.2dW^{(7)}_{3,t}+\int_{\mathbb{R}\backslash \{0\}}zp_{3,7}(dt,dz),\quad 
    \varepsilon^{(7)}_{0,0}=0
\end{align*}
and
\begin{align*}
    d\varepsilon^{(8)}_{0,t}&=-6\varepsilon^{(8)}_{0,t-}dt
    +1.1dW^{(8)}_{3,t}+\int_{\mathbb{R}\backslash \{0\}}zp_{3,8}(dt,dz),\quad 
    \varepsilon^{(8)}_{0,0}=0,
\end{align*}
where $\{W_{3,t}\}_{t\geq 0}$ is an $8$-dimensional standard Wiener process, and 
$p_{3,i}(dt,dz)$ for $i=1,\ldots,8$ are Poisson random measures on $\mathbb{R}^+\times \mathbb{R}\backslash\{0\}$ with the jump densities $f_{3,1}(z)=g(z|0,2)$, $f_{3,2}(z)=g(z|0,3)$, $f_{3,3}(z)=g(z|0,5)$, $f_{3,4}(z)=g(z|0,4)$, $f_{3,5}(z)=g(z|0,2)$, $f_{3,6}(z)=g(z|0,3)$, $f_{3,7}(z)=g(z|0,5)$ and $f_{3,8}(z)=g(z|0,3)$, respectively. The stochastic process $\{\zeta_{0,t}\}_{t\geq 0}$ is defined by
the following stochastic differential equations:
\begin{align*}
    d\zeta^{(1)}_{0,t}&=-2\zeta^{(1)}_{0,t-}dt+1.4dW^{(1)}_{4,t}
    +\int_{\mathbb{R}\backslash \{0\}}z p_{4,1}(dt,dz), \quad 
    \zeta^{(1)}_{0,0}=0
\end{align*}
and
\begin{align*}
    d\zeta^{(2)}_{0,t}&=-4\zeta^{(2)}_{0,t-}dt+0.6dW^{(2)}_{4,t}
    +\int_{\mathbb{R}\backslash \{0\}}zp_{4,2}(dt,dz), \quad 
    \zeta^{(2)}_{0,0}=0,
\end{align*}
where $\{W_{4,t}\}_{t\geq 0}$ is a $2$-dimensional standard Wiener process, and $p_{4,1}(dt,dz)$ and $p_{4,2}(dt,dz)$ are Poisson random measures on $\mathbb{R}^+\times \mathbb{R}\backslash\{0\}$ with the jump densities $f_{4,1}(z)=g(z|0,3)$ and $f_{4,2}(z)=g(z|0,2)$, respectively. 
Figure \ref{modeltrue} shows the path diagram of the true model.
\begin{figure}[h]
    \includegraphics[width=0.9\columnwidth]{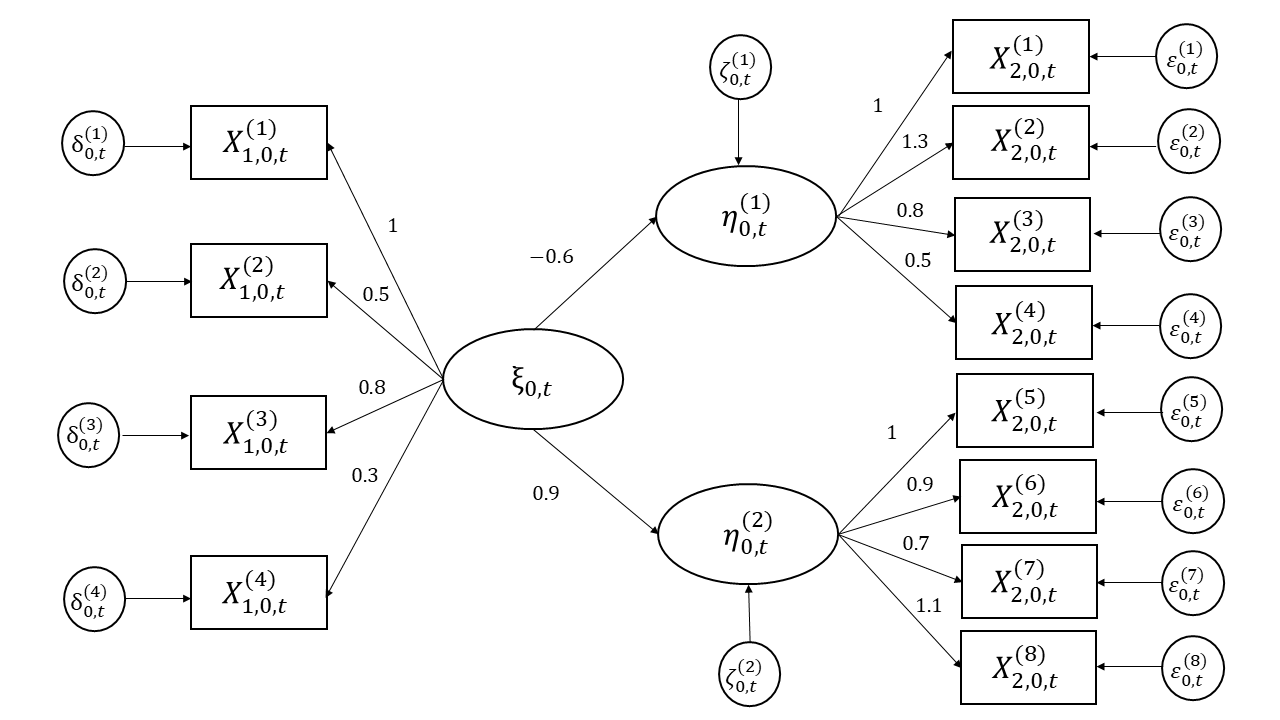}
    \caption{The path diagram of the true model at time $t$.} \label{modeltrue}
\end{figure}
\subsection{Model 1}
Let $p_1=4$, $p_2=8$, $k_1=1$, $k_2=2$ and $q=26$. Suppose that
\begin{align*}
    {\bf{\Lambda}}_{1,1}^{\theta}=\begin{pmatrix}
    1 & \theta_1^{(1)} & \theta_1^{(2)} & \theta_1^{(3)}
    \end{pmatrix}^{\top}
\end{align*}
and
\begin{align*}
    {\bf{\Lambda}}_{2,1}^{\theta}=\begin{pmatrix}
    1 & \theta_1^{(4)} & \theta_1^{(5)} & \theta_1^{(6)} & 0 & 0 & 0 & 0\\
    0 & 0 & 0 & 0 & 1 & \theta_1^{(7)} & \theta_1^{(8)} & \theta_1^{(9)}
    \end{pmatrix}^{\top},\quad {\bf{\Gamma}}_1^{\theta}=\begin{pmatrix}
    \theta_1^{(10)}\\
    \theta_1^{(11)}
    \end{pmatrix},\quad {\bf{\Psi}}_1^{\theta}=\mathbb{I}_2,
\end{align*}
where $\theta_1^{(i)}$ for $i=1,\ldots,11$ are not zero. Moreover, assume that
\begin{align*}
    {\bf{\Sigma}}_{\xi\xi,1}^{\theta}=\theta_1^{(12)},\quad 
    {\bf{\Sigma}}_{\delta\delta,1}^{\theta}=\Diag\bigl(\theta_1^{(13)}, \theta_1^{(14)},\theta_1^{(15)},\theta_1^{(16)}\bigr)^{\top}
\end{align*}
and
\begin{align*}
    {\bf{\Sigma}}_{\varepsilon\varepsilon,1}^{\theta}=\Diag\bigl( \theta_1^{(17)},\theta_1^{(18)},\theta_1^{(19)}, \theta_1^{(20)}, \theta_1^{(21)}, \theta_1^{(22)}, \theta_1^{(23)}, \theta_1^{(24)}\bigr)^{\top},\quad {\bf{\Sigma}}_{\zeta\zeta,1}^{\theta}=\Diag\bigl(\theta_1^{(25)}, \theta_1^{(26)}\bigr)^{\top},
\end{align*}
where $\theta_1^{(i)}$ for $i=12,\ldots,26$ are positive. Set the parameter space of Model 1 as
\begin{align*}
    \Theta_1=\bigl[(-100,0)\cup (0,100)\bigr]^{11}\times (0.01,100)^{15}.
\end{align*}
This model is correctly specified since 
\begin{align*}
    {\bf{\Sigma}}_0={\bf{\Sigma}}_1(\theta_{1,0}),
\end{align*}
where 
\begin{align*}
    \theta_{1,0}&=\bigl(0.5,0.8,0.3,1.3,0.8,0.5,0.9,0.7,1.1,-0.6,0.9,0.64,0.36,1.44,  \\
    &\qquad\qquad 0.64,0.49,1.69,0.25,0.49,0.36,0.81,0.64,1.44,1.21,1.96,0.36
    \bigr)^{\top}\in\Theta_1.
\end{align*}
In a similar manner to the proof in Appendix of Kusano and Uchida \cite{Kusano(BIC)}, we have
\begin{align*}
    {\bf{\Sigma}}_1(\theta_{1})={\bf{\Sigma}}_1(\theta_{1,0})\Longrightarrow\theta_1=\theta_{1,0}.
\end{align*}
Furthermore, $\rank\Delta_{1,0}=26$, which ensures ${\bf{[B1]}}$. Figure \ref{model1} illustrates the path diagram of Model 1 at time $t$. 
\begin{figure}[h]
    \includegraphics[width=0.9\columnwidth]{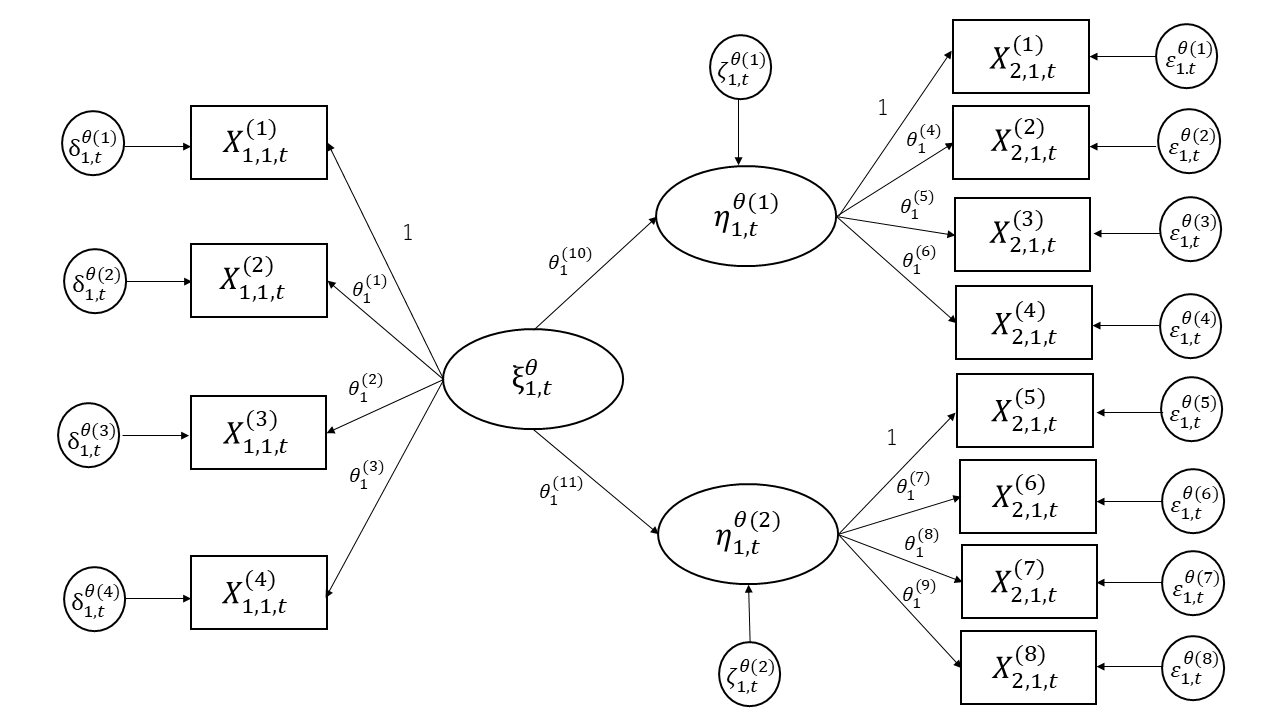}
    \caption{The path diagram of Model 1 at time $t$.} \label{model1}
\end{figure}
\subsection{Model 2}
Let $p_1=4$, $p_2=8$, $k_1=1$, $k_2=2$ and $q=27$. Assume that
\begin{align*}
    {\bf{\Lambda}}_{1,2}^{\theta}=\begin{pmatrix}
    1 & \theta_2^{(1)} & \theta_2^{(2)} & \theta_2^{(3)}
    \end{pmatrix}^{\top}
\end{align*}
and
\begin{align*}
    {\bf{\Lambda}}_{2,2}^{\theta}=\begin{pmatrix}
    1 & \theta_2^{(4)} & \theta_2^{(5)} & \theta_2^{(6)} & 0 & \theta_2^{(7)} & 0 & 0\\
    0 & 0 & 0 & 0 & 1 & \theta_2^{(8)} & \theta_2^{(9)} & \theta_2^{(10)}
    \end{pmatrix}^{\top},\quad {\bf{\Gamma}}_2^{\theta}=\begin{pmatrix}
    \theta_2^{(11)}\\
    \theta_2^{(12)}
    \end{pmatrix},\quad {\bf{\Psi}}_2^{\theta}=\mathbb{I}_2,
\end{align*}
where $\theta_2^{(i)}$ for $i=1,\ldots,6$ and $i=8,\ldots,12$ are not zero. Furthermore, suppose that
\begin{align*}
    {\bf{\Sigma}}_{\xi\xi,2}^{\theta}=\theta_2^{(13)},\quad 
    {\bf{\Sigma}}_{\delta\delta,2}^{\theta}=\Diag\bigl(\theta_2^{(14)}, \theta_2^{(15)},\theta_2^{(16)},\theta_2^{(17)}\bigr)^{\top}
\end{align*}
and
\begin{align*}
    {\bf{\Sigma}}_{\varepsilon\varepsilon,2}^{\theta}=\Diag\bigl( \theta_2^{(18)},\theta_2^{(19)}, \theta_2^{(20)}, \theta_2^{(21)}, \theta_2^{(22)}, \theta_2^{(23)}, \theta_2^{(24)},\theta_2^{(25)}\bigr)^{\top},\quad {\bf{\Sigma}}_{\zeta\zeta,2}^{\theta}=\Diag\bigl(\theta_2^{(26)}, \theta_2^{(27)}\bigr)^{\top},
\end{align*}
where $\theta_2^{(i)}$ for $i=13,\ldots,27$ are positive. The parameter space of Model 2 is defined by
\begin{align*}
    \Theta_2=\bigl[(-100,0)\cup (0,100)\bigr]^{6}\times (-100,100)\times
    \bigl[(-100,0)\cup (0,100)\bigr]^{5}
    \times (0.01,100)^{15}.
\end{align*}
This model is correctly specified since 
\begin{align*}
    {\bf{\Sigma}}_0={\bf{\Sigma}}_2(\theta_{2,0}),
\end{align*}
where 
\begin{align*}
    \theta_{2,0}&=\bigl(0.5,0.8,0.3,1.3,0.8,0.5,0,0.9,0.7,1.1,-0.6,0.9,0.64,0.36,1.44,  \\
    &\qquad\qquad 0.64,0.49,1.69,0.25,0.49,0.36,0.81,0.64,1.44,1.21,1.96,0.36
    \bigr)^{\top}\in\Theta_2.
\end{align*}
Since $\rank\Delta_{2,0}=27$ and 
\begin{align}
    {\bf{\Sigma}}_2(\theta_{2})={\bf{\Sigma}}_2(\theta_{2,0})\Longrightarrow\theta_2=\theta_{2,0}, \label{model2iden}
\end{align}
Model 2 satisfies ${\bf{[B1]}}$. See Appendix \ref{iden} for the proof of (\ref{model2iden}). Figure \ref{model2} shows the path diagram of Model 2 at time $t$. 
\begin{figure}[h]
    \includegraphics[width=0.9\columnwidth]{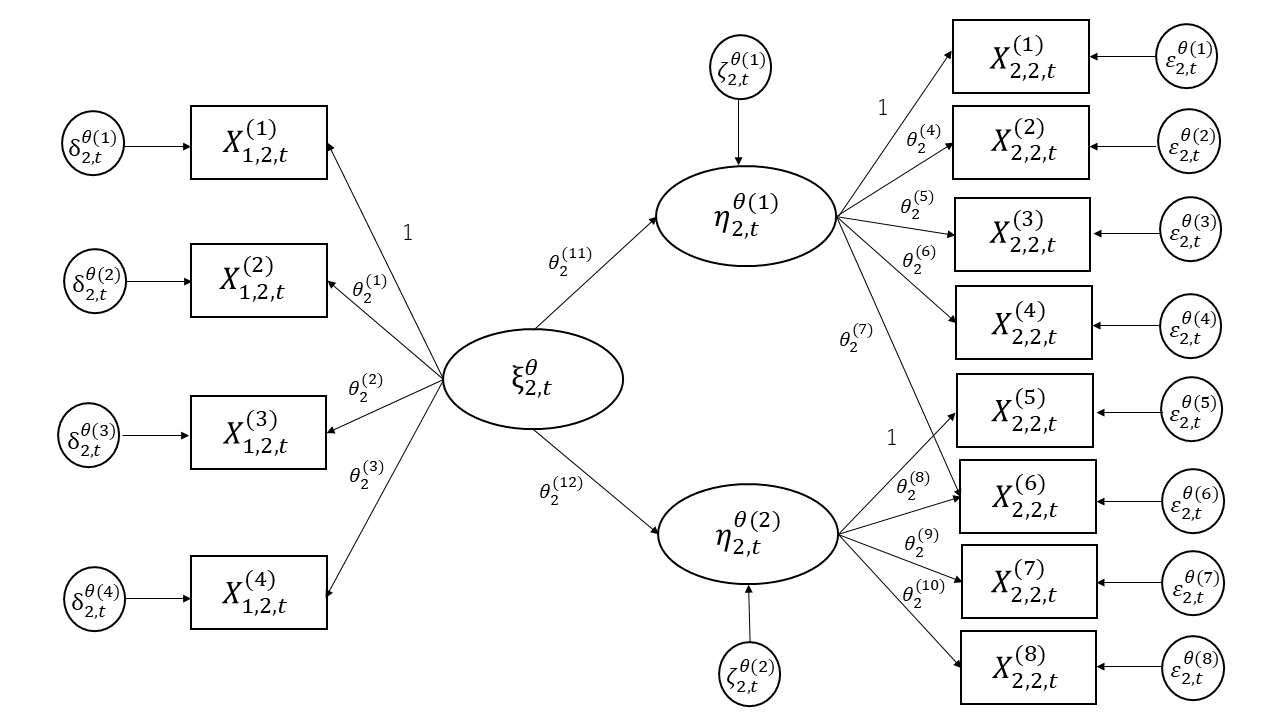}
    \caption{The path diagram of Model 2 at time $t$.} \label{model2}
\end{figure}
\subsection{Model 3}
Let $p_1=4$, $p_2=8$, $k_1=1$, $k_2=1$ and $q=25$. Suppose that
\begin{align*}
    {\bf{\Lambda}}_{1,3}^{\theta}=\begin{pmatrix}
    1 & \theta_3^{(1)} & \theta_3^{(2)} & \theta_3^{(3)}
    \end{pmatrix}^{\top}
\end{align*}
and
\begin{align*}
    {\bf{\Lambda}}_{2,3}^{\theta}=\begin{pmatrix}
    1 & \theta_3^{(4)} & \theta_3^{(5)} & \theta_3^{(6)} &\theta_3^{(7)} & \theta_3^{(8)} & \theta_3^{(9)} & \theta_3^{(10)}
    \end{pmatrix}^{\top},\quad {\bf{\Gamma}}_3^{\theta}=
    \theta_3^{(11)},
\end{align*}
where $\theta_3^{(i)}$ for $i=1,\ldots,11$ are not zero. Assume that
\begin{align*}
    {\bf{\Sigma}}_{\xi\xi,3}^{\theta}=\theta_3^{(12)},\quad 
    {\bf{\Sigma}}_{\delta\delta,3}^{\theta}=\Diag\bigl(\theta_3^{(13)}, \theta_3^{(14)},\theta_3^{(15)},\theta_3^{(16)} \bigr)^{\top}
\end{align*}
and
\begin{align*}
    {\bf{\Sigma}}_{\varepsilon\varepsilon,3}^{\theta}=\Diag\bigl( \theta_3^{(17)},\theta_3^{(18)},\theta_3^{(19)}, \theta_3^{(20)}, \theta_3^{(21)}, \theta_3^{(22)},\theta_3^{(23)}, \theta_3^{(24)}\bigr)^{\top},\quad {\bf{\Sigma}}_{\zeta\zeta,3}^{\theta}=\theta_3^{(25)},
\end{align*}
where $\theta_3^{(i)}$ for $i=12,\ldots,25$ are positive. The parameter space of Model 3 is set to
\begin{align*}
    \Theta_3=\bigl[(-100,0)\cup (0,100)\bigr]^{11}\times (0.01,100)^{14}.
\end{align*}
For all $\theta_3\in\Theta_3$, it holds that
\begin{align*}
    {\bf{\Sigma}}_0\neq{\bf{\Sigma}}_3(\theta_{3,0}),
\end{align*}
so that Model $3$ is a misspecified model. Figure \ref{model3} gives the path diagram of Model 3 at time $t$.
\begin{figure}[h]
    \includegraphics[width=0.9\columnwidth]{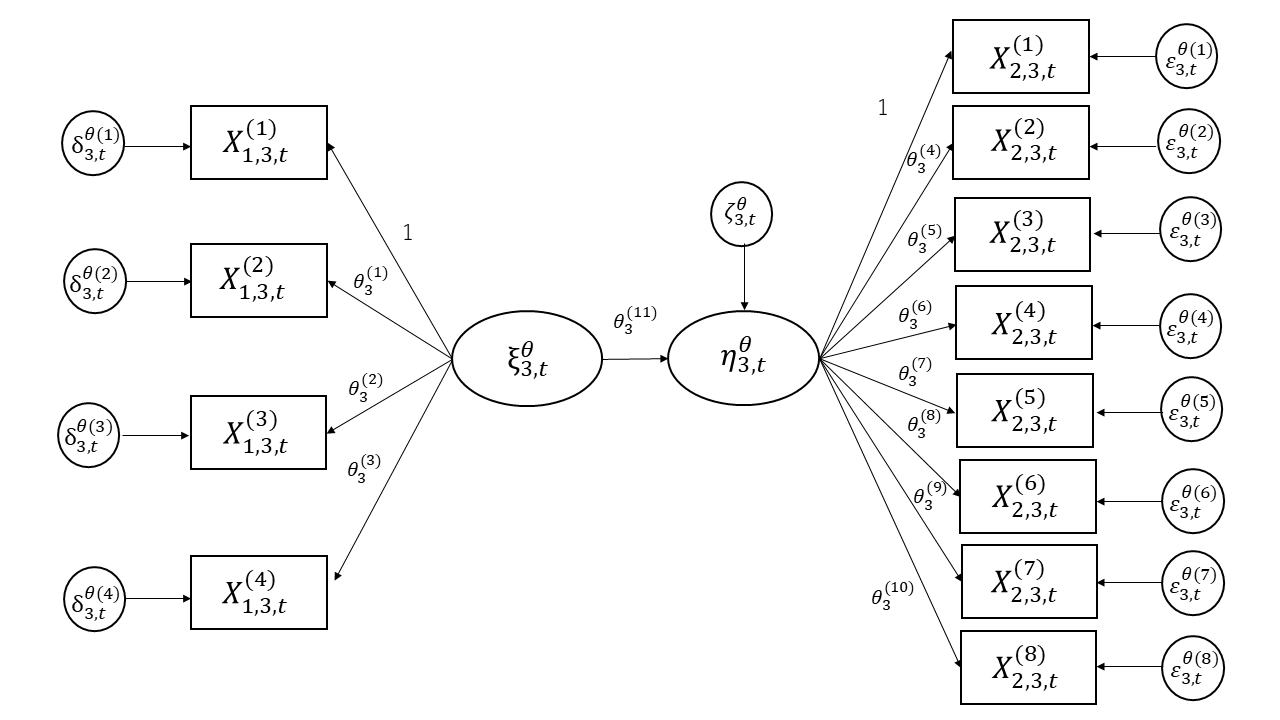}
    \caption{The path diagram of Model 3 at time $t$.} \label{model3}
\end{figure}
\subsection{Simulation results}
The simulation study was conducted by using the optim() function in R with the BFGS method, where the initial parameter was set to the true parameter. The simulation was conducted for 10,000 iterations with $T=1$, $D=10$ and $\rho=0.4$, and the sample sizes were chosen as $10^3$, $10^4$ and $10^5$. Tables 1 and 2 summarize the sample means and standard deviations of the estimators $\hat{\theta}_{1,n}$ and $\hat{\theta}_{2,n}$. The results imply that both estimators are consistent. Table 3 reports the number of times each model was selected by QAIC.
Since Model 3 was never chosen, Theorem 3 holds in this setting. Moreover, QAIC often selects the over-fitted Model 2 even for large sample sizes. This implies that QAIC lacks model-selection consistency.
\begin{table}[h]
\ \\ \ \\ \ \\ \ \\ \ \\ \ \\ \ \\ \ \\ \ \\ \ \\
\centering
\begin{tabular}{ccccccc}
    & $\hat{\theta}_{1,n}^{(1)}$ & $\hat{\theta}_{1,n}^{(2)}$ & $\hat{\theta}_{1,n}^{(3)}$ & $\hat{\theta}_{1,n}^{(4)}$ & $\hat{\theta}_{1,n}^{(5)}$ & $\hat{\theta}_{1,n}^{(6)}$\\\hline
    $n=10^3$ & 0.501 (0.060) & 0.801 (0.052) & 0.300 (0.037) & 1.301 (0.042) & 0.801 (0.029) & 0.501 (0.020)\\
    $n=10^4$  & 0.500 (0.018) & 0.800 (0.016) & 0.300 (0.011)& 1.300 (0.013) & 0.800 (0.009) & 0.500 (0.006)\\
    $n=10^5$ & 0.500 (0.006) & 0.800 (0.005) & 0.300 (0.003)& 1.300 (0.004) & 0.800 (0.003) & 0.500 (0.002)\\
    & $\hat{\theta}_{1,n}^{(7)}$ & $\hat{\theta}_{1,n}^{(8)}$ & $\hat{\theta}_{1,n}^{(9)}$ & $\hat{\theta}_{1,n}^{(10)}$ & $\hat{\theta}_{1,n}^{(11)}$ & $\hat{\theta}_{1,n}^{(12)}$ \\\hline
    $n=10^3$ & 0.901 (0.051) & 0.701 (0.055) & 1.102 (0.063)& -0.601 (0.072) & 0.900 (0.063) & 0.679 (0.069) \\
    $n=10^4$ & 0.900 (0.015) & 0.700 (0.017) & 1.100 (0.019) & -0.600 (0.023) & 0.900 (0.019) & 0.644 (0.017)\\
    $n=10^5$ & 0.900 (0.005) & 0.700 (0.005) & 1.100 (0.006) & -0.600 (0.007) & 0.900 (0.006) & 0.640 (0.005)\\
    & $\hat{\theta}_{1,n}^{(13)}$ & $\hat{\theta}_{1,n}^{(14)}$ & $\hat{\theta}_{1,n}^{(15)}$ & $\hat{\theta}_{1,n}^{(16)}$ & $\hat{\theta}_{1,n}^{(17)}$ & $\hat{\theta}_{1,n}^{(18)}$ \\\hline
    $n=10^3$ & 0.385 (0.073) & 1.482 (0.102) & 0.674 (0.092) & 0.527 (0.079) & 1.724 (0.118) & 0.279 (0.093)\\
    $n=10^4$ & 0.362 (0.013) & 1.443 (0.022) & 0.642 (0.014) & 0.493 (0.011) & 1.692 (0.027) & 0.252 (0.016)\\
    $n=10^5$ & 0.360 (0.003) & 1.440 (0.007) & 0.640 (0.004) & 0.490 (0.002) & 1.690 (0.008) & 0.250 (0.005)\\
    & $\hat{\theta}_{1,n}^{(19)}$ & $\hat{\theta}_{1,n}^{(20)}$ & $\hat{\theta}_{1,n}^{(21)}$ & $\hat{\theta}_{1,n}^{(22)}$ & 
    $\hat{\theta}_{1,n}^{(23)}$ & $\hat{\theta}_{1,n}^{(24)}$
    \\\hline
    $n=10^3$ & 0.546 (0.102) & 0.388 (0.067) & 0.844 (0.098) & 0.680 (0.089) & 1.471 (0.096) & 1.257 (0.110)\\
    $n=10^4$ & 0.495 (0.013) & 0.362 (0.009) & 0.812 (0.018) & 0.644 (0.015) & 1.442 (0.023) & 1.214 (0.024)\\
    $n=10^5$ & 0.490 (0.003) & 0.360 (0.002) & 0.810 (0.005) & 0.640 (0.004) & 1.440 (0.007) & 1.210 (0.007)\\
    & $\hat{\theta}_{1,n}^{(25)}$ & $\hat{\theta}_{1,n}^{(26)}$
    \\\hline
    $n=10^3$ & 1.969 (0.149) & 0.372 (0.053)\\
    $n=10^4$ & 1.961 (0.046) & 0.361 (0.014)\\
    $n=10^5$ & 1.960 (0.015) & 0.360 (0.004)\\
\end{tabular}
\caption{Sample mean (sample standard deviation) of $\hat{\theta}_{1,n}$.}
\label{thetatable1}
\end{table}
\begin{table}[h]
\ \\ \ \\ \ \\ \ \\ \ \\ \ \\
\centering
\begin{tabular}{ccccccc}
    & $\hat{\theta}_{2,n}^{(1)}$ & $\hat{\theta}_{2,n}^{(2)}$ & $\hat{\theta}_{2,n}^{(3)}$ & $\hat{\theta}_{2,n}^{(4)}$ & $\hat{\theta}_{2,n}^{(5)}$ & $\hat{\theta}_{2,n}^{(6)}$\\\hline
    $n=10^3$ & 0.501 (0.060) & 0.801 (0.052) & 0.300 (0.037) & 1.301 (0.042) & 0.801 (0.029) & 0.501 (0.021)\\
    $n=10^4$  & 0.500 (0.018) & 0.800 (0.016) & 0.300 (0.011) & 1.300 (0.013) & 0.800 (0.009) & 0.500 (0.006)\\
    $n=10^5$ & 0.500 (0.006) & 0.800 (0.005) & 0.300 (0.003) & 1.300 (0.004) & 0.800 (0.003) & 0.500 (0.002)\\
    & $\hat{\theta}_{2,n}^{(7)}$ & $\hat{\theta}_{2,n}^{(8)}$ & $\hat{\theta}_{2,n}^{(9)}$ & $\hat{\theta}_{2,n}^{(10)}$ & $\hat{\theta}_{2,n}^{(11)}$ & $\hat{\theta}_{2,n}^{(12)}$ \\\hline
    $n=10^3$ & 0.000 (0.028) & 0.901 (0.051) & 0.701 (0.055)& 1.102 (0.063) & -0.601 (0.072) & 0.900 (0.063)\\
    $n=10^4$ & 0.000 (0.008) & 0.900 (0.015) & 0.700 (0.017) & 1.100 (0.019) & -0.600 (0.023) & 0.900 (0.019)\\
    $n=10^5$ & 0.000 (0.002) & 0.900 (0.005) & 0.700 (0.005) & 1.100 (0.006) & -0.600 (0.007) & 0.900 (0.006)\\
    & $\hat{\theta}_{2,n}^{(13)}$ & $\hat{\theta}_{2,n}^{(14)}$ & $\hat{\theta}_{2,n}^{(15)}$ & $\hat{\theta}_{2,n}^{(16)}$ & $\hat{\theta}_{2,n}^{(17)}$ & $\hat{\theta}_{2,n}^{(18)}$ \\\hline
    $n=10^3$ & 0.679 (0.069) & 0.385 (0.073) & 1.482 (0.102)& 0.674 (0.092) & 0.527 (0.079) & 1.724 (0.118)\\
    $n=10^4$ & 0.644 (0.017) & 0.362 (0.013) & 1.443 (0.022) & 0.642 (0.014) & 0.493 (0.011) & 1.692 (0.027)\\
    $n=10^5$ & 0.640 (0.005) & 0.360 (0.003) & 1.440 (0.007) & 0.640 (0.004) & 0.490 (0.002) & 1.690 (0.008)\\
    & $\hat{\theta}_{2,n}^{(19)}$ & $\hat{\theta}_{2,n}^{(20)}$ & $\hat{\theta}_{2,n}^{(21)}$ & $\hat{\theta}_{2,n}^{(22)}$ & 
    $\hat{\theta}_{2,n}^{(23)}$ & $\hat{\theta}_{2,n}^{(24)}$
    \\\hline
    $n=10^3$ & 0.279 (0.093) & 0.546 (0.102) & 0.387 (0.067) & 0.844 (0.098) & 0.680 (0.089) & 1.471 (0.096) \\
    $n=10^4$ & 0.252 (0.016) & 0.495 (0.013) & 0.362 (0.009) & 0.812 (0.018) & 0.644 (0.015) & 1.442 (0.023)\\
    $n=10^5$ & 0.250 (0.005) & 0.490 (0.003) & 0.360 (0.002) & 0.810 (0.005) & 0.640 (0.004) & 1.440 (0.007)\\
    & $\hat{\theta}_{2,n}^{(25)}$ & $\hat{\theta}_{2,n}^{(26)}$ & $\hat{\theta}_{2,n}^{(27)}$
    \\\hline
    $n=10^3$ & 1.257 (0.110) & 1.969 (0.149) & 0.372 (0.053) \\
    $n=10^4$ & 1.214 (0.024) & 1.961 (0.046) & 0.361 (0.014)\\
    $n=10^5$ & 1.210 (0.007) & 1.960 (0.015) & 0.360 (0.004)\\
\end{tabular}
\caption{Sample mean (sample standard deviation) of $\hat{\theta}_{2,n}$.}
\label{thetatable2}
\end{table}
\begin{table}[h]
\ \\ \ \\ \ \\ \ \\
\centering
\begin{tabular}{ccccc}
    &  $n=10^3$ & $n=10^4$ & $n=10^5$ \\\hline
    Model 1 & 8065 & 8331 & 8396 \\
    Model 2 & 1935 & 1669 & 1604 \\
    Model 3 & 0 & 0 & 0
\end{tabular}
    \caption{The number of times each model was selected by QAIC.}\label{table}
\end{table}
\newpage
\clearpage
\section{Proofs}\label{proofs}
For a positive sequence $\{a_{n}\}_{n\in\mathbb{N}}$, let $R:(0,\infty)\times \mathbb{R}^d\rightarrow \mathbb{R}$ denote a class of functions satisfying 
\begin{align*}
    |R({a_{n}},x)|\leq a_{n}C(1+|x|)^C
\end{align*}
for some constant $C>0$ independent of $x$ and $\theta$. For a process $S=\{S_t\}_{t\geq 0}$, define
\begin{align*}
    R_i(a_n,S)=R(a_n,S_{t_i^n}).
\end{align*}
and set $\mathcal{F}_{i}^n=\mathcal{F}_{t_{i}^n}$ for $i=0,\ldots,n$. 
For a real-valued random variable $F$, $F\in L^1({\bf{P}})$ and $F\in L^2({\bf{P}})$ mean that $|F|$ and $F^2$ are integrable under the probability measure ${\bf{P}}$, respectively; that is,
\begin{align*}
    {\bf{E}}[|F|]<\infty,\quad {\bf{E}}[F^2]<\infty.
\end{align*}
In addition, define
\begin{align*}
    R_{i}(a_n,\xi,\delta,\varepsilon,\zeta)=R_{i}(a_n,\xi_0)+
    R_{i}(a_n,\delta_0)+R_{i}(a_n,\varepsilon_0)+R_{i}(a_n,\zeta_0)
\end{align*}
and
\begin{align*}
    \bar{R}_{i}(a_n,\xi,\delta,\varepsilon,\zeta)=1-R_{i}(a_n,\xi,\delta,\varepsilon,\zeta)
\end{align*}
for $i=0,\ldots,n$. Throughout this section, the model index “$m$” may be omitted for simplicity.
\begin{lemma}\label{EX2abslemma}
Suppose that $[{\bf{A1}}]$-$[{\bf{A4}}]$ hold. Then, 
\begin{align*}
    {\bf{E}}\Bigl[\bigl|(\Delta_{i}^n X)^{(j_1)}\bigr|\bigl|(\Delta_{i}^n X)^{(j_2)}\bigr|
    {\bf{1}}_{\{|\Delta_{i}^n X|\leq Dh_n^{\rho}\}}\big|\mathcal{F}^n_{i-1}\Bigr]=R_{i-1}(h_n,\xi,\delta,\varepsilon,\zeta)
\end{align*}
for a sufficiently large $n$ and $j_1,j_2=1,\ldots,p$.
\end{lemma}
\begin{proof}
The result can be shown in an analogous manner to Propositions 2 and 6 in Kusano and Uchida \cite{Kusano(jump)}. See also Appendix \ref{proofEX2abslemma}.
\end{proof}
\begin{lemma}\label{H3lemma}
Suppose that $[{\bf{A1}}]$-$[{\bf{A4}}]$ hold. Then, for all $L>0$, 
\begin{align*}
    \sup_{n\in\mathbb{N}}{\bf{E}}\Biggl[\biggl(\frac{1}{n}\sup_{\theta\in\Theta}\bigl|\partial^3_{\theta}{\bf{H}}_n(\theta)\bigr|\biggr)^L\Biggr]<\infty.
\end{align*}
\end{lemma}
\begin{proof}
It is sufficient to prove that for all $L>0$,
\begin{align}
    \sup_{n\in\mathbb{N}}{\bf{E}}\Biggl[\biggl(\sup_{\theta\in\Theta}\Bigl|\frac{1}{n}\partial_{\theta^{(j_1)}}\partial_{\theta^{(j_2)}}\partial_{\theta^{(j_3)}}{\bf{H}}_n(\theta)\Bigr|\biggr)^L\Biggr]<\infty \label{suffH31}
\end{align}
for $j_1,j_2,j_3=1,\ldots,q$. Since 
\begin{align*}
    \frac{1}{n}\partial_{\theta^{(j_1)}}\partial_{\theta^{(j_2)}}\partial_{\theta^{(j_3)}}{\bf{H}}_n(\theta)
    &=-\frac{1}{2nh_n}\sum_{i=1}^n\sum_{k_1=1}^p\sum_{k_2=1}^p
    \bigl(\partial_{\theta^{(j_1)}}\partial_{\theta^{(j_2)}}
    \partial_{\theta^{(j_3)}}{\bf{\Sigma}}(\theta)^{-1}\bigr)_{k_1k_2}\\
    &\qquad\qquad\qquad\qquad\quad\times(\Delta_{i}^n X)^{(k_1)}(\Delta_{i}^n X)^{(k_2)}
    {\bf{1}}_{\{|\Delta_{i}^n X|\leq Dh_n^{\rho}\}}\\
    &\qquad-\partial_{\theta^{(j_1)}}\partial_{\theta^{(j_2)}}\partial_{\theta^{(j_3)}}\log\det {\bf{\Sigma}}(\theta)\times\frac{1}{2n}\sum_{i=1}^n{\bf{1}}_{\{|\Delta_{i}^n X|\leq Dh_n^{\rho}\}},
\end{align*}
we have
\begin{align*}
    &\quad\ \sup_{\theta\in\Theta}\Bigl|\frac{1}{n}\partial_{\theta^{(j_1)}}\partial_{\theta^{(j_2)}}\partial_{\theta^{(j_3)}}{\bf{H}}_n(\theta)\Bigr|\\
    &\leq\frac{1}{2nh_n}\sum_{i=1}^n\sum_{k_1=1}^p\sum_{k_2=1}^p\sup_{\theta\in\Theta}\bigl|\bigl(\partial_{\theta^{(j_1)}}\partial_{\theta^{(j_2)}}
    \partial_{\theta^{(j_3)}}{\bf{\Sigma}}(\theta)^{-1}\bigr)_{k_1k_2}\bigr|\\
    &\qquad\qquad\qquad\qquad\qquad\quad\times\bigl|(\Delta_{i}^n X)^{(k_1)}\bigr|\bigl|(\Delta_{i}^n X)^{(k_2)}\bigr|
    {\bf{1}}_{\{|\Delta_{i}^n X|\leq Dh_n^{\rho}\}}\\
    &\qquad+\sup_{\theta\in\Theta}\bigl|\partial_{\theta^{(j_1)}}
    \partial_{\theta^{(j_2)}}\partial_{\theta^{(j_3)}}\log\det {\bf{\Sigma}}(\theta)\bigr|\times\frac{1}{2n}\sum_{i=1}^n{\bf{1}}_{\{|\Delta_{i}^n X|\leq Dh_n^{\rho}\}} \\
    &\leq \frac{C}{nh_n}\sum_{i=1}^n\sum_{k_1=1}^p\sum_{k_2=1}^p\bigl|(\Delta_{i}^n X)^{(k_1)}\bigr|\bigl|(\Delta_{i}^n X)^{(k_2)}\bigr|
    {\bf{1}}_{\{|\Delta_{i}^n X|\leq Dh_n^{\rho}\}}+C.
\end{align*}
Consequently, it holds that
\begin{align*}
    &\quad\ {\bf{E}}\Biggl[\biggl(\sup_{\theta\in\Theta}\Bigl|\frac{1}{n}\partial_{\theta^{(j_1)}}\partial_{\theta^{(j_2)}}\partial_{\theta^{(j_3)}}{\bf{H}}_n(\theta)\Bigr|\biggr)^L\Biggr]\\
    &\leq C_L{\bf{E}}\Biggl[\biggl(\frac{1}{nh_n}\sum_{i=1}^n\sum_{k_1=1}^p\sum_{k_2=1}^p\bigl|(\Delta_{i}^n X)^{(k_1)}\bigr|\bigl|(\Delta_{i}^n X)^{(k_2)}\bigr|
    {\bf{1}}_{\{|\Delta_{i}^n X|\leq Dh_n^{\rho}\}}\biggr)^L\Biggr]+C_L\\
    &\leq C_L\sum_{k_1=1}^p\sum_{k_2=1}^p{\bf{E}}\Biggl[\biggl(\frac{1}{nh_n}\sum_{i=1}^n\bigl|(\Delta_{i}^n X)^{(k_1)}\bigr|\bigl|(\Delta_{i}^n X)^{(k_2)}\bigr|
    {\bf{1}}_{\{|\Delta_{i}^n X|\leq Dh_n^{\rho}\}}\biggr)^L\Biggr]+C_L\\
    &\leq C_L\sum_{k_1=1}^p\sum_{k_2=1}^p{\bf{E}}\biggl[\bigl|{\bf{M}}_{n,k_1,k_2}\bigr|^L\biggr]+C_L\sum_{k_1=1}^p\sum_{k_2=1}^p{\bf{E}}\biggl[\bigl|{\bf{R}}_{n,k_1,k_2}\bigr|^L\biggr]
    +C_L
\end{align*}
for all $L\geq 1$, where 
\begin{align*}
    {\bf{M}}_{n,k_1,k_2}&=\frac{1}{nh_n}\sum_{i=1}^n\Bigl\{\bigl|(\Delta_{i}^n X)^{(k_1)}\bigr|\bigl|(\Delta_{i}^n X)^{(k_2)}\bigr|
    {\bf{1}}_{\{|\Delta_{i}^n X|\leq Dh_n^{\rho}\}}\\
    &\qquad\qquad\quad -{\bf{E}}\Bigl[\bigl|(\Delta_{i}^n X)^{(k_1)}\bigr|\bigl|(\Delta_{i}^n X)^{(k_2)}\bigr|
    {\bf{1}}_{\{|\Delta_{i}^n X|\leq Dh_n^{\rho}\}}\big|\mathcal{F}^n_{i-1}\Bigr]\Bigr\}
\end{align*}
and
\begin{align*}
    {\bf{R}}_{n,k_1,k_2}&=\frac{1}{nh_n}\sum_{i=1}^n{\bf{E}}\Bigl[\bigl|(\Delta_{i}^n X)^{(k_1)}\bigr|\bigl|(\Delta_{i}^n X)^{(k_2)}\bigr|
    {\bf{1}}_{\{|\Delta_{i}^n X|\leq Dh_n^{\rho}\}}\big|\mathcal{F}^n_{i-1}\Bigr]
\end{align*}
for $k_1,k_2=1,\ldots,p$. To prove (\ref{suffH31}), it is enough to show 
\begin{align}
    \sup_{n\in\mathbb{N}}{\bf{E}}\biggl[\bigl|{\bf{M}}_{n,k_1,k_2}\bigr|^L\biggr]<\infty
    \label{M}
\end{align}
and 
\begin{align}
    \sup_{n\in\mathbb{N}}{\bf{E}}\biggl[\bigl|{\bf{R}}_{n,k_1,k_2}\bigr|^L\biggr]<\infty
    \label{R}
\end{align}
for all $L\geq 1$. First, we consider (\ref{M}). Set
\begin{align*}
    {\bf{N}}_{0,n,k_1,k_2}=0,\quad {\bf{N}}_{\ell,n,k_1,k_2}=\frac{1}{nh_n}\sum_{i=1}^{\ell} {\bf{L}}_{i,n,k_1,k_2}
\end{align*}
for $\ell=1,\ldots,n$, where
\begin{align*}
    {\bf{L}}_{i,n,k_1,k_2}&=\bigl|(\Delta_{i}^n X)^{(k_1)}\bigr|\bigl|(\Delta_{i}^n X)^{(k_2)}\bigr|
    {\bf{1}}_{\{|\Delta_{i}^n X|\leq Dh_n^{\rho}\}}\\
    &\qquad\qquad\qquad-{\bf{E}}\Bigl[\bigl|(\Delta_{i}^n X)^{(k_1)}\bigr|\bigl|(\Delta_{i}^n X)^{(k_2)}\bigr|
    {\bf{1}}_{\{|\Delta_{i}^n X|\leq Dh_n^{\rho}\}}\big|\mathcal{F}^n_{i-1}\Bigr].
\end{align*}
Note that the stochastic process
\begin{align*}
    \bigl\{{\bf{N}}_{\ell,n,k_1,k_2}\bigr\}_{\ell=0}^n
\end{align*}
is a discrete-time martingale with respect to the filtration $(\mathcal{F}_{\ell}^n)_{\ell=0}^n$, and that ${\bf{M}}_{n,k_1,k_2}$ is its terminal value:
\begin{align*}
    {\bf{M}}_{n,k_1,k_2}={\bf{N}}_{n,n,k_1,k_2}.
\end{align*}
Using the Burkholder inequality, we see
\begin{align*}
    {\bf{E}}\biggl[\bigl|{\bf{M}}_{n,k_1,k_2}\bigr|^L\biggr]
    &\leq C_L{\bf{E}}\biggl[\bigl\langle {\bf{N}}_{k_1,k_2} \bigr\rangle_n^{L/2}\biggr]\\
    &\leq\frac{C_L}{n^{L/2}h_n^L}{\bf{E}}\Biggl[\biggl|\frac{1}{n}\sum_{i=1}^n {\bf{L}}^2_{i,n,k_1,k_2}\biggr|^{L/2}\Biggr]\\
    &\leq \frac{C_L}{n^{L/2}h_n^L}\times\frac{1}{n}
    \sum_{i=1}^n {\bf{E}}\biggl[\bigl|{\bf{L}}_{i,n,k_1,k_2}\bigr|^L\biggr]
\end{align*}
for all $L\geq 2$, where 
\begin{align*}
    \bigl\langle {\bf{N}}_{k_1,k_2} \bigr\rangle_n&=\sum_{i=1}^n \bigl(
    {\bf{N}}_{i,n,k_1,k_2}-{\bf{N}}_{i-1,n,k_1,k_2}\bigr)^2=\frac{1}{n^2h_n^2}\sum_{i=1}^n {\bf{L}}_{i,n,k_1,k_2}^2.
\end{align*}
Furthermore, one gets
\begin{align*}
    {\bf{E}}\biggl[\bigl|{\bf{L}}_{i,n,k_1,k_2}\bigr|^L\biggr]
    &\leq C_L{\bf{E}}\biggl[\bigl|(\Delta_{i}^n X)^{(k_1)}\bigr|^L\bigl|(\Delta_{i}^n X)^{(k_2)}\bigr|^L
    {\bf{1}}_{\{|\Delta_{i}^n X|\leq Dh_n^{\rho}\}}\biggr]\\
    &\quad+C_L{\bf{E}}\biggl[\Bigl|{\bf{E}}\Bigl[\bigl|(\Delta_{i}^n X)^{(k_1)}\bigr|\bigl|(\Delta_{i}^n X)^{(k_2)}\bigr|
    {\bf{1}}_{\{|\Delta _{i}^nX|\leq Dh_n^{\rho}\}}\big|\mathcal{F}^n_{i-1}\Bigr]\Bigr|^L\biggr]\\
    &\leq C_L{\bf{E}}\biggl[\bigl|(\Delta_{i}^n X)^{(k_1)}\bigr|^L\bigl|(\Delta_{i}^n X)^{(k_2)}\bigr|^L
    {\bf{1}}_{\{|\Delta_{i}^n X|\leq Dh_n^{\rho}\}}\biggr]\\
    &\leq C_L{\bf{E}}\biggl[\bigl|\Delta_{i}^n X\bigr|^{2L}{\bf{1}}_{\{|\Delta_{i}^n X|\leq Dh_n^{\rho}\}}\biggr]\leq C_Lh_n^{2L\rho}
\end{align*}
for all $L\geq 1$. Hence, 
\begin{align*}
    {\bf{E}}\biggl[\bigl|{\bf{M}}_{n,k_1,k_2}\bigr|^L\biggr]
    \leq\frac{C_L}{n^{L/2}h_n^L}\times h_n^{2L\rho}
\end{align*}
for all $L\geq 2$. Since $\rho>1/4$, we have
\begin{align*}
    \frac{1}{n^{L/2}h_n^L}\times
    h_n^{2L\rho}=n^{(1/2-2\rho)L}\longrightarrow 0
\end{align*}
as $n\longrightarrow\infty$, which yields
\begin{align*}
    \sup_{n\in\mathbb{N}}{\bf{E}}\biggl[\bigl|{\bf{M}}_{n,k_1,k_2}\bigr|^L\biggr]<\infty
\end{align*}
for all $L\geq 2$. Therefore, (\ref{M}) holds. Next, we show (\ref{R}). As it follows from Lemma \ref{EX2abslemma} that 
\begin{align*}
    {\bf{R}}_{n,k_1,k_2}&=\frac{1}{nh_n}\sum_{i=1}^n{\bf{E}}\Bigl[\bigl|(\Delta_{i}^n X)^{(k_1)}\bigr|\bigl|(\Delta_{i}^n X)^{(k_2)}\bigr|
    {\bf{1}}_{\{|\Delta_{i}^n X|\leq Dh_n^{\rho}\}}\big|\mathcal{F}^n_{i-1}\Bigr]\\
    &=\frac{1}{n}\sum_{i=1}^n R_{i-1}(1,\xi,\delta,\varepsilon,\zeta)
\end{align*}
for a sufficiently large $n$, one has
\begin{align*}
    {\bf{E}}\biggl[\bigl|{\bf{R}}_{n,k_1,k_2}\bigr|^L\biggr]\leq \frac{1}{n}\sum_{i=1}^n 
    {\bf{E}}\biggl[\bigl|R_{i-1}(1,\xi,\delta,\varepsilon,\zeta)\bigr|^L\biggr]\leq C_L
\end{align*}
for a sufficiently large $n$ and for all $L\geq 1$. This implies
\begin{align*}
    \sup_{n\in\mathbb{N}}{\bf{E}}\biggl[\bigl|{\bf{R}}_{n,k_1,k_2}\bigr|^L\biggr]<\infty
\end{align*}
for all $L\geq 1$. Therefore, we obtain (\ref{R}), which completes the proof.
\end{proof}
\begin{lemma}[Propositions 2-4 in Kusano and Uchida \cite{Kusano(jump)}]\label{Elemma}
Suppose that $[{\bf{A1}}]$-$[{\bf{A4}}]$ hold. Then, 
\begin{align*}
    {\bf{E}}\Bigl[(\Delta_i^n X)^{(j_1)}(\Delta_i^n X)^{(j_2)}{\bf{1}}_{\{|\Delta_i^n X|\leq Dh_n^{\rho}\}}\big|\mathcal{F}_{i-1}^n\Bigr]=h_n {\bf{\Sigma}}(\theta_0)_{j_1j_2}
    +R_{i-1}(h_n^{2},\xi,\delta,\varepsilon,\zeta)
\end{align*}
and
\begin{align*}
    {\bf{E}}\Bigl[(\Delta_{i}^n X)^{(j_1)}(\Delta_{i}^n X)^{(j_2)}(\Delta_{i}^n X)^{(j_3)}(\Delta_{i}^n X)^{(j_4)}{\bf{1}}_{\{|\Delta_{i}^n X|\leq Dh_n^{\rho}\}}|\mathcal{F}^n_{i-1}\Bigr]=R_{i-1}(h_n^{2},\xi,\delta,\varepsilon,\zeta)
\end{align*}
for a sufficiently large $n$ and $j_1,j_2,j_3,j_4=1,\ldots,p$.
\end{lemma}
\begin{lemma}\label{Ylemma}
Suppose that $[{\bf{A1}}]$-$[{\bf{A4}}]$ hold. Then, for all $L>0$ and $\varepsilon\in(0,1/4]$,
\begin{align*}
    \sup_{n\in\mathbb{N}}{\bf{E}}\Biggl[\biggl(n^{\varepsilon}\sup_{\theta\in\Theta}\bigl|{\bf{Y}}_n(\theta)-{\bf{Y}}(\theta)\bigr|\biggr)^L\Biggr]<\infty.
\end{align*}
\end{lemma}
\begin{proof}
Note that
\begin{align*}
    {\bf{Y}}_n(\theta)-{\bf{Y}}(\theta)
    &=-\frac{1}{2nh_n}\sum_{i=1}^n\sum_{k_1=1}^p\sum_{k_2=1}^p\bigl({\bf{\Sigma}}(\theta)^{-1}-{\bf{\Sigma}}(\theta_0)^{-1}\bigr)_{k_1k_2}\\
    &\qquad\qquad\qquad\times(\Delta_{i}^n X)^{(k_1)}(\Delta_{i}^n X)^{(k_2)}{\bf{1}}_{\{|\Delta_{i}^n X|\leq Dh_n^{\rho}\}}\\
    &\qquad-\frac{1}{2}\log\frac{\det {\bf{\Sigma}}(\theta)}{\det {\bf{\Sigma}}(\theta_0)}\biggl\{\frac{1}{n}\sum_{i=1}^n{\bf{1}}_{\{|\Delta_{i}^n X|\leq Dh_n^{\rho}\}}-1\biggr\}\\
    &\qquad+\frac{1}{2}\tr\Bigl(\bigl({\bf{\Sigma}}(\theta)^{-1}-{\bf{\Sigma}}(\theta_0)^{-1}\bigr){\bf{\Sigma}}(\theta_0)\Bigr)\\
    &=-\frac{1}{2nh_n}\sum_{i=1}^n\sum_{k_1=1}^p\sum_{k_2=1}^p\bigl({\bf{\Sigma}}(\theta)^{-1}-{\bf{\Sigma}}(\theta_0)^{-1}\bigr)_{k_1k_2}\\
    &\qquad\times\Bigl\{(\Delta_{i}^n X)^{(k_1)}(\Delta_{i}^n X)^{(k_2)}{\bf{1}}_{\{|\Delta_{i}^n X|\leq Dh_n^{\rho}\}}-h_n{\bf{\Sigma}}(\theta_0)_{k_1k_2}\Bigr\}\\
    &\qquad-\frac{1}{2}\log\frac{\det {\bf{\Sigma}}(\theta)}{\det {\bf{\Sigma}}(\theta_0)}\biggl\{\frac{1}{n}\sum_{i=1}^n{\bf{1}}_{\{|\Delta_{i}^n X|\leq Dh_n^{\rho}\}}-1\biggr\}.
\end{align*}
A decomposition is given by
\begin{align*}
    {\bf{Y}}_n(\theta)-{\bf{Y}}(\theta)
    &=-\frac{1}{2}\sum_{k_1=1}^p\sum_{k_2=1}^p\bigl({\bf{\Sigma}}(\theta)^{-1}-{\bf{\Sigma}}(\theta_0)^{-1}\bigr)_{k_1k_2}{\bf{M}}^{\dagger}_{1,n,k_1,k_2}\\
    &\qquad-\frac{1}{2}\sum_{k_1=1}^p\sum_{k_2=1}^p\bigl({\bf{\Sigma}}(\theta)^{-1}-{\bf{\Sigma}}(\theta_0)^{-1}\bigr)_{k_1k_2}{\bf{R}}^{\dagger}_{1,n,k_1,k_2}\\
    &\qquad\qquad\quad-\frac{1}{2}\log\frac{\det {\bf{\Sigma}}(\theta)}{\det {\bf{\Sigma}}(\theta_0)}{\bf{M}}^{\dagger}_{2,n}-\frac{1}{2}\log\frac{\det {\bf{\Sigma}}(\theta)}{\det {\bf{\Sigma}}(\theta_0)}{\bf{R}}^{\dagger}_{2,n},
\end{align*}
where 
\begin{align*}
    {\bf{M}}^{\dagger}_{1,n,k_1,k_2}&=\frac{1}{nh_n}\sum_{i=1}^n\biggl\{(\Delta_{i}^n X)^{(k_1)}(\Delta_{i}^n X)^{(k_2)}{\bf{1}}_{\{|\Delta_{i}^n X|\leq Dh_n^{\rho}\}}\\
    &\qquad\qquad\qquad-{\bf{E}}\Bigl[(\Delta_i^n X)^{(k_1)}(\Delta_i^n X)^{(k_2)}{\bf{1}}_{\{|\Delta_i^n X|\leq Dh_n^{\rho}\}}\big|\mathcal{F}_{i-1}^n\Bigr]\biggr\},\\
    {\bf{R}}^{\dagger}_{1,n,k_1,k_2}&=\frac{1}{nh_n}\sum_{i=1}^n\biggl\{{\bf{E}}\Bigl[(\Delta_i^n X)^{(k_1)}(\Delta_i^n X)^{(k_2)}{\bf{1}}_{\{|\Delta_i^n X|\leq Dh_n^{\rho}\}}\big|\mathcal{F}_{i-1}^n\Bigr]-h_n{\bf{\Sigma}}(\theta_0)_{k_1k_2}\biggr\}
\end{align*}
and
\begin{align*}
    {\bf{M}}^{\dagger}_{2,n}&=\frac{1}{n}\sum_{i=1}^n\biggl\{{\bf{1}}_{\{|\Delta_{i}^n X|\leq Dh_n^{\rho}\}}-{\bf{E}}\Bigl[{\bf{1}}_{\{|\Delta_i^n X|\leq Dh_n^{\rho}\}}\big|\mathcal{F}_{i-1}^n\Bigr]\biggr\},\\
    {\bf{R}}^{\dagger}_{2,n}&=\frac{1}{n}\sum_{i=1}^n\biggl\{{\bf{E}}\Bigl[{\bf{1}}_{\{|\Delta_i^n X|\leq Dh_n^{\rho}\}}\big|\mathcal{F}_{i-1}^n\Bigr]-1\biggr\}.
\end{align*}
It follows that
\begin{align*}
    \sup_{\theta\in\Theta}\bigl|{\bf{Y}}_n(\theta)-{\bf{Y}}(\theta)\bigr|
    &\leq \frac{1}{2} \sum_{k_1=1}^p\sum_{k_2=1}^p\sup_{\theta\in\Theta}
    \bigl|\bigl({\bf{\Sigma}}(\theta)^{-1}-{\bf{\Sigma}}(\theta_0)^{-1}\bigr)_{k_1k_2}\bigr|\bigl|{\bf{M}}^{\dagger}_{1,n,k_1,k_2}\bigr|\\
    &\quad+\frac{1}{2}\sum_{k_1=1}^p\sum_{k_2=1}^p\sup_{\theta\in\Theta}
    \bigl|\bigl({\bf{\Sigma}}(\theta)^{-1}-{\bf{\Sigma}}(\theta_0)^{-1}\bigr)_{k_1k_2}\bigr|\bigl|{\bf{R}}^{\dagger}_{1,n,k_1,k_2}\bigr|\\
    &\quad+\frac{1}{2}\sup_{\theta\in\Theta}\biggl|\log\frac{\det {\bf{\Sigma}}(\theta)}{\det {\bf{\Sigma}}(\theta_0)}\biggr|\bigl|{\bf{M}}^{\dagger}_{2,n}\bigr|+
    \frac{1}{2}\sup_{\theta\in\Theta}\biggl|\log\frac{\det {\bf{\Sigma}}(\theta)}{\det {\bf{\Sigma}}(\theta_0)}\biggr|\bigl|{\bf{R}}^{\dagger}_{2,n}\bigr|\\
    &\leq C\sum_{k_1=1}^p\sum_{k_2=1}^p\bigl|{\bf{M}}^{\dagger}_{1,n,k_1,k_2}\bigr|\\
    &\quad+C\sum_{k_1=1}^p\sum_{k_2=1}^p\bigl|{\bf{R}}^{\dagger}_{1,n,k_1,k_2}\bigr|+C\bigl|{\bf{M}}^{\dagger}_{2,n}\bigr|+C\bigl|{\bf{R}}^{\dagger}_{2,n}\bigr|,
\end{align*}
which yields
\begin{align*}
    {\bf{E}}\Biggl[\biggl(n^{\varepsilon}\sup_{\theta\in\Theta}\bigl|{\bf{Y}}_n(\theta)-{\bf{Y}}(\theta)\bigr|\biggr)^L\Biggr]
    &\leq C_L\sum_{k_1=1}^p\sum_{k_2=1}^p{\bf{E}}\Biggl[\biggl(n^{\varepsilon}\bigl|{\bf{M}}^{\dagger}_{1,n,k_1,k_2}\bigr|\biggr)^L\Biggr]\\
    &\quad+C_L\sum_{k_1=1}^p\sum_{k_2=1}^p{\bf{E}}\Biggl[\biggl(n^{\varepsilon}\bigl|{\bf{R}}^{\dagger}_{1,n,k_1,k_2}\bigr|\biggr)^L\Biggr]\\
    &\quad+C_L{\bf{E}}\Biggl[\biggl(n^{\varepsilon}\bigl|{\bf{M}}^{\dagger}_{2,n}\bigr|\biggr)^L\Biggr]+C_L{\bf{E}}\Biggl[\biggl(n^{\varepsilon}\bigl|{\bf{R}}^{\dagger}_{2,n}\bigr|\biggr)^L\Biggr]
\end{align*}
for all $L\geq 1$. Hence, it is sufficient to prove
\begin{align}
    &\sum_{k_1=1}^p\sum_{k_2=1}^p\sup_{n\in\mathbb{N}}{\bf{E}}\Biggl[\biggl(n^{\varepsilon}\bigl|{\bf{M}}^{\dagger}_{1,n,k_1,k_2}\bigr|\biggr)^L\Biggr]<\infty,\label{M1d}\\
    &\sum_{k_1=1}^p\sum_{k_2=1}^p\sup_{n\in\mathbb{N}}{\bf{E}}\Biggl[\biggl(n^{\varepsilon}\bigl|{\bf{R}}^{\dagger}_{1,n,k_1,k_2}\bigr|\biggr)^L\Biggr]<\infty,\label{R1d}\\
    &\qquad\quad\sup_{n\in\mathbb{N}}{\bf{E}}\Biggl[\biggl(n^{\varepsilon}\bigl|{\bf{M}}^{\dagger}_{2,n}\bigr|\biggr)^L\Biggr]<\infty \label{M2d}
\end{align}
and
\begin{align}
    \sup_{n\in\mathbb{N}}{\bf{E}}\Biggl[\biggl(n^{\varepsilon}\bigl|{\bf{R}}^{\dagger}_{2,n}\bigr|\biggr)^L\Biggr]
    <\infty \label{R2d}
\end{align}
for all $L\geq 1$. First, we show (\ref{M1d}). Note that
\begin{align*}
    {\bf{M}}^{\dagger}_{1,n,k_1,k_2}&=\frac{1}{nh_n}\sum_{i=1}^n{\bf{L}}^{\dagger}_{1,i,n,k_1,k_2},
\end{align*}
where
\begin{align*}
    {\bf{L}}^{\dagger}_{1,i,n,k_1,k_2}&=(\Delta_{i}^n X)^{(k_1)}(\Delta_{i}^n X)^{(k_2)}{\bf{1}}_{\{|\Delta_{i}^nX|\leq Dh_n^{\rho}\}}\\
    &\qquad\qquad-{\bf{E}}\Bigl[(\Delta_i^n X)^{(k_1)}(\Delta_i^n X)^{(k_2)}{\bf{1}}_{\{|\Delta_i^n X|\leq Dh_n^{\rho}\}}\big|\mathcal{F}_{i-1}^n\Bigr].
\end{align*}
In an analogous manner to the proof of Lemma \ref{H3lemma}, it is shown that
\begin{align*}
    {\bf{E}}\biggl[\bigl|{\bf{M}}^{\dagger}_{1,n,k_1,k_2}\bigr|^L\biggr]
    &\leq \frac{C_L}{n^{L/2}h_n^L}\times\frac{1}{n}
    \sum_{i=1}^n {\bf{E}}\biggl[\bigl|{\bf{L}}^{\dagger}_{1,i,n,k_1,k_2}\bigr|^L\biggr]
\end{align*}
and
\begin{align*}
    {\bf{E}}\biggl[\bigl|{\bf{L}}^{\dagger}_{1,i,n,k_1,k_2}\bigr|^L\biggr]\leq C_Lh_n^{2L\rho}
\end{align*}
for all $L\geq 2$, so that
\begin{align*}
    {\bf{E}}\Biggl[\biggl(n^{\varepsilon}\bigl|{\bf{M}}^{\dagger}_{1,n,k_1,k_2}\bigr|\biggr)^L\Biggr]\leq C_Ln^{\varepsilon L}\times \frac{h_n^{2L\rho}}{n^{L/2} h_n^L}
\end{align*}
for all $L\geq 2$. Since $\rho\in[3/8,1/2)$ and $\varepsilon\in(0,1/4]$,
\begin{align*}
    n^{\varepsilon L}\times \frac{h_n^{2L\rho}}{n^{L/2} h_n^L}=n^{(\varepsilon+1/2-2\rho)L}
\end{align*}
and
\begin{align*}
    \varepsilon+\frac{1}{2}-2\rho\leq \frac{1}{4}+\frac{1}{2}-\frac{3}{4}=0,
\end{align*}
which implies
\begin{align*}
    \sup_{n\in\mathbb{N}}{\bf{E}}\Biggl[\biggl(n^{\varepsilon}\bigl|{\bf{M}}^{\dagger}_{1,n,k_1,k_2}\bigr|\biggr)^L\Biggr]<\infty
\end{align*}
for all $L\geq 2$. Consequently, we obtain (\ref{M1d}). From Lemma \ref{Elemma}, one has
\begin{align*}
    {\bf{R}}^{\dagger}_{1,n,k_1,k_2}
    &=\frac{1}{nh_n}\sum_{i=1}^n R_{i-1}(h_n^2,\xi,\delta,\varepsilon,\zeta)
\end{align*}
for a sufficiently large $n$, so that
\begin{align*}
    {\bf{E}}\Biggl[\biggl(n^{\varepsilon}\bigl|{\bf{R}}^{\dagger}_{1,n,k_1,k_2}\bigr|\biggr)^L\Biggr]&=n^{\varepsilon L}h_n^L
    {\bf{E}}\Biggl[\biggl|\frac{1}{n}\sum_{i=1}^n R_{i-1}(1,\xi,\delta,\varepsilon,\zeta)\biggr|^L\Biggr]\\
    &\leq n^{\varepsilon L}h_n^L\times\frac{1}{n}\sum_{i=1}^n{\bf{E}}\biggl[\bigl|R_{i-1}(1,\xi,\delta,\varepsilon,\zeta)\bigr|^L\biggr]\\
    &\leq C_Ln^{\varepsilon L}h_n^L
\end{align*}
for a sufficiently large $n$ and all $L\geq 1$. Using the fact that
\begin{align*}
    n^{\varepsilon L}h_n^L=n^{(\varepsilon-1)L}\longrightarrow 0
\end{align*}
as $n\longrightarrow\infty$, one gets (\ref{R1d}). Set
\begin{align*}
    {\bf{L}}^{\dagger}_{2,i,n}={\bf{1}}_{\{|\Delta_{i}^n X|\leq Dh_n^{\rho}\}}-{\bf{E}}\Bigl[{\bf{1}}_{\{|\Delta_{i}^n X|\leq Dh_n^{\rho}\}}\big|\mathcal{F}^n_{i-1}\Bigr].
\end{align*}
It holds that
\begin{align*}
    {\bf{M}}^{\dagger}_{2,n}=\frac{1}{n}\sum_{i=1}^n{\bf{L}}^{\dagger}_{2,i,n}.
\end{align*}
In a similar way to the proof of Lemma \ref{H3lemma}, we can prove that
\begin{align*}
    {\bf{E}}\biggl[\bigl|{\bf{M}}^{\dagger}_{2,n}\bigr|^L\biggr]
    &\leq \frac{C_L}{n^{L/2}}\times\frac{1}{n}
    \sum_{i=1}^n {\bf{E}}\biggl[\bigl|{\bf{L}}^{\dagger}_{2,i,n}\bigr|^L\biggr]
\end{align*}
and
\begin{align*}
    {\bf{E}}\biggl[\bigl|{\bf{L}}^{\dagger}_{2,i,n}\bigr|^L\biggr]&\leq
    C_L{\bf{E}}\biggl[\bigl|
    {\bf{1}}_{\{|\Delta_{i}^n X|\leq Dh_n^{\rho}\}}\bigr|^{L}\biggr]+C_L{\bf{E}}\biggl[\Bigl|{\bf{E}}\Bigl[{\bf{1}}_{\{|\Delta_{i}^n X|\leq Dh_n^{\rho}\}}\big|\mathcal{F}^n_{i-1}\Bigr]\Bigr|^{L}\biggr]\\
    &\leq C_L{\bf{E}}\biggl[\bigl|
    {\bf{1}}_{\{|\Delta_{i}^n X|\leq Dh_n^{\rho}\}}\bigr|^{L}\biggr]\leq C_L
\end{align*}
for all $L\geq 2$, which yields
\begin{align*}
    {\bf{E}}\biggl[\bigl|n^{\varepsilon}{\bf{M}}^{\dagger}_{2,n}\bigr|^L\biggr]\leq \frac{C_Ln^{\varepsilon L}}{n^{L/2}}
\end{align*}
for all $L\geq 2$. Noting that $\varepsilon\in(0,1/4]$, we have
\begin{align*}
    \frac{n^{\varepsilon L}}{n^{L/2}}=n^{(\varepsilon-1/2)L}\longrightarrow 0
\end{align*}
as $n\longrightarrow\infty$. Hence, (\ref{M2d}) holds. As it follows from the proof of Proposition 1 in Kusano and Uchida \cite{Kusano(jump)} that
\begin{align*}
    {\bf{P}}\Bigl(|\Delta_{i}^n X|\leq Dh_n^{\rho}\big|\mathcal{F}_{i-1}^n\Bigr)
    &=\bar{R}_{i-1}(h_n,\xi,\delta,\varepsilon,\zeta)
\end{align*}
for a sufficiently large $n$,
\begin{align*}
    \left| {\bf{E}}\Bigl[{\bf{1}}_{\{|\Delta_{i}^n X|\leq Dh_n^{\rho}\}}\big|\mathcal{F}_{i-1}^n\Bigr]-1 \right|
    &= \left| {\bf{P}}\Bigl(|\Delta_{i}^n X|\leq Dh_n^{\rho}\big|\mathcal{F}_{i-1}^n\Bigr)-1 \right| \\
    &=R_{i-1}(h_n,\xi,\delta,\varepsilon,\zeta)
\end{align*}
for a sufficiently large $n$. Therefore, one gets
\begin{align*}
    {\bf{E}}\biggl[\bigl|n^{\varepsilon}{\bf{R}}^{\dagger}_{2,n}\bigr|^L\biggr]
    &\leq n^{\varepsilon L}{\bf{E}}\Biggl[\biggl( \frac{1}{n}\sum_{i=1}^n
\left| {\bf{E}}\Bigl[{\bf{1}}_{\{|\Delta_{i}^n X|\leq Dh_n^{\rho}\}}\big|\mathcal{F}_{i-1}^n\Bigr]-1 \right| \biggr)^L\Biggr]\\
    &\leq n^{\varepsilon L}\times\frac{1}{n}\sum_{i=1}^n {\bf{E}}\biggl[\bigl|R_{i-1}(h_n,\xi,\delta,\varepsilon,\zeta)\bigr|^L\biggr]\\
    &\leq C_L n^{\varepsilon L}h_n^L=C_Ln^{(\varepsilon-1)L}
\end{align*}
for a sufficiently large $n$ and all $L\geq 1$, which implies (\ref{R2d}). This completes the proof.
\end{proof}
\begin{lemma}\label{H2lemma}
Suppose that $[{\bf{A1}}]$-$[{\bf{A4}}]$ hold. Then, for all $L>0$ and $\varepsilon\in(0,1/4]$,
\begin{align*}
     \sup_{n\in\mathbb{N}}{\bf{E}}\Biggl[\biggl(n^{\varepsilon}\biggl|\frac{1}{n}\partial^2_{\theta}{\bf{H}}_n(\theta_0)+{\bf{I}}(\theta_0)\biggr|\biggr)^L\Biggr]<\infty.
\end{align*}
\end{lemma}
\begin{proof}
It is sufficient to prove
\begin{align}
    \sup_{n\in\mathbb{N}}{\bf{E}}\Biggl[\biggl(n^{\varepsilon}\biggl|\frac{1}{n}\partial_{\theta^{(j_1)}}\partial_{\theta^{(j_2)}}{\bf{H}}_n(\theta_0)+{\bf{I}}(\theta_0)_{j_1j_2}\biggr|\biggr)^L\Biggr]<\infty \label{H2j}
\end{align}
for $j_1,j_2=1,\ldots,q$ and all $L>0$. Recall that 
\begin{align*}
    {\bf{M}}^{\dagger}_{1,n,k_1,k_2}&=\frac{1}{nh_n}\sum_{i=1}^n\biggl\{(\Delta_{i}^n X)^{(k_1)}(\Delta_{i}^n X)^{(k_2)}{\bf{1}}_{\{|\Delta_{i}^n X|\leq Dh_n^{\rho}\}}\\
    &\qquad\qquad\qquad-{\bf{E}}\Bigl[(\Delta_i^n X)^{(k_1)}(\Delta_i^n X)^{(k_2)}{\bf{1}}_{\{|\Delta_i^n X|\leq Dh_n^{\rho}\}}\big|\mathcal{F}_{i-1}^n\Bigr]\biggr\},\\
    {\bf{R}}^{\dagger}_{1,n,k_1,k_2}&=\frac{1}{nh_n}\sum_{i=1}^n\biggl\{{\bf{E}}\Bigl[(\Delta_i^n X)^{(k_1)}(\Delta_i^n X)^{(k_2)}{\bf{1}}_{\{|\Delta_i^n X|\leq Dh_n^{\rho}\}}\big|\mathcal{F}_{i-1}^n\Bigr]-h_n{\bf{\Sigma}}(\theta_0)_{k_1k_2}\biggr\}
\end{align*}
and
\begin{align*}
    {\bf{M}}^{\dagger}_{2,n}&=\frac{1}{n}\sum_{i=1}^n\biggl\{{\bf{1}}_{\{|\Delta_{i}^n X|\leq Dh_n^{\rho}\}}-{\bf{E}}\Bigl[{\bf{1}}_{\{|\Delta_i^n X|\leq Dh_n^{\rho}\}}\big|\mathcal{F}_{i-1}^n\Bigr]\biggr\},\\
    {\bf{R}}^{\dagger}_{2,n}&=\frac{1}{n}\sum_{i=1}^n\biggl\{{\bf{E}}\Bigl[{\bf{1}}_{\{|\Delta_i^n X|\leq Dh_n^{\rho}\}}\big|\mathcal{F}_{i-1}^n\Bigr]-1\biggr\}
\end{align*}
for $k_1,k_2=1,\ldots,p$. Since
\begin{align*}
    &\quad\ \frac{1}{n}\partial_{\theta^{(j_1)}}\partial_{\theta^{(j_2)}}{\bf{H}}_n(\theta_0)+{\bf{I}}(\theta_0)_{j_1j_2}\\
    &=-\frac{1}{2nh_n}\sum_{i=1}^n\sum_{k_1=1}^p\sum_{k_2=1}^p\bigl(\partial_{\theta^{(j_1)}}\partial_{\theta^{(j_2)}}{\bf{\Sigma}}(\theta_0)^{-1}\bigr)_{k_1k_2}(\Delta_{i}^n X)^{(k_1)}(\Delta_{i}^n X)^{(k_2)}{\bf{1}}_{\{|\Delta_{i}^n X|\leq Dh_n^{\rho}\}}\\
    &\qquad+\frac{1}{2}\sum_{k_1=1}^p\sum_{k_2=1}^p
    \bigl(\partial_{\theta^{(j_1)}}\partial_{\theta^{(j_2)}}{\bf{\Sigma}}(\theta_0)^{-1}\bigr)_{k_1k_2}{\bf{\Sigma}}(\theta_0)_{k_1k_2}\\
    &\qquad-\frac{1}{2}\partial_{\theta^{(j_1)}}\partial_{\theta^{(j_2)}}
    \log\det {\bf{\Sigma}}(\theta_0)\biggl\{\frac{1}{n}\sum_{i=1}^n{\bf{1}}_{\{|\Delta_{i}^n X|\leq Dh_n^{\rho}\}}-1\biggr\}\\
    &=-\frac{1}{2nh_n}\sum_{i=1}^n\sum_{k_1=1}^p\sum_{k_2=1}^p\bigl(\partial_{\theta^{(j_1)}}\partial_{\theta^{(j_2)}}{\bf{\Sigma}}(\theta_0)^{-1}\bigr)_{k_1k_2}\\
    &\qquad\qquad\qquad\qquad\times\Bigl\{(\Delta_{i}^n X)^{(k_1)}(\Delta_{i}^n X)^{(k_2)}{\bf{1}}_{\{|\Delta_{i}^n X|\leq Dh_n^{\rho}\}}-h_n{\bf{\Sigma}}(\theta_0)_{k_1k_2}\Bigr\}\\
    &\qquad-\frac{1}{2}\partial_{\theta^{(j_1)}}\partial_{\theta^{(j_2)}}
    \log\det {\bf{\Sigma}}(\theta_0)\biggl\{\frac{1}{n}\sum_{i=1}^n{\bf{1}}_{\{|\Delta_{i}^n X|\leq Dh_n^{\rho}\}}-1\biggr\},
\end{align*}
we have the following expression:
\begin{align*}
    \frac{1}{n}\partial_{\theta^{(j_1)}}\partial_{\theta^{(j_2)}}{\bf{H}}_n(\theta_0)+{\bf{I}}(\theta_0)_{j_1j_2}
    &=-\frac{1}{2}\sum_{k_1=1}^p\sum_{k_2=1}^p\bigl(\partial_{\theta^{(j_1)}}\partial_{\theta^{(j_2)}}{\bf{\Sigma}}(\theta_0)^{-1}\bigr)_{k_1k_2}{\bf{M}}^{\dagger}_{1,n,k_1,k_2}\\
    &\quad-\frac{1}{2}\sum_{k_1=1}^p\sum_{k_2=1}^p\bigl(\partial_{\theta^{(j_1)}}\partial_{\theta^{(j_2)}}{\bf{\Sigma}}(\theta_0)^{-1}\bigr)_{k_1k_2}{\bf{R}}^{\dagger}_{1,n,k_1,k_2}\\
    &\quad-\frac{1}{2}\Bigl(\partial_{\theta^{(j_1)}}\partial_{\theta^{(j_2)}}
    \log\det {\bf{\Sigma}}(\theta_0)\Bigr){\bf{M}}^{\dagger}_{2,n}\\
    &\quad-\frac{1}{2}\Bigl(\partial_{\theta^{(j_1)}}\partial_{\theta^{(j_2)}}
    \log\det {\bf{\Sigma}}(\theta_0)\Bigr){\bf{R}}^{\dagger}_{2,n}.
\end{align*}
Therefore, it holds from (\ref{M1d})-(\ref{R2d}) that
\begin{align*}
    &\quad\ \sup_{n\in\mathbb{N}}{\bf{E}}\Biggl[\biggl(n^{\varepsilon}\biggl|\frac{1}{n}\partial_{\theta^{(j_1)}}\partial_{\theta^{(j_2)}}{\bf{H}}_n(\theta_0)+{\bf{I}}(\theta_0)_{j_1j_2}\biggr|\biggr)^L\Biggr]\\
    &\leq C_L\sum_{k_1=1}^p\sum_{k_2=1}^p\sup_{n\in\mathbb{N}}{\bf{E}}\Biggl[\biggl(n^{\varepsilon}\bigl|{\bf{M}}^{\dagger}_{1,n,k_1,k_2}\bigr|\biggr)^L\Biggr]\\
    &\qquad+C_L\sum_{k_1=1}^p\sum_{k_2=1}^p\sup_{n\in\mathbb{N}}{\bf{E}}\Biggl[\biggl(n^{\varepsilon}\bigl|{\bf{R}}^{\dagger}_{1,n,k_1,k_2}\bigr|\biggr)^L\Biggr]\\
    &\qquad+C_L\sup_{n\in\mathbb{N}}{\bf{E}}\Biggl[\biggl(n^{\varepsilon}\bigl|{\bf{M}}^{\dagger}_{2,n}\bigr|\biggr)^L\Biggr]+C_L\sup_{n\in\mathbb{N}}{\bf{E}}\Biggl[\biggl(n^{\varepsilon}\bigl|{\bf{R}}^{\dagger}_{2,n}\bigr|\biggr)^L\Biggr]<\infty
\end{align*}
for all $L\geq 1$, which deduces (\ref{H2j}). This completes the proof.
\end{proof}
\begin{lemma}[Proposition 3 in Ogihara and Yoshida \cite{Ogihara(2011)}]\label{Ogihara}
Let $k\in\mathbb{N}$ and $L\geq 2^{k-1}$. Assume that $\{F_i\}_{i=1}^n$ is 
a sequence of real-valued random variables adapted to the filtration $(\mathcal{F}_{i})_{i=0}^n$ such that
\begin{align*}
    {\bf{E}}\Bigl[|F_i|^L\Bigr]<\infty
\end{align*}
for $i=1,\ldots,n$. Then,
\begin{align*}
\begin{split}
    {\bf{E}}\Biggl[\biggl|\sum_{i=1}^nF_i\biggr|^L\Biggr]&\leq C_{L,k}{\bf{E}}
    \Biggl[\biggl|\sum_{i=1}^n\psi_i^{k+1}(F_i)\biggr|^{L/2^k}\Biggr]\\
    &\qquad+C_{L,k}\sum_{j=1}^k{\bf{E}}
    \Biggl[\biggl|\sum_{i=1}^n{\bf{E}}\bigl[\psi_i^{j}(F_i)|\mathcal{F}^n_{i-1}\bigr]\biggr|^{L/2^{j-1}}\Biggr],
\end{split}
\end{align*}
where $\psi_i^{1}(F)=F$ and 
\begin{align*}
    \psi_i^{j+1}(F)=\Bigl(\psi_i^{j}(F)-{\bf{E}}\bigl[\psi_i^{j}(F)|\mathcal{F}^n_{i-1}\bigr]\Bigr)^2
\end{align*}
for $j\in\mathbb{N}$.
\end{lemma}
\begin{lemma}\label{H1lemma}
Suppose that $[{\bf{A1}}]$-$[{\bf{A4}}]$ hold. Then, for any $L>0$,
\begin{align*}
    \sup_{n\in\mathbb{N}} {\bf{E}}\Biggl[\biggl|\frac{1}{\sqrt{n}}\partial_{\theta}{\bf{H}}_n(\theta_0)\biggr|^L\Biggr]<\infty.
\end{align*}
\end{lemma}
\begin{proof}
It is enough to show
\begin{align*}
    \sup_{n\in\mathbb{N}} {\bf{E}}\Biggl[\biggl|\frac{1}{\sqrt{n}}\partial_{\theta^{(j)}}{\bf{H}}_n(\theta_0)\biggr|^L\Biggr]<\infty
\end{align*}
for $j=1,\ldots,q$ and all $L>0$. Set
\begin{align*}
     K_{i,n}^{(j)}&=\Bigl\{h_n^{-1}(\Delta_{i}^n X)^{\top}\partial_{\theta^{(j)}}{\bf{\Sigma}}(\theta_0)^{-1}(\Delta_{i}^n X)+\partial_{\theta^{(j)}}\log\det {\bf{\Sigma}}(\theta_0)\Bigr\}{\bf{1}}_{\{|\Delta_{i}^n X|\leq Dh_n^{\rho}\}}
\end{align*}
for $i=1,\ldots,n$. Note that 
\begin{align*}
    \bigl\{K_{i,n}^{(j)}\bigr\}_{i=1}^n
\end{align*}
is adapted to $(\mathcal{F}_{i})_{i=0}^n$. Since $1-2\rho>0$ and
\begin{align*}
    {\bf{E}}\Bigl[\bigl|K_{i,n}^{(j)}\bigr|^L\Bigr]
    &\leq C_Lh_n^{-L}{\bf{E}}\Bigl[\bigl|(\Delta_{i}^n X)^{\top}\partial_{\theta^{(j)}}{\bf{\Sigma}}(\theta_0)^{-1}(\Delta_{i}^n X)\bigr|^L{\bf{1}}_{\{|\Delta_{i}^n X|\leq Dh_n^{\rho}\}}\Bigr]\\
    &\qquad\qquad\qquad\quad+C_L\bigl|\partial_{\theta^{(j)}}\log\det {\bf{\Sigma}}(\theta_0)\bigr|^L{\bf{E}}\Bigl[{\bf{1}}_{\{|\Delta_{i}^n X|\leq Dh_n^{\rho}\}}\Bigr]\\
    &\leq C_Lh_n^{-L}\bigl|\partial_{\theta^{(j)}}{\bf{\Sigma}}(\theta_0)^{-1}\bigr|^L{\bf{E}}\Bigl[\bigl|\Delta_{i}^n X\bigr|^{2L}{\bf{1}}_{\{|\Delta_{i}^n X|\leq Dh_n^{\rho}\}}\Bigr]+C_L\\
    &\leq C_Lh_n^{-L}\times h_n^{2L\rho}+C_L\\
    &\leq C_Ln^{L(1-2\rho)}+C_L
\end{align*}
for all $L\geq 1$, we obtain
\begin{align}
    {\bf{E}}\Bigl[\bigl|K_{i,n}^{(j)}\bigr|^L\Bigr]\leq C_Ln^{L(1-2\rho)}.\label{KL}
\end{align}
Noting that
\begin{align*}
    \frac{1}{\sqrt{n}}\partial_{\theta^{(j)}}{\bf{H}}_n(\theta_0)=-\frac{1}{2\sqrt{n}}\sum_{i=1}^n K_{i,n}^{(j)},
\end{align*}
by Lemma \ref{Ogihara}, one gets
\begin{align*}
    {\bf{E}}\Biggl[\biggl|\frac{1}{\sqrt{n}}\partial_{\theta^{(j)}}{\bf{H}}_n(\theta_0)\biggr|^L\Biggr]&\leq\frac{C_L}{n^{L/2}}{\bf{E}}\Biggl[\biggl|\sum_{i=1}^nK_{i,n}^{(j)}\biggr|^L\Biggr]\\
    &\leq \frac{C_L}{n^{L/2}}{\bf{E}}
    \Biggl[\biggl|\sum_{i=1}^n\psi_i^{3}\bigl(K_{i,n}^{(j)}\bigr)\biggr|^{L/4}\Biggr]\\
    &\quad+\frac{C_L}{n^{L/2}}{\bf{E}}
    \Biggl[\biggl|\sum_{i=1}^n{\bf{E}}\Bigl[\psi_i^{1}\bigl(K_{i,n}^{(j)}\bigr)
    \big|\mathcal{F}^n_{i-1}\Bigr]\biggr|^{L}\Biggr]\\
    &\quad+\frac{C_L}{n^{L/2}}{\bf{E}}
    \Biggl[\biggl|\sum_{i=1}^n{\bf{E}}\Bigl[\psi_i^{2}\bigl(K_{i,n}^{(j)}\bigr)
    \big|\mathcal{F}^n_{i-1}\Bigr]\biggr|^{L/2}\Biggr]
\end{align*}
for all $L\geq 2$. Hence, it is sufficient to prove 
\begin{align}
    \sup_{n\in\mathbb{N}}\frac{1}{n^{L/2}}{\bf{E}}
    \Biggl[\biggl|\sum_{i=1}^n{\bf{E}}\Bigl[\psi_i^{1}\bigl(K_{i,n}^{(j)}\bigr)
    \big|\mathcal{F}^n_{i-1}\Bigr]\biggr|^{L}\Biggr]<\infty, \label{H1-1}\\
    \sup_{n\in\mathbb{N}}\frac{1}{n^{L/2}}{\bf{E}}
    \Biggl[\biggl|\sum_{i=1}^n{\bf{E}}\Bigl[\psi_i^{2}\bigl(K_{i,n}^{(j)}\bigr)
    \big|\mathcal{F}^n_{i-1}\Bigr]\biggr|^{L/2}\Biggr]<\infty \label{H1-2}
\end{align}
and
\begin{align}
    \sup_{n\in\mathbb{N}}\frac{1}{n^{L/2}}{\bf{E}}
    \Biggl[\biggl|\sum_{i=1}^n\psi_i^{3}\bigl(K_{i,n}^{(j)}\bigr)\biggr|^{L/4}\Biggr]<\infty \label{H1-3}
\end{align}
for all $L\geq 2$. \\
\ \\
Proof of (\ref{H1-1}).
For a sufficiently large $n$, we see from Lemma \ref{Elemma} that
\begin{align*}
    &\quad\ {\bf{E}}\Bigl[(\Delta_{i}^n X)^{\top}\partial_{\theta^{(j)}}{\bf{\Sigma}}(\theta_0)^{-1}(\Delta_{i}^n X){\bf{1}}_{\{|\Delta_{i}^n X|\leq Dh_n^{\rho}\}}|\mathcal{F}^n_{i-1}\Bigr]\\
    &=\sum_{k_1=1}^p\sum_{k_2=1}^p\bigl(\partial_{\theta^{(j)}}{\bf{\Sigma}}(\theta_0)^{-1}\bigr)_{k_1k_2}{\bf{E}}\Bigl[(\Delta_{i}^n X)^{(k_1)}(\Delta_{i}^n X)^{(k_2)}{\bf{1}}_{\{|\Delta_{i}^n X|\leq Dh_n^{\rho}\}}|\mathcal{F}^n_{i-1}\Bigr]\\
    &=\sum_{k_1=1}^p\sum_{k_2=1}^p\bigl(\partial_{\theta^{(j)}}{\bf{\Sigma}}(\theta_0)^{-1}\bigr)_{k_1k_2}\Bigl\{h_n {\bf{\Sigma}}(\theta_0)_{k_1k_2}
    +R_{i-1}(h_n^{2},\xi,\delta,\varepsilon,\zeta)\Bigr\}\\
    &=h_n\sum_{k_1=1}^p\sum_{k_2=1}^p\bigl(\partial_{\theta^{(j)}}{\bf{\Sigma}}(\theta_0)^{-1}\bigr)_{k_1k_2}{\bf{\Sigma}}(\theta_0)_{k_2k_1}
    +R_{i-1}(h_n^{2},\xi,\delta,\varepsilon,\zeta)
\end{align*}
and
\begin{align*}
    \sum_{k_1=1}^p\sum_{k_2=1}^p\bigl(\partial_{\theta^{(j)}}{\bf{\Sigma}}(\theta_0)^{-1}\bigr)_{k_1k_2}{\bf{\Sigma}}(\theta_0)_{k_2k_1}
    &=\tr\Bigl(\bigl(\partial_{\theta^{(j)}}{\bf{\Sigma}}(\theta_0)^{-1}\bigr){\bf{\Sigma}}(\theta_0)\Bigr)\\
    &=-\tr\Bigl({\bf{\Sigma}}(\theta_0)^{-1}\bigl(\partial_{\theta^{(j)}}{\bf{\Sigma}}(\theta_0)\bigr)\Bigr),
\end{align*}
which yields
\begin{align}
\begin{split}
    &{\bf{E}}\Bigl[(\Delta_{i}^n X)^{\top}\partial_{\theta^{(j)}}{\bf{\Sigma}}(\theta_0)^{-1}(\Delta_{i}^n X){\bf{1}}_{\{|\Delta_{i}^n X|\leq Dh_n^{\rho}\}}|\mathcal{F}^n_{i-1}\Bigr]\\
    &\qquad\qquad\qquad\qquad=-h_n\tr\Bigl({\bf{\Sigma}}(\theta_0)^{-1}\bigl(\partial_{\theta^{(j)}}{\bf{\Sigma}}(\theta_0)\bigr)\Bigr)+R_{i-1}(h_n^{2},\xi,\delta,\varepsilon,\zeta)
\end{split}\label{EXh}
\end{align}
for a sufficiently large $n$. Moreover, 
\begin{align}
    {\bf{E}}\Bigl[{\bf{1}}_{\{|\Delta_{i}^n X|\leq Dh_n^{\rho}\}}|\mathcal{F}^n_{i-1}\Bigr]
    &=\bar{R}_{i-1}(h_n,\xi,\delta,\varepsilon,\zeta) \label{indi}
\end{align}
for a sufficiently large $n$, so that
\begin{align*}
    &{\bf{E}}\Bigl[\partial_{\theta^{(j)}}\log\det {\bf{\Sigma}}(\theta_0){\bf{1}}_{\{|\Delta_{i}^n X|\leq Dh_n^{\rho}\}}|\mathcal{F}^n_{i-1}\Bigr]\\
    &\qquad\qquad\qquad\qquad=\tr\Bigl({\bf{\Sigma}}(\theta_0)^{-1}\bigl(\partial_{\theta^{(j)}}{\bf{\Sigma}}(\theta_0)\bigr)\Bigr)+R_{i-1}(h_n,\xi,\delta,\varepsilon,\zeta).
\end{align*}
Consequently, one has
\begin{align}
\begin{split}
    {\bf{E}}\Bigl[K_{i,n}^{(j)}\big|\mathcal{F}^n_{i-1}\Bigr]&=h_n^{-1}{\bf{E}}\Bigl[(\Delta_{i}^n X)^{\top}\partial_{\theta^{(j)}}{\bf{\Sigma}}(\theta_0)^{-1}(\Delta_{i}^n X){\bf{1}}_{\{|\Delta_{i}^n X|\leq Dh_n^{\rho}\}}|\mathcal{F}^n_{i-1}\Bigr]\\
    &\quad+{\bf{E}}\Bigl[\partial_{\theta^{(j)}}\log\det {\bf{\Sigma}}(\theta_0){\bf{1}}_{\{|\Delta_{i}^n X|\leq Dh_n^{\rho}\}}|\mathcal{F}^n_{i-1}\Bigr]=R_{i-1}(h_n,\xi,\delta,\varepsilon,\zeta)
\end{split}\label{K}
\end{align}
for a sufficiently large $n$. Therefore, 
\begin{align*}
    \frac{1}{n^{L/2}}{\bf{E}}\Biggl[\biggl|\sum_{i=1}^n{\bf{E}}\Bigl[\psi_i^{1}\bigl(K_{i,n}^{(j)}\bigr)\big|\mathcal{F}^n_{i-1}\Bigr]\biggr|^{L}\Biggr]&=\frac{1}{n^{L/2}} {\bf{E}}
    \Biggl[\biggl|\sum_{i=1}^n{\bf{E}}\Bigl[K_{i,n}^{(j)}\big|\mathcal{F}^n_{i-1}\Bigr]\biggr|^{L}\Biggr]\\
    &=\frac{1}{n^{L/2}} {\bf{E}}
    \Biggl[\biggl|\sum_{i=1}^n R_{i-1}(h_n,\xi,\delta,\varepsilon,\zeta)\biggr|^{L}\Biggr]\\
    &\leq \frac{C_Ln^{L-1}h_n^L}{n^{L/2}}\sum_{i=1}^n {\bf{E}}\Bigl[\bigl|R_{i-1}(1,\xi,\delta,\varepsilon,\zeta)\bigr|^L\Bigr]\\
    &\leq C_Ln^{-L/2}
\end{align*}
for a sufficiently large $n$ and all $ L\geq 1$, which gives
\begin{align*}
    \sup_{n\in\mathbb{N}}\frac{1}{n^{L/2}}{\bf{E}}\Biggl[\biggl|\sum_{i=1}^n{\bf{E}}\Bigl[\psi_i^{1}\bigl(K_{i,n}^{(j)}\bigr)\big|\mathcal{F}^n_{i-1}\Bigr]\biggr|^{L}\Biggr]<\infty
\end{align*}
for all $L\geq 1$. This completes the proof of (\ref{H1-1}).\\
\ \\
Proof of (\ref{H1-2}).
Since 
\begin{align*}
    \psi_i^{2}(F)&=\bigl(\psi_i^{1}(F)-{\bf{E}}\bigl[\psi_i^{1}(F)|\mathcal{F}^n_{i-1}\bigr]\bigr)^2\\
    &=\bigl(F-{\bf{E}}\bigl[F|\mathcal{F}^n_{i-1}\bigr]\bigr)^2=F^2-2F{\bf{E}}\bigl[F|\mathcal{F}^n_{i-1}\bigr]+{\bf{E}}\bigl[F|\mathcal{F}^n_{i-1}\bigr]^2
\end{align*}
for $F\in L^1({\bf{P}})$, we have
\begin{align*}
    {\bf{E}}\bigl[\psi_i^{2}(F)|\mathcal{F}^n_{i-1}\bigr]={\bf{E}}\bigl[F^2|\mathcal{F}^n_{i-1}\bigr]-{\bf{E}}\bigl[F|\mathcal{F}^n_{i-1}\bigr]^2.
\end{align*}
As it follows from Lemma \ref{Elemma}, (\ref{EXh}) and (\ref{indi}) that
\begin{align*}
    &\quad\ {\bf{E}}\Bigl[K_{i,n}^{(j)2}\big|\mathcal{F}^n_{i-1}\Bigr]\\
    &=h_n^{-2}{\bf{E}}\Bigl[\bigl\{(\Delta_{i}^n X)^{\top}\partial_{\theta^{(j)}}{\bf{\Sigma}}(\theta_0)^{-1}(\Delta_{i}^n X)\bigr\}^2{\bf{1}}_{\{|\Delta_{i}^n X|\leq Dh_n^{\rho}\}}\big|\mathcal{F}^n_{i-1}\Bigr]\\
    &\quad+2h_n^{-1}\bigl(\partial_{\theta^{(j)}}\log\det {\bf{\Sigma}}(\theta_0)\bigr){\bf{E}}\Bigl[(\Delta_{i}^n X)^{\top}\partial_{\theta^{(j)}}{\bf{\Sigma}}(\theta_0)^{-1}(\Delta_{i}^n X){\bf{1}}_{\{|\Delta_{i}^n X|\leq Dh_n^{\rho}\}}\big|\mathcal{F}^n_{i-1}\Bigr]\\
    &\quad+\bigl(\partial_{\theta^{(j)}}\log\det {\bf{\Sigma}}(\theta_0)\bigr)^2{\bf{E}}\Bigl[{\bf{1}}_{\{|\Delta_{i}^n X|\leq Dh_n^{\rho}\}}\big|\mathcal{F}^n_{i-1}\Bigr]\\
    &=h_n^{-2}\sum_{k_1=1}^p\sum_{k_2=1}^p\sum_{k_3=1}^p\sum_{k_4=1}^p
    \bigl(\partial_{\theta^{(j)}}{\bf{\Sigma}}(\theta_0)^{-1}\bigr)_{k_1k_2}\bigl(\partial_{\theta^{(j)}}{\bf{\Sigma}}(\theta_0)^{-1}\bigr)_{k_3k_4}\\
    &\qquad\qquad\quad\times {\bf{E}}\Bigl[(\Delta_{i}^n X)^{(k_1)}(\Delta_{i}^n X)^{(k_2)}(\Delta_{i}^n X)^{(k_3)}(\Delta_{i}^n X)^{(k_4)}{\bf{1}}_{\{|\Delta_{i}^n X|\leq Dh_n^{\rho}\}}|\mathcal{F}^n_{i-1}\Bigr]\\
    &\quad+R_{i-1}(1,\xi,\delta,\varepsilon,\zeta)+\bigl(\partial_{\theta^{(j)}}\log\det {\bf{\Sigma}}(\theta_0)\bigr)^2-R_{i-1}(h_n,\xi,\delta,\varepsilon,\zeta)\\
    &=R_{i-1}(1,\xi,\delta,\varepsilon,\zeta)
\end{align*}
for a sufficiently large $n$, we see from (\ref{K}) that
\begin{align*}
    {\bf{E}}\Bigl[\psi_i^{2}\bigl(K_{i,n}^{(j)}\bigr)\big|\mathcal{F}^n_{i-1}\Bigr]&={\bf{E}}\Bigl[K_{i,n}^{(j)2}\big|\mathcal{F}^n_{i-1}\Bigr]-{\bf{E}}\Bigl[K_{i,n}^{(j)}\big|\mathcal{F}^n_{i-1}\Bigr]^2=R_{i-1}(1,\xi,\delta,\varepsilon,\zeta)
\end{align*}
for a sufficiently large $n$. Hence, 
\begin{align*}
    \frac{1}{n^{L/2}}{\bf{E}}
    \Biggl[\biggl|\sum_{i=1}^n{\bf{E}}\Bigl[\psi_i^{2}\bigl(K_{i,n}^{(j)}\bigr)
    \big|\mathcal{F}^n_{i-1}\Bigr]\biggr|^{L/2}\Biggr]&=\frac{1}{n^{L/2}}{\bf{E}}
    \Biggl[\biggl|\sum_{i=1}^n R_{i-1}(1,\xi,\delta,\varepsilon,\zeta) \biggr|^{L/2}\Biggr]\\
    &\leq \frac{C_Ln^{L/2-1}}{n^{L/2}}\sum_{i=1}^n{\bf{E}}\Bigl[\bigl|R_{i-1}(1,\xi,\delta,\varepsilon,\zeta)\bigr|^{L/2}\Bigr]\\
    &\leq C_L
\end{align*}
for a sufficiently large $n$ and all $L\geq 2$, which implies
\begin{align*}
    \sup_{n\in\mathbb{N}}\frac{1}{n^{L/2}}{\bf{E}}
    \Biggl[\biggl|\sum_{i=1}^n{\bf{E}}\Bigl[\psi_i^{2}\bigl(K_{i,n}^{(j)}\bigr)
    \big|\mathcal{F}^n_{i-1}\Bigr]\biggr|^{L/2}\Biggr]<\infty
\end{align*}
for all $L\geq 2$. This completes the proof of (\ref{H1-2}).\\
\ \\
Proof of (\ref{H1-3}).
First, we note that 
\begin{align*}
    \psi_i^{3}(F)
    &=\bigl(\psi_i^{2}(F)-{\bf{E}}\bigl[\psi_i^{2}(F)|\mathcal{F}^n_{i-1}\bigr]\bigr)^2\\
    &=\bigl(F^2-2F{\bf{E}}\bigl[F|\mathcal{F}^n_{i-1}\bigr]+2{\bf{E}}\bigl[F|\mathcal{F}^n_{i-1}\bigr]^2-{\bf{E}}\bigl[F^2|\mathcal{F}^n_{i-1}\bigr]\bigr)^2\\
    &=F^4-4F^3{\bf{E}}\bigl[F|\mathcal{F}^n_{i-1}\bigr]+8F^2{\bf{E}}\bigl[F|\mathcal{F}^n_{i-1}\bigr]^2\\
    &\quad-8
    F{\bf{E}}\bigl[F|\mathcal{F}^n_{i-1}\bigr]^3+4{\bf{E}}\bigl[F|\mathcal{F}^n_{i-1}\bigr]^4-2F^2{\bf{E}}\bigl[F^2|\mathcal{F}^n_{i-1}\bigr]\\
    &\quad+4F{\bf{E}}\bigl[F|\mathcal{F}^n_{i-1}\bigr]{\bf{E}}\bigl[F^2|\mathcal{F}^n_{i-1}\bigr]-4{\bf{E}}\bigl[F|\mathcal{F}^n_{i-1}\bigr]^2{\bf{E}}\bigl[F^2|\mathcal{F}^n_{i-1}\bigr]+{\bf{E}}\bigl[F^2|\mathcal{F}^n_{i-1}\bigr]^2
\end{align*}
for $F\in L^2({\bf{P}})$, so that
\begin{align*}
    {\bf{E}}\Bigl[\big|\psi_i^{3}(F)\bigr|^L\Bigr]&\leq C_L{\bf{E}}\Bigl[|F|^{4L}\Bigr]+C_L{\bf{E}}\Bigl[\big|F^3{\bf{E}}\bigl[F|\mathcal{F}^n_{i-1}\bigr]\bigr|^L\Bigr]+C_L{\bf{E}}\Bigl[\big|F^2{\bf{E}}\bigl[F|\mathcal{F}^n_{i-1}\bigr]^2\bigr|^L\Bigr]\\
    &\quad+C_L{\bf{E}}\Bigl[\big|F{\bf{E}}\bigl[F|\mathcal{F}^n_{i-1}\bigr]^3\bigr|^L\Bigr]+C_L{\bf{E}}\Bigl[\big|{\bf{E}}\bigl[F|\mathcal{F}^n_{i-1}\bigr]\bigr|^{4L}\Bigr]\\
    &\quad+C_L{\bf{E}}\Bigl[\big|F^2{\bf{E}}\bigl[F^2|\mathcal{F}^n_{i-1}\bigr]\bigr|^L\Bigr]+C_L{\bf{E}}\Bigl[\big|F{\bf{E}}\bigl[F|\mathcal{F}^n_{i-1}\bigr]{\bf{E}}\bigl[F^2|\mathcal{F}^n_{i-1}\bigr]\bigr|^L\Bigr]\\
    &\quad+C_L{\bf{E}}\Bigl[\big|{\bf{E}}\bigl[F|\mathcal{F}^n_{i-1}\bigr]^2{\bf{E}}\bigl[F^2|\mathcal{F}^n_{i-1}\bigr]\bigr|^L\Bigr]+C_L{\bf{E}}\Bigl[\big|{\bf{E}}\bigl[F^2|\mathcal{F}^n_{i-1}\bigr]\bigr|^{2L}\Bigr].
\end{align*}
The second term on the right-hand side can be evaluated as follows:
\begin{align*}
    {\bf{E}}\Bigl[\big|F^3{\bf{E}}\bigl[F|\mathcal{F}^n_{i-1}\bigr]\bigr|^L\Bigr]&\leq 
    {\bf{E}}\Bigl[|F|^{6L}\Bigr]^{1/2}{\bf{E}}\Bigl[\big|{\bf{E}}\bigl[F|\mathcal{F}^n_{i-1}\bigr]\bigr|^{2L}\Bigr]^{1/2}\\
    &\leq {\bf{E}}\Bigl[|F|^{6L}\Bigr]^{1/2}{\bf{E}}\Bigl[{\bf{E}}\bigl[|F|^{2L}|\mathcal{F}^n_{i-1}\bigr]\Bigr]^{1/2}\\
    &\leq {\bf{E}}\Bigl[|F|^{6L}\Bigr]^{1/2}{\bf{E}}\Bigl[|F|^{2L}\Bigr]^{1/2}.
\end{align*}
Similarly, we have
\begin{align*}
    {\bf{E}}\Bigl[\big|F^2{\bf{E}}\bigl[F|\mathcal{F}^n_{i-1}\bigr]^2\bigr|^L\Bigr]&\leq {\bf{E}}\Bigl[|F|^{4L}\Bigr]^{1/2}{\bf{E}}\Bigl[|F|^{4L}\Bigr]^{1/2},\\
    {\bf{E}}\Bigl[\big|F{\bf{E}}\bigl[F|\mathcal{F}^n_{i-1}\bigr]^3\bigr|^L\Bigr]&\leq 
    {\bf{E}}\Bigl[|F|^{2L}\Bigr]^{1/2}{\bf{E}}\Bigl[|F|^{6L}\Bigr]^{1/2},\\
    {\bf{E}}\Bigl[\big|{\bf{E}}\bigl[F|\mathcal{F}^n_{i-1}\bigr]\bigr|^{4L}\Bigr]&\leq {\bf{E}}\Bigl[|F|^{4L}\Bigr],\\
    {\bf{E}}\Bigl[\big|F^2{\bf{E}}\bigl[F^2|\mathcal{F}^n_{i-1}\bigr]\bigr|^L\Bigr]&\leq
    {\bf{E}}\Bigl[|F|^{4L}\Bigr]^{1/2}{\bf{E}}\Bigl[|F|^{4L}\Bigr]^{1/2},\\
    {\bf{E}}\Bigl[\big|F{\bf{E}}\bigl[F|\mathcal{F}^n_{i-1}\bigr]{\bf{E}}\bigl[F^2|\mathcal{F}^n_{i-1}\bigr]\bigr|^L\Bigr]&\leq 
    {\bf{E}}\Bigl[|F|^{2L}\Bigr]^{1/2}{\bf{E}}\Bigl[|F|^{4L}\Bigr]^{1/4}{\bf{E}}\Bigl[|F|^{8L}\Bigr]^{1/4},\\
    {\bf{E}}\Bigl[\big|{\bf{E}}\bigl[F|\mathcal{F}^n_{i-1}\bigr]^2{\bf{E}}\bigl[F^2|\mathcal{F}^n_{i-1}\bigr]\bigr|^L\Bigr]&\leq
    {\bf{E}}\Bigl[|F|^{4L}\Bigr]^{1/2}{\bf{E}}\Bigl[|F|^{4L}\Bigr]^{1/2}
\end{align*}
and
\begin{align*}
    {\bf{E}}\Bigl[\big|{\bf{E}}\bigl[F^2|\mathcal{F}^n_{i-1}\bigr]\bigr|^{2L}\Bigr]&\leq 
    {\bf{E}}\Bigl[|F|^{4L}\Bigr],
\end{align*}
which yields
\begin{align*}
    {\bf{E}}\Bigl[\big|\psi_i^{3}(F)\bigr|^L\Bigr]&\leq C_L{\bf{E}}\Bigl[|F|^{4L}\Bigr]+C_L
    {\bf{E}}\Bigl[|F|^{2L}\Bigr]^{1/2}{\bf{E}}\Bigl[|F|^{6L}\Bigr]^{1/2}\\
    &\quad+C_L{\bf{E}}\Bigl[|F|^{4L}\Bigr]^{1/2}{\bf{E}}\Bigl[|F|^{4L}\Bigr]^{1/2}\\
    &\quad+C_L{\bf{E}}\Bigl[|F|^{2L}\Bigr]^{1/2}{\bf{E}}\Bigl[|F|^{4L}\Bigr]^{1/4}{\bf{E}}\Bigl[|F|^{8L}\Bigr]^{1/4}.
\end{align*}
Therefore, from (\ref{KL}), it is shown that
\begin{align*}
    {\bf{E}}\Bigl[\big|\psi_i^{3}\bigl(K_{i,n}^{(j)}\bigr)\bigr|^L\Bigr]\leq C_Ln^{4L(1-2\rho)}
\end{align*}
for all $L\geq 1$, and
\begin{align*}
    \frac{1}{n^{L/2}}{\bf{E}}
    \Biggl[\biggl|\sum_{i=1}^n\psi_i^{3}\bigl(K_{i,n}^{(j)}\bigr)\biggr|^{L/4}\Biggr]&\leq \frac{C_Ln^{L/4-1}}{n^{L/2}}\sum_{i=1}^n{\bf{E}}\Bigl[\big|\psi_i^{3}\bigl(K_{i,n}^{(j)}\bigr)\bigr|^{L/4}\Bigr]\\
    &\leq C_Ln^{-L/4}n^{L(1-2\rho)}
\end{align*}
for all $L\geq 4$. Since $3/8-\rho\leq 0$ and
\begin{align*}
    n^{-L/4}n^{L(1-2\rho)}=n^{L(3/4-2\rho)}=n^{2L(3/8-\rho)},
\end{align*}
we have
\begin{align*}
    \sup_{n\in\mathbb{N}}n^{2L(3/8-\rho)}<\infty,
\end{align*}
which gives
\begin{align*}
    \sup_{n\in\mathbb{N}}\frac{1}{n^{L/2}}{\bf{E}}
    \Biggl[\biggl|\sum_{i=1}^n\psi_i^{3}\bigl(K_{i,n}^{(j)}\bigr)\biggr|^{L/4}\Biggr]<\infty
\end{align*}
for all $L\geq 4$. This completes the proof of (\ref{H1-3}).
\end{proof}
\begin{proof}[\textbf{Proof of Theorem \ref{Zine}}]
We need to verify the conditions [A1$^{\prime\prime}$], [A4$^{\prime}$], [A6], [B1] and [B2] in Theorem 3 (c) of Yoshida \cite{Yoshida(2011)}. First, we define the constants $\alpha$, $\rho_1$, $\rho_2$, $\beta_1$, and $\beta_2$ to ensure [A4$^{\prime}$] as follows:
\begin{align*}
    0<\beta_1\leq\tfrac{1}{4},\quad
    0<\rho_1<\min\Bigl\{1,\beta,\tfrac{2\beta_1}{1-\alpha}\Bigr\},\quad
    2\alpha<\rho_2,\quad
    \beta_2\geq\tfrac{1}{4},\quad
    1-2\beta_2-\rho_2>0,
\end{align*}
where $\beta=\alpha(1-\alpha)^{-1}$. From Lemmas \ref{Ylemma} and \ref{H1lemma}, for all $L>0$, we have
\begin{align*}
    \sup_{n\in\mathbb{N}}{\bf E}\Biggl[\biggl|\frac{1}{\sqrt{n}}
    \partial_{\theta}{\bf H}_n(\theta_0)\biggr|^{M_1}\Biggr]<\infty
\end{align*}
and
\begin{align*}
    \sup_{n\in\mathbb{N}}{\bf E}
    \Biggl[\Biggl(\sup_{\theta\in\Theta}n^{1/2-\beta_2}
    \bigl|{\bf Y}_n(\theta)-{\bf Y}(\theta)\bigr|\biggr)^{M_2}\Biggr]
    <\infty,
\end{align*}
where $M_1=L(1-\rho_1)^{-1}>0$ and $M_2=L(1-2\beta_2-\rho_2)^{-1}>0$. Hence, [A6] is satisfied. Furthermore, from Lemmas \ref{H3lemma} and \ref{H2lemma}, for all $L>0$,
\begin{align*}
    \sup_{n\in\mathbb{N}}{\bf E}\Biggl[\biggl(\sup_{\theta\in\Theta}\frac{1}{n}
    \bigl|\partial^{3}_{\theta}{\bf H}_n(\theta)\bigr|\biggr)^{M_3}\Biggr]
    <\infty
\end{align*}
and
\begin{align*}
    \sup_{n\in\mathbb{N}}{\bf E}\Biggl[\biggl|n^{\beta_1}
    \biggl(\frac{1}{n}\partial^2_{\theta}{\bf H}_{n}(\theta_0)
    +{\bf I}(\theta_0)\biggr)\biggr|^{M_4}\Biggr]<\infty,
\end{align*}
where $M_3=L(\beta-\rho_1)^{-1}>0$ and 
\begin{align*}
    M_4=L\biggl(\frac{2\beta_1}{1-\alpha}-\rho_1\biggr)^{-1}>0.
\end{align*}
Therefore, [A1$^{\prime\prime}$] is also satisfied. Finally, ${\bf [B1]}$ implies 
[B1] and [B2]
in Theorem 3 (c) of Yoshida \cite{Yoshida(2011)}. This completes the proof.
\end{proof}
\begin{lemma}\label{nHlemma}
Suppose that $[{\bf{A1}}]$-$[{\bf{A4}}]$ hold. Then, for all $L>0$,
\begin{align*}
    \sup_{n\in\mathbb{N}}{\bf{E}}_{\mathbb{X}_n}\Biggl[\sup_{\theta\in\Theta}\biggl|\frac{1}{n}{\bf{H}}_{n}(\mathbb{X}_n,\theta)\biggr|^L\Biggr]<\infty.
\end{align*}
\end{lemma}
\begin{proof}
In an analogous manner to Lemma \ref{Ylemma}, it is shown that
\begin{align}
    \sup_{n\in\mathbb{N}}{\bf{E}}_{\mathbb{X}_n}\Biggl[\sup_{\theta\in\Theta}\biggl|\frac{1}{n}{\bf{H}}_{n}(\mathbb{X}_n,\theta)-{\bf{H}}(\theta)\biggr|^L\Biggr]<\infty \label{EH}
\end{align}
for all $L>0$. Since
\begin{align*}
    \biggl|\frac{1}{n}{\bf{H}}_{n}(\mathbb{X}_n,\theta)\biggr|^L\leq C_L\biggl|\frac{1}{n}{\bf{H}}_{n}(\mathbb{X}_n,\theta)-{\bf{H}}(\theta)\biggr|^L+C_L\bigl|{\bf{H}}(\theta)\bigr|^L,
\end{align*}
we have
\begin{align*}
    \sup_{\theta\in\Theta}\biggl|\frac{1}{n}{\bf{H}}_{n}(\mathbb{X}_n,\theta)\biggr|^L&\leq C_L\sup_{\theta\in\Theta}\biggl|\frac{1}{n}{\bf{H}}_{n}(\mathbb{X}_n,\theta)-{\bf{H}}(\theta)\biggr|^L+C_L\sup_{\theta\in\Theta}\bigl|{\bf{H}}(\theta)\bigr|^L\\
    &\leq C_L\sup_{\theta\in\Theta}\biggl|\frac{1}{n}{\bf{H}}_{n}(\mathbb{X}_n,\theta)-{\bf{H}}(\theta)\biggr|^L+C_L.
\end{align*}
Hence, it follows from (\ref{EH}) that
\begin{align*}
    {\bf{E}}_{\mathbb{X}_n}\Biggl[\sup_{\theta\in\Theta}\biggl|\frac{1}{n}{\bf{H}}_{n}(\mathbb{X}_n,\theta)\biggr|^L\Biggr]&\leq C_L{\bf{E}}_{\mathbb{X}_n}\Biggl[\sup_{\theta\in\Theta}\biggl|\frac{1}{n}{\bf{H}}_{n}(\mathbb{X}_n,\theta)-{\bf{H}}(\theta)\biggr|^L\Biggr]+C_L\\
    &\leq C_L\sup_{n\in\mathbb{N}}{\bf{E}}_{\mathbb{X}_n}\Biggl[\sup_{\theta\in\Theta}\biggl|\frac{1}{n}{\bf{H}}_{n}(\mathbb{X}_n,\theta)-{\bf{H}}(\theta)\biggr|^L\Biggr]+C_L<\infty
\end{align*}
for all $L\geq 1$, which completes the proof.
\end{proof}
\begin{lemma}\label{problemma}
Suppose that $[{\bf{A1}}]$-$[{\bf{A4}}]$ hold. Then, for all $L>0$,
\begin{align*}
    \frac{1}{\sqrt{n}}\partial_{\theta}{\bf{H}}_{n}(\mathbb{X}_n,\theta_{0})\stackrel{d}{\longrightarrow}{\bf{I}}(\theta_0)^{1/2}Z_q
\end{align*}
and
\begin{align*}
    -\frac{1}{n}\partial^2_{\theta}{\bf{H}}_n(\mathbb{X}_n,\theta_0)\stackrel{p}{\longrightarrow}{\bf{I}}(\theta_0)
\end{align*}
as $n\longrightarrow\infty$.
\end{lemma}
\begin{proof}
    The results can be shown in a similar manner to the proof of Theorem 2 in Kusano and Uchida \cite{Kusano(jump)}.
\end{proof}
\begin{proof}[\textbf{Proof of Theorem \ref{QAICtheorem}}]
The following decomposition is considered:
\begin{align*}
    {\bf{E}}_{\mathbb{X}_n}\Bigl[{\bf{H}}_{n}(\mathbb{X}_n,\hat{\theta}_{n})\Bigr]-{\bf{E}}_{\mathbb{X}_n}\Bigl[{\bf{E}}_{\mathbb{Z}_n}\Bigl[{\bf{H}}_{n}(\mathbb{Z}_n,\hat{\theta}_{n})\Bigr]\Bigr]
    &=H_{1,n}+H_{2,n}+H_{3,n},
\end{align*}
where
\begin{align*}
    H_{1,n}&={\bf{E}}_{\mathbb{X}_n}\Bigl[{\bf{H}}_{n}(\mathbb{X}_n,\hat{\theta}_{n})\Bigr]-{\bf{E}}_{\mathbb{X}_n}\Bigl[{\bf{H}}_{n}(\mathbb{X}_n,\theta_0)\Bigr],\\
    H_{2,n}&={\bf{E}}_{\mathbb{X}_n}\Bigr[{\bf{H}}_{n}(\mathbb{X}_n,\theta_0)\Bigr]
    -{\bf{E}}_{\mathbb{Z}_n}\Bigl[{\bf{H}}_{n}(\mathbb{Z}_n,\theta_0)\Bigr],\\
    H_{3,n}&={\bf{E}}_{\mathbb{Z}_n}\Bigl[{\bf{H}}_{n}(\mathbb{Z}_n,\theta_0)\Bigr]-
    {\bf{E}}_{\mathbb{X}_n}\Bigl[{\bf{E}}_{\mathbb{Z}_n}\Bigl[{\bf{H}}_{n}(\mathbb{Z}_n,\hat{\theta}_{n})\Bigr]\Bigr].
\end{align*}
$\mathbb{X}_n$ and $\mathbb{Z}_n$ have the same distribution, which implies
\begin{align*}
    H_{2,n}=0.
\end{align*}
Consequently, to show 
\begin{align*}
    {\bf{E}}_{\mathbb{X}_n}\Bigl[{\bf{H}}_{n}(\mathbb{X}_n,\hat{\theta}_{n})\Bigr]-{\bf{E}}_{\mathbb{X}_n}\Bigl[{\bf{E}}_{\mathbb{Z}_n}\Bigl[{\bf{H}}_{n}(\mathbb{Z}_n,\hat{\theta}_{n})\Bigr]\Bigr]=q+o(1)
\end{align*}
as $n\longrightarrow\infty$, it is sufficient to prove
\begin{align}
    H_{1,n} &\longrightarrow \frac{q}{2} \label{H1}
\end{align}
and
\begin{align}
    H_{3,n} &\longrightarrow \frac{q}{2} \label{H3}
\end{align}
as $n\longrightarrow\infty$.\\
\ \\
Proof of (\ref{H1}).
Since $\Theta$ is  an open subset of $\mathbb{R}^q$ and $\theta_0\in\Theta$, there exists $\rho'>0$ such that
\begin{align*}
    B(\theta_0,\rho')\subset\Theta,
\end{align*}
where $B(\theta_0,\rho')$ denotes the open ball centered at $\theta_0$ with radius $\rho'$:
\begin{align*}
    B(\theta_0,\rho')=\bigl\{\theta\in\mathbb{R}^q:|\theta-\theta_0|<\rho'\bigr\}.
\end{align*}
The event $B_n$ is defined as
\begin{align*}
    B_n=\Bigl\{|\hat{\theta}_n-\theta_0|\leq 2^{-1}n^{-1/4}\rho'\Bigr\}.
\end{align*}
Note that $\hat{\theta}_n\in B(\theta_0,\rho')$ on $B_n$. From Proposition \ref{moment}, for all $L>0$, there exists $C_L>0$ such that
\begin{align*}
    {\bf{P}}\Bigl(\bigl|\sqrt{n}(\hat{\theta}_{n}-\theta_{0})\bigr|>r\Bigr)\leq \frac{C_L}{r^L}
\end{align*}
for all $r>0$ and $n\in\mathbb{N}$, which implies
\begin{align}
\begin{split}
    {\bf{P}}\bigl(B_n^c\bigr)&={\bf{P}}\Bigl(|\hat{\theta}_n-\theta_0|>2^{-1}n^{-1/4}\rho'\Bigr)\\
    &={\bf{P}}\Bigl(\bigl|\sqrt{n}(\hat{\theta}_n-\theta_0)\bigr|>2^{-1} n^{1/4}\rho'\Bigr)\leq\frac{C}{n^4}.
\end{split}\label{Bnc}
\end{align}
Using Taylor's theorem, we have
\begin{align}
    {\bf{H}}_{n}(\mathbb{X}_n,\hat{\theta}_{n})-{\bf{H}}_{n}(\mathbb{X}_n,\theta_0)
    &=H'_{1,n}+H'_{2,n}+H'_{3,n} \label{TaylorH1}
\end{align}
on $B_n$, where $\bar{\theta}_{n,\lambda}=\theta_0+\lambda(\hat{\theta}_n-\theta_0)$,
$\hat{u}_n^{(j)} = \sqrt{n} (\hat{\theta}_n^{(j)} - \theta_0^{(j)})$ and
\begin{align*}
    H'_{1,n}&=\sum_{j_1=1}^{q}\biggl(\frac{1}{\sqrt{n}}\partial_{\theta^{(j_1)}}{\bf{H}}_{n}(\mathbb{X}_n,\theta_{0})\biggr)\hat{u}^{(j_1)}_{n},\\
    H'_{2,n}&=\frac{1}{2}\sum_{j_1=1}^q\sum_{j_2=1}^q\biggl(\frac{1}{n}\partial_{\theta^{(j_1)}}\partial_{\theta^{(j_2)}}
    {\bf{H}}_{n}(\mathbb{X}_n,\theta_{0})\biggr)\hat{u}^{(j_1)}_{n}\hat{u}^{(j_2)}_{n},\\
    H'_{3,n}&=\frac{1}{2}\sum_{j_1=1}^q\sum_{j_2=1}^q\sum_{j_3=1}^q
    \biggl(\frac{1}{n\sqrt{n}}\int_0^1(1-\lambda)^2\partial_{\theta^{(j_1)}}\partial_{\theta^{(j_2)}}\partial_{\theta^{(j_3)}}
    {\bf{H}}_{n}(\mathbb{X}_n,\bar{\theta}_{n,\lambda})d\lambda\biggr)
    \hat{u}^{(j_1)}_{n}\hat{u}^{(j_2)}_{n}\hat{u}^{(j_3)}_{n}.
\end{align*}
For all $\lambda\in(0,1)$, we see
\begin{align*}
    |\bar{\theta}_{n,\lambda}-\theta_0|=\lambda|\hat{\theta}_n-\theta_0|\leq |\hat{\theta}_n-\theta_0|\leq 2^{-1}n^{-1/4}\rho'<\rho' 
\end{align*}
on $B_n$, which yields that
\begin{align*}
    \bar{\theta}_{n,\lambda}\in B(\theta_0,\rho')\subset\Theta
\end{align*}
on $B_n$. Hence, the expansion (\ref{TaylorH1}) is well defined on $B_n$. Set
\begin{align*}
    \tilde{H}'_{3,n}=\left\{
    \begin{alignedat}{2}   
    &H'_{3,n} &\quad\bigl(\mbox{on}\ B_{n}\bigr), \\
    &1 &\bigl(\mbox{on}\ B_{n}^c\bigr).
  \end{alignedat} 
  \right.
\end{align*}
The following expression is obtained:
\begin{align*}
    \Bigl\{{\bf{H}}_{n}(\mathbb{X}_n,\hat{\theta}_{n})-{\bf{H}}_{n}(\mathbb{X}_n,\theta_0)\Bigr\}{\bf{1}}_{B_n}=H'_{1,n}{\bf{1}}_{B_n}+H'_{2,n}{\bf{1}}_{B_n}+\tilde{H}'_{3,n}{\bf{1}}_{B_n}.
\end{align*}
Therefore, it is shown that
\begin{align*}
    H_{1,n}&={\bf{E}}_{\mathbb{X}_n}\Bigl[H'_{1,n}{\bf{1}}_{B_n}\Bigr]+{\bf{E}}_{\mathbb{X}_n}\Bigl[H'_{2,n}{\bf{1}}_{B_n}\Bigr]\\
    &\qquad+{\bf{E}}_{\mathbb{X}_n}\Bigl[\tilde{H}'_{3,n}{\bf{1}}_{B_n}\Bigr]+{\bf{E}}_{\mathbb{X}_n}\Bigl[\Bigl\{{\bf{H}}_{n}(\mathbb{X}_n,\hat{\theta}_{n})-{\bf{H}}_{n}(\mathbb{X}_n,\theta_0)\Bigr\}{\bf{1}}_{B_n^c}\Bigr].
\end{align*}
First, we show 
\begin{align}
    {\bf{E}}_{\mathbb{X}_n}\Bigl[H'_{1,n}{\bf{1}}_{B_n}\Bigr]\longrightarrow q \label{H1d}
\end{align}
as $n\longrightarrow\infty$. By Taylor's theorem, it holds that
\begin{align*}
    0&=\frac{1}{\sqrt{n}}\partial_{\theta^{(j_1)}}{\bf{H}}_{n}(\mathbb{X}_n,\hat{\theta}_{n})\\
    &=
    \frac{1}{\sqrt{n}}\partial_{\theta^{(j_1)}}{\bf{H}}_{n}(\mathbb{X}_n,\theta_0)+\sum_{j_2=1}^q\biggl(\frac{1}{n}\partial_{\theta^{(j_1)}}\partial_{\theta^{(j_2)}}{\bf{H}}_{n}(\mathbb{X}_n,\theta_0)\biggr)\hat{u}_n^{(j_2)}\\
    &\qquad+\sum_{j_2=1}^q\sum_{j_3=1}^q\biggl(\frac{1}{n\sqrt{n}}\int_{0}^{1}(1-\lambda)\partial_{\theta^{(j_1)}}\partial_{\theta^{(j_2)}}\partial_{\theta^{(j_3)}}{\bf{H}}_{n}(\mathbb{X}_n,\bar{\theta}_{n,\lambda})d\lambda\biggr)\hat{u}_n^{(j_2)}\hat{u}_n^{(j_3)}
\end{align*}
for $j_1=1,\ldots,q$ on $B_n$, so that 
\begin{align*}
    \frac{1}{\sqrt{n}}\partial_{\theta^{(j_1)}}{\bf{H}}_{n}(\mathbb{X}_n,\theta_0){\bf{1}}_{B_n}=
    \sum_{j_2=1}^q {\bf{I}}(\theta_0)_{j_1j_2}\hat{u}^{(j_2)}_n{\bf{1}}_{B_n}-R_{n}^{(j_1)}{\bf{1}}_{B_n},
\end{align*}
where 
\begin{align*}
    R_{n}^{(j_1)}=\left\{
    \begin{alignedat}{2}   
    \begin{split}
    &\sum_{j_2=1}^q\biggl(\frac{1}{n}\partial_{\theta^{(j_1)}}\partial_{\theta^{(j_2)}}{\bf{H}}_{n}(\mathbb{X}_n,\theta_0)+{\bf{I}}(\theta_0)_{j_1j_2}\biggr)\hat{u}^{(j_2)}_n \\
    &\quad+\sum_{j_2=1}^q\sum_{j_3=1}^q\biggl(\frac{1}{n\sqrt{n}}\int_{0}^{1}(1-\lambda)\partial_{\theta^{(j_1)}}\partial_{\theta^{(j_2)}}\partial_{\theta^{(j_3)}}{\bf{H}}_{n}(\mathbb{X}_n,\bar{\theta}_{n,\lambda})d\lambda\biggr)\hat{u}_n^{(j_2)}\hat{u}_n^{(j_3)}
    \end{split} &\quad\bigl(\mbox{on}\ B_{n}\bigr), \\
    &1
    &\quad\bigl(\mbox{on}\ B_{n}^c\bigr).
  \end{alignedat} 
  \right.
\end{align*}
Consequently, one gets
\begin{align}
\begin{split}
    H'_{1,n}{\bf{1}}_{B_n}&=\sum_{j_1=1}^{q}\biggl(\frac{1}{\sqrt{n}}\partial_{\theta^{(j_1)}}{\bf{H}}_{n}(\mathbb{X}_n,\theta_{0}){\bf{1}}_{B_n}\biggr)\hat{u}^{(j_1)}_{n}\\
    &=\sum_{j_1=1}^q\sum_{j_2=1}^q {\bf{I}}(\theta_0)_{j_1j_2}\hat{u}^{(j_1)}_n\hat{u}^{(j_2)}_n{\bf{1}}_{B_n}-\sum_{j_1=1}^q{R}_{n}^{(j_1)}{\bf{1}}_{B_n}\hat{u}^{(j_1)}_{n}\\
    &=\hat{u}_n^{\top}{\bf{I}}(\theta_0)\hat{u}_n{\bf{1}}_{B_n}-\sum_{j_1=1}^q{R}_{n}^{(j_1)}{\bf{1}}_{B_n}\hat{u}^{(j_1)}_{n}.
\end{split}\label{H1d-1}
\end{align}
Since we can evaluate $R_{n}^{(j_1)}$ by
\begin{align*}
    |R_{n}^{(j_1)}|&\leq \sum_{j_2=1}^q\biggl|\frac{1}{n}\partial_{\theta^{(j_1)}}\partial_{\theta^{(j_2)}}{\bf{H}}_{n}(\mathbb{X}_n,\theta_0)+{\bf{I}}(\theta_0)_{j_1j_2}\biggr|\bigl|\hat{u}^{(j_2)}_n\bigr|\\
    &\qquad+\frac{1}{\sqrt{n}}\sum_{j_2=1}^q\sum_{j_3=1}^q\biggl(\frac{1}{n}\int_{0}^{1}(1-\lambda)\Bigl|\partial_{\theta^{(j_1)}}\partial_{\theta^{(j_2)}}\partial_{\theta^{(j_3)}}{\bf{H}}_{n}(\mathbb{X}_n,\bar{\theta}_{n,\lambda})\Bigr|d\lambda\biggr)\bigl|\hat{u}^{(j_2)}_n\bigr|\bigl|\hat{u}^{(j_3)}_n\bigr|\\
    &\leq \frac{1}{n^{1/4}}\sum_{j_2=1}^q \biggl(n^{1/4}\biggl|\frac{1}{n}\partial_{\theta^{(j_1)}}\partial_{\theta^{(j_2)}}{\bf{H}}_{n}(\mathbb{X}_n,\theta_0)+{\bf{I}}(\theta_0)_{j_1j_2}\biggl|\, \biggr)\bigl|\hat{u}^{(j_2)}_n\bigr|\\
    &\qquad+\frac{1}{\sqrt{n}}\sum_{j_2=1}^q\sum_{j_3=1}^q
    \biggl(\frac{1}{n}\sup_{\theta\in\Theta}
    \Bigl|\partial_{\theta^{(j_1)}}\partial_{\theta^{(j_2)}}\partial_{\theta^{(j_3)}}{\bf{H}}_{n}(\mathbb{X}_n,\theta)\Bigr|\biggr)\bigl|\hat{u}^{(j_2)}_n\bigr|\bigl|\hat{u}^{(j_3)}_n\bigr|
\end{align*}
on $B_n$, we obtain
\begin{align*}
    |R_{n}^{(j_1)}|{\bf{1}}_{B_n}&\leq \frac{1}{n^{1/4}}\sum_{j_2=1}^q \biggl(n^{1/4}\biggl|\frac{1}{n}\partial_{\theta^{(j_1)}}\partial_{\theta^{(j_2)}}{\bf{H}}_{n}(\mathbb{X}_n,\theta_0)+{\bf{I}}(\theta_0)_{j_1j_2}\biggl|\, \biggr)\bigl|\hat{u}^{(j_2)}_n\bigr|{\bf{1}}_{B_n}\\
    &\qquad+\frac{1}{\sqrt{n}}\sum_{j_2=1}^q\sum_{j_3=1}^q
    \biggl(\frac{1}{n}\sup_{\theta\in\Theta}
    \Bigl|\partial_{\theta^{(j_1)}}\partial_{\theta^{(j_2)}}\partial_{\theta^{(j_3)}}{\bf{H}}_{n}(\mathbb{X}_n,\theta)\Bigr|\biggr)\bigl|\hat{u}^{(j_2)}_n\bigr|\bigl|\hat{u}^{(j_3)}_n\bigr|{\bf{1}}_{B_n}\\
    &\leq \frac{1}{n^{1/4}}\sum_{j_2=1}^q \biggl(n^{1/4}\biggl|\frac{1}{n}\partial_{\theta^{(j_1)}}\partial_{\theta^{(j_2)}}{\bf{H}}_{n}(\mathbb{X}_n,\theta_0)+{\bf{I}}(\theta_0)_{j_1j_2}\biggl|\, \biggr)\bigl|\hat{u}^{(j_2)}_n\bigr|\\
    &\qquad+\frac{1}{\sqrt{n}}\sum_{j_2=1}^q\sum_{j_3=1}^q
    \biggl(\frac{1}{n}\sup_{\theta\in\Theta}
    \Bigl|\partial_{\theta^{(j_1)}}\partial_{\theta^{(j_2)}}\partial_{\theta^{(j_3)}}{\bf{H}}_{n}(\mathbb{X}_n,\theta)\Bigr|\biggr)\bigl|\hat{u}^{(j_2)}_n\bigr|\bigl|\hat{u}^{(j_3)}_n\bigr|.
\end{align*}
Thus, by the Cauchy-Schwarz inequality, one has
\begin{align*}
    &\quad\ \Biggl|{\bf{E}}\Biggl[\sum_{j_1=1}^q{R}_{n}^{(j_1)}{\bf{1}}_{B_n}\hat{u}^{(j_1)}_{n}\Biggr]\Biggr|\\
    &\leq \sum_{j_1=1}^q
    {\bf{E}}\Bigl[|R_{n}^{(j_1)}|{\bf{1}}_{B_n}\bigl|\hat{u}^{(j_1)}_n\bigr|\Bigr]\\
    &\leq \frac{1}{n^{1/4}}\sum_{j_1=1}^q\sum_{j_2=1}^q{\bf{E}}\biggl[\biggl(n^{1/4}\biggl|\frac{1}{n}\partial_{\theta^{(j_1)}}\partial_{\theta^{(j_2)}}{\bf{H}}_{n}(\mathbb{X}_n,\theta_0)+{\bf{I}}(\theta_0)_{j_1j_2}\biggl|\, \biggr)\bigl|\hat{u}^{(j_1)}_n\bigr|\bigl|\hat{u}^{(j_2)}_n\bigr|\biggr]\\
    &\qquad+\frac{1}{\sqrt{n}}\sum_{j_1=1}^q\sum_{j_2=1}^q\sum_{j_3=1}^q{\bf{E}}\biggl[
    \biggl(\frac{1}{n}\sup_{\theta\in\Theta}
    \Bigl|\partial_{\theta^{(j_1)}}\partial_{\theta^{(j_2)}}\partial_{\theta^{(j_3)}}{\bf{H}}_{n}(\mathbb{X}_n,\theta)\Bigr|\biggr)\bigl|\hat{u}^{(j_1)}_n\bigr|\bigl|\hat{u}^{(j_2)}_n\bigr|\bigl|\hat{u}^{(j_3)}_n\bigr|\Biggr]\\
    &\leq\frac{1}{n^{1/4}}\sum_{j_1=1}^q\sum_{j_2=1}^q{\bf{E}}\Biggl[\biggl(n^{1/4}\biggl|\frac{1}{n}\partial_{\theta^{(j_1)}}\partial_{\theta^{(j_2)}}{\bf{H}}_{n}(\mathbb{X}_n,\theta_0)+{\bf{I}}(\theta_0)_{j_1j_2}\biggl|\biggr)^2\Biggr]^{1/2}\\
    &\qquad\qquad\qquad\qquad\qquad\qquad\qquad\qquad\qquad\qquad
    \times
    {\bf{E}}\Bigl[\bigl|\hat{u}^{(j_1)}_n\bigr|^4\Bigr]^{1/4}{\bf{E}}\Bigl[\bigl|\hat{u}^{(j_2)}_n\bigr|^4\Bigr]^{1/4}\\
    &\qquad+\frac{1}{\sqrt{n}}\sum_{j_1=1}^q\sum_{j_2=1}^q\sum_{j_3=1}^q{\bf{E}}\Biggl[
    \biggl(\frac{1}{n}\sup_{\theta\in\Theta}
    \Bigl|\partial_{\theta^{(j_1)}}\partial_{\theta^{(j_2)}}\partial_{\theta^{(j_3)}}{\bf{H}}_{n}(\mathbb{X}_n,\theta)\Bigr|\biggr)^4\Biggr]^{1/4}\\
    &\qquad\qquad\qquad\qquad\qquad\qquad\qquad\quad
    \times{\bf{E}}\Bigl[\bigl|\hat{u}^{(j_1)}_n\bigr|^4\Bigr]^{1/4}
    {\bf{E}}\Bigl[\bigl|\hat{u}^{(j_2)}_n\bigr|^4\Bigr]^{1/4}
    {\bf{E}}\Bigl[\bigl|\hat{u}^{(j_3)}_n\bigr|^4\Bigr]^{1/4}.
\end{align*}
It follows from Proposition \ref{moment}, Lemmas \ref{H3lemma} and \ref{H2lemma} that
\begin{align*}
    \sup_{n\in\mathbb{N}}{\bf{E}}\Biggl[\biggl(n^{1/4}\biggl|\frac{1}{n}\partial_{\theta^{(j_1)}}\partial_{\theta^{(j_2)}}{\bf{H}}_{n}(\mathbb{X}_n,\theta_0)+{\bf{I}}(\theta_0)_{j_1j_2}\biggl|\biggr)^2\Biggr]^{1/2}<\infty
\end{align*}
and 
\begin{align*}
    \sup_{n\in\mathbb{N}}{\bf{E}}\Biggl[
    \biggl(\frac{1}{n}\sup_{\theta\in\Theta}
    \Bigl|\partial_{\theta^{(j_1)}}\partial_{\theta^{(j_2)}}\partial_{\theta^{(j_3)}}{\bf{H}}_{n}(\mathbb{X}_n,\theta)\Bigr|\biggr)^4\Biggr]^{1/4}<\infty
    ,\quad
    \sup_{n\in\mathbb{N}}{\bf{E}}\Bigl[\bigl|\hat{u}^{(j_1)}_n\bigr|^4\Bigr]^{1/4}<\infty,
\end{align*}
so that 
\begin{align*}
    \Biggl|{\bf{E}}\Biggl[\sum_{j_1=1}^q{R}_{n}^{(j_1)}{\bf{1}}_{B_n}\hat{u}^{(j_1)}_{n}\Biggr]\Biggr|\leq \frac{C}{n^{1/4}}+\frac{C}{\sqrt{n}},
\end{align*}
which implies 
\begin{align}
    {\bf{E}}\Biggl[\sum_{j_1=1}^q{R}_{n}^{(j_1)}{\bf{1}}_{B_n}\hat{u}^{(j_1)}_{n}\Biggr]\longrightarrow 0 \label{H1d-2}
\end{align}
as $n\longrightarrow\infty$. By Proposition \ref{moment}, we have
\begin{align*}
    \begin{split}
     {\bf{E}}_{\mathbb{X}_n}\Bigl[\hat{u}_n^{\top}
    {\bf{I}}(\theta_0)\hat{u}_n\Bigr]
    &={\bf{E}}_{\mathbb{X}_n}\Bigl[f\bigl(\hat{u}_n\bigr)\Bigr]\\
    &\longrightarrow \mathbb{E}\biggl[f\Bigl({\bf{I}}(\theta_0)^{-1/2}Z_q\Bigr)\biggr]=\mathbb{E}\Bigl[Z_q^{\top}Z_q\Bigr]=q
    \end{split}
\end{align*}
as $n\longrightarrow\infty$, where $f(x)=x^{\top}{\bf{I}}(\theta_0)x$ for $x\in\mathbb{R}^q$. 
On the other hand, by using the Cauchy-Schwarz inequality, it holds from Proposition \ref{moment} and (\ref{Bnc}) that 
\begin{align*}
    \Bigl|{\bf{E}}_{\mathbb{X}_n}\Bigl[\hat{u}_n^{\top}
    {\bf{I}}(\theta_0)\hat{u}_n{\bf{1}}_{B_n^c}\Bigr]\Bigr|&\leq C
    {\bf{E}}_{\mathbb{X}_n}\Bigl[|\hat{u}_n|^2{\bf{1}}_{B_n^c}\Bigr]\\
    &\leq C{\bf{E}}_{\mathbb{X}_n}\Bigl[|\hat{u}_n|^4\Bigr]^{1/2}{\bf{E}}\Bigl[
    {\bf{1}}_{B_n^c}\Bigr]^{1/2}\\
    &\leq C\biggl(\sup_{n\in\mathbb{N}}{\bf{E}}_{\mathbb{X}_n}\Bigl[|\hat{u}_n|^4\Bigr]^{1/2}\biggr){\bf{P}}\bigl(B_n^c\bigr)^{1/2}\leq \frac{C}{n^2},
\end{align*}
which yields that
\begin{align*}
    {\bf{E}}_{\mathbb{X}_n}\Bigl[\hat{u}_n^{\top}
    {\bf{I}}(\theta_0)\hat{u}_n{\bf{1}}_{B_n^c}\Bigr]\longrightarrow 0
\end{align*}
as $n\longrightarrow\infty$. Consequently, we obtain
\begin{align}
\begin{split}
    &\quad\ {\bf{E}}_{\mathbb{X}_n}\Bigl[\hat{u}_n^{\top}
    {\bf{I}}(\theta_0)\hat{u}_n{\bf{1}}_{B_n}\Bigr]\\
    &={\bf{E}}_{\mathbb{X}_n}\Bigl[\hat{u}_n^{\top}
    {\bf{I}}(\theta_0)\hat{u}_n\Bigr]-{\bf{E}}_{\mathbb{X}_n}\Bigl[\hat{u}_n^{\top}
    {\bf{I}}(\theta_0)\hat{u}_n{\bf{1}}_{B_n^c}\Bigr]\longrightarrow q
    \label{H1d-3}
\end{split}
\end{align}
as $n\longrightarrow\infty$. 
Therefore, (\ref{H1d-1})-(\ref{H1d-3}) imply (\ref{H1d}). Next, we prove
\begin{align}
    {\bf{E}}_{\mathbb{X}_n}\Bigl[H'_{2,n}{\bf{1}}_{B_n}\Bigr]\longrightarrow -\frac{q}{2}\label{H2d}
\end{align}
as $n\longrightarrow\infty$. The following decomposition is considered:
\begin{align*}
    H'_{2,n}{\bf{1}}_{B_n}
    &=-\frac{1}{2}\sum_{j_1=1}^q\sum_{j_2=1}^q {\bf{I}}(\theta_0)_{j_1j_2}\hat{u}^{(j_1)}_n\hat{u}^{(j_2)}_n{\bf{1}}_{B_n}\\
    &\qquad+\frac{1}{2}\sum_{j_1=1}^q\sum_{j_2=1}^q\biggl(\frac{1}{n}\partial_{\theta^{(j_1)}}\partial_{\theta^{(j_2)}}
    {\bf{H}}_{n}(\mathbb{X}_n,\theta_{0})+{\bf{I}}(\theta_0)_{j_1j_2}\biggr)\hat{u}^{(j_1)}_n
    \hat{u}^{(j_2)}_n{\bf{1}}_{B_n}.
\end{align*}
As it follows from Proposition \ref{moment} and Lemma \ref{H2lemma} that
\begin{align*}
    &\quad\ \Biggl|{\bf{E}}_{\mathbb{X}_n}\Biggl[\biggl(\frac{1}{n}\partial_{\theta^{(j_1)}}\partial_{\theta^{(j_2)}}
    {\bf{H}}_{n}(\mathbb{X}_n,\theta_{0})+{\bf{I}}(\theta_0)_{j_1j_2}\biggr)\hat{u}^{(j_1)}_n
    \hat{u}^{(j_2)}_n{\bf{1}}_{B_n}\Biggr]\Biggr|\\
    &\leq \frac{1}{n^{1/4}}{\bf{E}}_{\mathbb{X}_n}\Biggl[\biggl|
    n^{1/4}\biggl(\frac{1}{n}\partial_{\theta^{(j_1)}}\partial_{\theta^{(j_2)}}
    {\bf{H}}_{n}(\mathbb{X}_n,\theta_{0})+{\bf{I}}(\theta_0)_{j_1j_2}\biggr)\biggr|\bigl|\hat{u}^{(j_1)}_n
    \bigr|\bigl|\hat{u}^{(j_2)}_n\bigr|\Biggr]\\
    & \leq \frac{1}{n^{1/4}} \sup_{n\in\mathbb{N}}
    {\bf{E}}_{\mathbb{X}_n}\Biggl[
    \biggl|
    n^{1/4}\biggl(\frac{1}{n}\partial_{\theta^{(j_1)}}\partial_{\theta^{(j_2)}}
    {\bf{H}}_{n}(\mathbb{X}_n,\theta_{0})+{\bf{I}}(\theta_0)_{j_1j_2}\biggr) 
    \biggr|^2\Biggr]^{1/2} \\
    &\qquad\qquad\qquad\qquad\qquad\qquad\times\sup_{n\in\mathbb{N}}{\bf{E}}_{\mathbb{X}_n}
    \Bigl[\bigl|\hat{u}_n^{(j_1)}\bigr|^{4}\Bigr]^{1/4}\sup_{n\in\mathbb{N}}{\bf{E}}_{\mathbb{X}_n}
    \Bigl[\bigl|\hat{u}_n^{(j_2)}\bigr|^{4}\Bigr]^{1/4}\leq\frac{C}{n^{1/4}}
\end{align*}
for $j_1,j_2=1,\ldots,q$, we obtain
\begin{align*}
    {\bf{E}}_{\mathbb{X}_n}\Biggl[\biggl(\frac{1}{n}\partial_{\theta^{(j_1)}}\partial_{\theta^{(j_2)}}
    {\bf{H}}_{n}(\mathbb{X}_n,\theta_{0})+{\bf{I}}(\theta_0)_{j_1j_2}\biggr)\hat{u}^{(j_1)}_n
    \hat{u}^{(j_2)}_n{\bf{1}}_{B_n}\Biggr]\longrightarrow 0
\end{align*}
as $n\longrightarrow\infty$. Thus, it holds from (\ref{H1d-3}) that
\begin{align*}
    &\quad\  {\bf{E}}_{\mathbb{X}_n}\Bigl[{H}'_{2,n}{\bf{1}}_{B_n}\Bigr]\\
    &=-\frac{1}{2}{\bf{E}}_{\mathbb{X}_n}\Bigl[\hat{u}_n^{\top}
    {\bf{I}}(\theta_0)\hat{u}_n{\bf{1}}_{B_n}\Bigr]\\
    &\quad+\frac{1}{2}\sum_{j_1=1}^q\sum_{j_2=1}^q
    {\bf{E}}_{\mathbb{X}_n}\Biggl[\biggl(\frac{1}{n}\partial_{\theta^{(j_1)}}\partial_{\theta^{(j_2)}}
    {\bf{H}}_{n}(\mathbb{X}_n,\theta_{0})+{\bf{I}}(\theta_0)_{j_1j_2}\biggr)\hat{u}^{(j_1)}_n
    \hat{u}^{(j_2)}_n{\bf{1}}_{B_n}\Biggr]\longrightarrow-\frac{q}{2}
\end{align*}
as $n\longrightarrow\infty$, which completes the proof of (\ref{H2d}). Moreover, we show
\begin{align}
    {\bf{E}}_{\mathbb{X}_n}\Bigl[\tilde{H}'_{3,n}{\bf{1}}_{B_n}\Bigr]\longrightarrow 0\label{H3d}
\end{align}
as $n\longrightarrow\infty$. On $B_n$, we have
\begin{align*}
    &\quad\ \bigl|\tilde{H}'_{3,n}\bigr|\\
    &\leq \frac{1}{\sqrt{n}}\sum_{j_1=1}^q\sum_{j_2=1}^q\sum_{j_3=1}^q\biggl(\frac{1}{n}\int_0^1(1-\lambda)^2\Bigl|\partial_{\theta^{(j_1)}}\partial_{\theta^{(j_2)}}\partial_{\theta^{(j_3)}}
    {\bf{H}}_{n}(\mathbb{X}_n,\bar{\theta}_{n,\lambda})\Bigr|d\lambda\biggr)
    \bigl|\hat{u}_n^{(j_1)}\bigr|\bigl|\hat{u}_n^{(j_2)}\bigr|\bigl|\hat{u}_n^{(j_3)}\bigr|\\
    &\leq \frac{1}{\sqrt{n}}\sum_{j_1=1}^q\sum_{j_2=1}^q\sum_{j_3=1}^q\biggl(\frac{1}{n}\sup_{\theta\in\Theta}\Bigr|
    \partial_{\theta^{(j_1)}}\partial_{\theta^{(j_2)}}\partial_{\theta^{(j_3)}}
    {\bf{H}}_{n}(\mathbb{X}_n,\theta)\Bigl|\biggr)
    \bigl|\hat{u}_n^{(j_1)}\bigr|\bigl|\hat{u}_n^{(j_2)}\bigr|\bigl|\hat{u}_n^{(j_3)}\bigr|,
\end{align*}
so that 
\begin{align*}
    \bigl|\tilde{H}'_{3,n}\bigr|{\bf{1}}_{B_n}
    &\leq\frac{1}{\sqrt{n}}\sum_{j_1=1}^q\sum_{j_2=1}^q\sum_{j_3=1}^q\biggl(\frac{1}{n}\sup_{\theta\in\Theta}\Bigr|
    \partial_{\theta^{(j_1)}}\partial_{\theta^{(j_2)}}\partial_{\theta^{(j_3)}}
    {\bf{H}}_{n}(\mathbb{X}_n,\theta)\Bigl|\biggr)
    \bigl|\hat{u}_n^{(j_1)}\bigr|\bigl|\hat{u}_n^{(j_2)}\bigr|\bigl|\hat{u}_n^{(j_3)}\bigr|.
\end{align*}
Hence, we see from Proposition \ref{moment}, Lemma \ref{H3lemma} and the Cauchy-Schwarz inequality that
\begin{align*}
    \Bigl|{\bf{E}}_{\mathbb{X}_n}\Bigl[\tilde{H}'_{3,n}{\bf{1}}_{B_n}\Bigr]\Bigr|
    &\leq{\bf{E}}_{\mathbb{X}_n}\Bigl[\bigl|\tilde{H}'_{3,n}\bigr|{\bf{1}}_{B_n}\Bigr]\\
    &\leq\frac{1}{\sqrt{n}}\sum_{j_1=1}^q\sum_{j_2=1}^q\sum_{j_3=1}^q\sup_{n\in\mathbb{N}}{\bf{E}}_{\mathbb{X}_n}\left[\biggl(\frac{1}{n}\sup_{\theta\in\Theta}\Bigr|
    \partial_{\theta^{(j_1)}}\partial_{\theta^{(j_2)}}\partial_{\theta^{(j_3)}}
    {\bf{H}}_{n}(\mathbb{X}_n,\theta)\Bigl|\biggr)^4\right]^{1/4}\\
    &\qquad
    \times\sup_{n\in\mathbb{N}}{\bf{E}}_{\mathbb{X}_n}
    \Bigl[\bigl|\hat{u}_n^{(j_1)}\bigr|^{4}\Bigr]^{1/4}\sup_{n\in\mathbb{N}}{\bf{E}}_{\mathbb{X}_n}
    \Bigl[\bigl|\hat{u}_n^{(j_2)}\bigr|^{4}\Bigr]^{1/4}\sup_{n\in\mathbb{N}}{\bf{E}}_{\mathbb{X}_n}
    \Bigl[\bigl|\hat{u}_n^{(j_3)}\bigr|^{4}\Bigr]^{1/4}\\
    &\longrightarrow 0
\end{align*}
as $n\longrightarrow\infty$, which yields (\ref{H3d}). Finally, we show
\begin{align}
    {\bf{E}}_{\mathbb{X}_n}\Bigl[\Bigl\{{\bf{H}}_{n}(\mathbb{X}_n,\hat{\theta}_{n})-{\bf{H}}_{n}(\mathbb{X}_n,\theta_0)\Bigr\}{\bf{1}}_{B_n^c}\Bigr]\longrightarrow 0 \label{H1dC}
\end{align}
as $n\longrightarrow\infty$. Note that
\begin{align*}
    &\quad\ \Bigl|{\bf{E}}_{\mathbb{X}_n}\Bigl[\Bigl\{{\bf{H}}_{n}(\mathbb{X}_n,\hat{\theta}_{n})-{\bf{H}}_{n}(\mathbb{X}_n,\theta_0)\Bigr\}{\bf{1}}_{B_n^c}\Bigr]\Bigr|\\
    &\leq n{\bf{E}}_{\mathbb{X}_n}\Biggl[\biggl|\frac{1}{n}{\bf{H}}_{n}(\mathbb{X}_n,\hat{\theta}_{n})-\frac{1}{n}{\bf{H}}_{n}(\mathbb{X}_n,\theta_0)\biggr|{\bf{1}}_{B_n^c}\Biggr]\\
    &\leq n{\bf{E}}_{\mathbb{X}_n}\Biggl[\biggl|\frac{1}{n}{\bf{H}}_{n}(\mathbb{X}_n,\hat{\theta}_{n})\biggr|{\bf{1}}_{B_n^c}\Biggr]+n{\bf{E}}_{\mathbb{X}_n}\Biggl[\biggl|\frac{1}{n}{\bf{H}}_{n}(\mathbb{X}_n,\theta_0)\biggr|{\bf{1}}_{B_n^c}\Biggr]\\
    &\leq 2n{\bf{E}}_{\mathbb{X}_n}\Biggl[\sup_{\theta\in\Theta}\biggl|\frac{1}{n}{\bf{H}}_{n}(\mathbb{X}_n,\theta)\biggr|{\bf{1}}_{B_n^c}\Biggr].
\end{align*}
Since we have
\begin{align*}
    n{\bf{E}}_{\mathbb{X}_n}\Biggl[\sup_{\theta\in\Theta}\biggl|\frac{1}{n}{\bf{H}}_{n}(\mathbb{X}_n,\theta)\biggr|{\bf{1}}_{B_n^c}\Biggr]&\leq n{\bf{E}}_{\mathbb{X}_n}\Biggl[\biggl(\sup_{\theta\in\Theta}\biggl|\frac{1}{n}{\bf{H}}_{n}(\mathbb{X}_n,\theta)\biggr|\biggr)^2\Biggr]^{1/2}{\bf{E}}\Bigl[{\bf{1}}_{B_n^c}\Bigr]^{1/2}\\
    &\leq n{\bf{E}}_{\mathbb{X}_n}\Biggl[\sup_{\theta\in\Theta}\biggl|\frac{1}{n}{\bf{H}}_{n}(\mathbb{X}_n,\theta)\biggr|^2\Biggr]^{1/2}{\bf{P}}\bigl(B_n^c\bigr)^{1/2},
\end{align*}
it follows from Lemma \ref{nHlemma} and (\ref{Bnc})  that
\begin{align*}
    {\bf{E}}_{\mathbb{X}_n}\Bigl[\Bigl\{{\bf{H}}_{n}(\mathbb{X}_n,\hat{\theta}_{n})-{\bf{H}}_{n}(\mathbb{X}_n,\theta_0)\Bigr\}{\bf{1}}_{B_n^c}\Bigr]\leq \frac{C}{n},
\end{align*}
which implies (\ref{H1dC}). Therefore,  (\ref{H1d}), (\ref{H2d}), (\ref{H3d}) and (\ref{H1dC}) lead to (\ref{H1}). This completes the proof of (\ref{H1}).\\
\ \\
Proof of (\ref{H3}). By Taylor's theorem, the decomposition is given by
\begin{align}
    {\bf{H}}_n(\mathbb{Z}_n,\hat{\theta}_n)-{\bf{H}}_n(\mathbb{Z}_n,\theta_0)
    =H''_{1,n}+H''_{2,n}+H''_{3,n} \label{TaylorH3}
\end{align}
on $B_n$, where 
\begin{align*}
    H''_{1,n}&=\sum_{j_1=1}^q\biggl(\frac{1}{\sqrt{n}}\partial_{\theta^{(j_1)}}
    {\bf{H}}_n(\mathbb{Z}_n,\theta_0)\biggr)\hat{u}^{(j_1)}_{n},\\
    H''_{2,n}&=\frac{1}{2}\sum_{j_1=1}^q\sum_{j_2=1}^q\biggl(\frac{1}{n}\partial_{\theta^{(j_1)}}\partial_{\theta^{(j_2)}}
    {\bf{H}}_n(\mathbb{Z}_n,\theta_0)\biggr)\hat{u}^{(j_1)}_{n}\hat{u}^{(j_2)}_{n},\\
    H''_{3,n}&=\frac{1}{2}\sum_{j_1=1}^q\sum_{j_2=1}^q\sum_{j_3=1}^q\biggl(\frac{1}{n\sqrt{n}}\int_0^1(1-\lambda)^2\partial_{\theta^{(j_1)}}\partial_{\theta^{(j_2)}}\partial_{\theta^{(j_3)}}{\bf{H}}_n(\mathbb{Z}_n,\bar{\theta}_{n,\lambda})d\lambda\biggr)\hat{u}^{(j_1)}_{n}\hat{u}^{(j_2)}_{n}\hat{u}^{(j_3)}_{n}.
\end{align*}
Note that the expansion (\ref{TaylorH3}) is well defined on $B_n$. Define
\begin{align*}
    \tilde{H}''_{3,n}=\left\{
    \begin{alignedat}{2}   
    &H''_{3,n} &\quad\bigl(\mbox{on}\ B_{n}\bigr), \\
    &1 &\bigl(\mbox{on}\ B_{n}^c\bigr).
  \end{alignedat} 
  \right. 
\end{align*}
Since it is shown that
\begin{align*}
    \Bigl\{{\bf{H}}_{n}(\mathbb{Z}_n,\hat{\theta}_{n})-{\bf{H}}_{n}(\mathbb{Z}_n,\theta_0)\Bigr\}{\bf{1}}_{B_n}=H''_{1,n}{\bf{1}}_{B_n}+H''_{2,n}{\bf{1}}_{B_n}+\tilde{H}''_{3,n}{\bf{1}}_{B_n},
\end{align*}
it follows that
\begin{align*}
    H_{3,n}&=-{\bf{E}}_{\mathbb{X}_n}\Bigl[{\bf{E}}_{\mathbb{Z}_n}\Bigl[H''_{1,n}\Bigr]{\bf{1}}_{B_n}\Bigr]-{\bf{E}}_{\mathbb{X}_n}\Bigl[{\bf{E}}_{\mathbb{Z}_n}\Bigl[H''_{2,n}\Bigr]{\bf{1}}_{B_n}\Bigr]\\
    &\qquad-{\bf{E}}_{\mathbb{X}_n}\Bigl[{\bf{E}}_{\mathbb{Z}_n}\Bigl[\tilde{H}''_{3,n}\Bigr]{\bf{1}}_{B_n}\Bigr]-{\bf{E}}_{\mathbb{X}_n}\Bigl[{\bf{E}}_{\mathbb{Z}_n}\Bigl[{\bf{H}}_{n}(\mathbb{Z}_n,\hat{\theta}_{n})-{\bf{H}}_{n}(\mathbb{Z}_n,\theta_0)\Bigr]{\bf{1}}_{B_n^c}\Bigr].
\end{align*}
We first prove that 
\begin{align}
    {\bf{E}}_{\mathbb{X}_n}\Bigl[{\bf{E}}_{\mathbb{Z}_n}\Bigl[H''_{1,n}\Bigr]{\bf{1}}_{B_n}\Bigr]\longrightarrow 0 \label{H1dd}
\end{align}
as $n\longrightarrow\infty$. From Lemmas \ref{H1lemma} and \ref{problemma}, we obtain
\begin{align}
    {\bf{E}}_{\mathbb{Z}_n}\left[\frac{1}{\sqrt{n}}
    \partial_{\theta}{\bf{H}}_n(\mathbb{Z}_n,\theta_0)\right]
    \longrightarrow \mathbb{E}\Bigl[{\bf{I}}(\theta_0)^{1/2}Z_q\Bigr]
    =0 \label{EH1}
\end{align}
as $n\longrightarrow\infty$. Moreover, we see from Proposition \ref{moment} and (\ref{Bnc}) that
\begin{align*}
    \Bigl|{\bf{E}}_{\mathbb{X}_n}\Bigl[\hat{u}_n^{(j_1)}{\bf{1}}_{B_n^c}\Bigr]\Bigr|&\leq 
    {\bf{E}}_{\mathbb{X}_n}\Bigl[\bigl|\hat{u}_n^{(j_1)}\bigr|{\bf{1}}_{B_n^c}\Bigr]\\
    &\leq{\bf{E}}_{\mathbb{X}_n}\Bigl[\bigl|\hat{u}_n^{(j_1)}\bigr|^2\Bigr]^{1/2}
    {\bf{E}}_{\mathbb{X}_n}\bigl[{\bf{1}}_{B_n^c}\bigr]^{1/2}\\
    &\leq\biggl(\sup_{n\in\mathbb{N}}{\bf{E}}_{\mathbb{X}_n}\Bigl[\bigl|\hat{u}_n^{(j_1)}\bigr|^2\Bigr]^{1/2}\biggr){\bf{P}}\bigl(B_n^c\bigr)^{1/2}\leq\frac{C}{n^2}
\end{align*}
for $j_1=1,\ldots,q$, so that 
\begin{align*}
    {\bf{E}}_{\mathbb{X}_n}\Bigl[\hat{u}_n^{(j_1)}{\bf{1}}_{B_n^c}\Bigr]\longrightarrow 0
\end{align*}
as $n\longrightarrow\infty$, which yields
\begin{align}
    {\bf{E}}_{\mathbb{X}_n}\Bigl[\hat{u}_n^{(j_1)}{\bf{1}}_{B_n}\Bigr]={\bf{E}}_{\mathbb{X}_n}\Bigl[\hat{u}_n^{(j_1)}\Bigr]-{\bf{E}}_{\mathbb{X}_n}\Bigl[\hat{u}_n^{(j_1)}{\bf{1}}_{B_n^c}\Bigr]\longrightarrow 0 \label{Eu1B}
\end{align}
as $n\longrightarrow\infty$. Consequently, it holds from (\ref{EH1}) and (\ref{Eu1B}) that
\begin{align*}
    &\quad\ {\bf{E}}_{\mathbb{X}_n}\Bigl[{\bf{E}}_{\mathbb{Z}_n}\Bigl[H''_{1,n}\Bigr]{\bf{1}}_{B_n}\Bigr]\\
    &=\sum_{j_1=1}^q{\bf{E}}_{\mathbb{Z}_n}\biggl[\frac{1}{\sqrt{n}}\partial_{\theta^{(j_1)}}
    {\bf{H}}_n(\mathbb{Z}_n,\theta_0)\biggr]{\bf{E}}_{\mathbb{X}_n}\Bigl[\hat{u}_n^{(j_1)}{\bf{1}}_{B_n}\Bigr]\longrightarrow 0
\end{align*}
as $n\longrightarrow\infty$, which completes the proof of (\ref{H1dd}). Next, we show
\begin{align}
    {\bf{E}}_{\mathbb{X}_n}\Bigl[{\bf{E}}_{\mathbb{Z}_n}\Bigl[H''_{2,n}\Bigr]{\bf{1}}_{B_n}\Bigr]\longrightarrow -\frac{q}{2} \label{H2dd}
\end{align}
as $n\longrightarrow\infty$. As it follows from Proposition \ref{moment} and (\ref{Bnc}) that
\begin{align*}
    \Bigl|{\bf{E}}_{\mathbb{X}_n}\Bigl[\hat{u}_n\hat{u}_n^{\top}{\bf{1}}_{B_n^c}\Bigr]\Bigr|&\leq 
    {\bf{E}}_{\mathbb{X}_n}\Bigl[\bigl|\hat{u}_n\bigr|^2{\bf{1}}_{B_n^c}\Bigr]\\
    &\leq{\bf{E}}_{\mathbb{X}_n}\Bigl[\bigl|\hat{u}_n\bigr|^4\Bigr]^{1/2}
    {\bf{E}}_{\mathbb{X}_n}\bigl[{\bf{1}}_{B_n^c}\bigr]^{1/2}\\
    &\leq\biggl(\sup_{n\in\mathbb{N}}{\bf{E}}_{\mathbb{X}_n}\Bigl[\bigl|\hat{u}_n\bigr|^4\Bigr]^{1/2}\biggr){\bf{P}}\bigl(B_n^c\bigr)^{1/2}\leq\frac{C}{n^2},
\end{align*}
one gets
\begin{align*}
    {\bf{E}}_{\mathbb{X}_n}\Bigl[\hat{u}_n\hat{u}_n^{\top}{\bf{1}}_{B_n^c}\Bigr]\longrightarrow 0
\end{align*}
as $n\longrightarrow\infty$. Moreover, by Proposition \ref{moment}, we have
\begin{align*}
    {\bf{E}}_{\mathbb{X}_n}
    \Bigl[\hat{u}_n\hat{u}_n^{\top}\Bigr]&={\bf{E}}_{\mathbb{X}_n}\Bigl[g(\hat{u}_n)\Bigr]\\
    &\longrightarrow \mathbb{E}\Bigl[g\bigl({\bf{I}}(\theta_0)^{-1/2}Z_q\bigr)\Bigr]={\bf{I}}(\theta_0)^{-1}
\end{align*}
as $n\longrightarrow\infty$, which gives
\begin{align}
    {\bf{E}}_{\mathbb{X}_n}
    \Bigl[\hat{u}_n\hat{u}_n^{\top}{\bf{1}}_{B_n}\Bigr]={\bf{E}}_{\mathbb{X}_n}
    \Bigl[\hat{u}_n\hat{u}_n^{\top}\Bigr]-{\bf{E}}_{\mathbb{X}_n}\Bigl[\hat{u}_n\hat{u}_n^{\top}{\bf{1}}_{B_n^c}\Bigr]\longrightarrow {\bf{I}}(\theta_0)^{-1} \label{uuB}
\end{align}
as $n\longrightarrow\infty$, where $g(x)=xx^{\top}$ for $x\in\mathbb{R}^q$. By Lemmas \ref{H2lemma} and \ref{problemma}, it is proven that
\begin{align}
    -{\bf{E}}_{\mathbb{Z}_n}\biggl[\frac{1}{n}\partial^2_{\theta}{\bf{H}}_n(\mathbb{Z}_n,\theta_0)\biggr]
    \longrightarrow {\bf{I}}(\theta_0) \label{EH2}
\end{align}
as $n\longrightarrow\infty$. Hence, from  (\ref{uuB}) and (\ref{EH2}), we obtain
\begin{align*}
    {\bf{E}}_{\mathbb{X}_n}\Bigl[{\bf{E}}_{\mathbb{Z}_n}\Bigl[H''_{2,n}\Bigr]{\bf{1}}_{B_n}\Bigr]
    &=\frac{1}{2}{\bf{E}}_{\mathbb{X}_n}\biggl[
    {\bf{E}}_{\mathbb{Z}_n}\biggl[\hat{u}_n^{\top}\biggl(\frac{1}{n}\partial^2_{\theta}{\bf{H}}_n(\mathbb{Z}_n,\theta_0)\biggr)\hat{u}_n\biggr]{\bf{1}}_{B_n}\biggr]\\
    &=\frac{1}{2}\tr\biggl({\bf{E}}_{\mathbb{Z}_n}\biggl[\frac{1}{n}\partial^2_{\theta}{\bf{H}}_n(\mathbb{Z}_n,\theta_0)\biggr]{\bf{E}}_{\mathbb{X}_n}
    \Bigl[\hat{u}_n\hat{u}_n^{\top}{\bf{1}}_{B_n}\Bigr]\biggr)\\
    &\longrightarrow-\frac{1}{2}\tr \Bigl({\bf{I}}(\theta_0){\bf{I}}(\theta_0)^{-1}\Bigr)=-\frac{q}{2}
\end{align*}
as $n\longrightarrow\infty$, which yields (\ref{H2dd}). Furthermore, we show 
\begin{align}
    {\bf{E}}_{\mathbb{X}_n}\Bigl[{\bf{E}}_{\mathbb{Z}_n}\Bigl[\tilde{H}''_{3,n}\Bigr]{\bf{1}}_{B_n}\Bigr]\longrightarrow 0\label{H3dd}
\end{align}
as $n\longrightarrow\infty$. Since
\begin{align*}
    &\quad\ \bigl|\tilde{H}'''_{3,n}\bigr|\\
    &\leq \frac{1}{\sqrt{n}}\sum_{j_1=1}^q\sum_{j_2=1}^q\sum_{j_3=1}^q\frac{1}{n}\int_0^1(1-\lambda)^2\bigl|\partial_{\theta^{(j_1)}}\partial_{\theta^{(j_2)}}\partial_{\theta^{(j_3)}}
    {\bf{H}}_{n}(\mathbb{Z}_n,\bar{\theta}_{n,\lambda})\bigr|d\lambda
    \bigl|\hat{u}_n^{(j_1)}\bigr|\bigl|\hat{u}_n^{(j_2)}\bigr|\bigl|\hat{u}_n^{(j_3)}\bigr|\\
    &\leq \frac{1}{\sqrt{n}}\sum_{j_1=1}^q\sum_{j_2=1}^q\sum_{j_3=1}^q\biggl(\frac{1}{n}\sup_{\theta\in\Theta}\Bigr|
    \partial_{\theta^{(j_1)}}\partial_{\theta^{(j_2)}}\partial_{\theta^{(j_3)}}
    {\bf{H}}_{n}(\mathbb{Z}_n,\theta)\Bigl|\biggr)
    \bigl|\hat{u}_n^{(j_1)}\bigr|\bigl|\hat{u}_n^{(j_2)}\bigr|\bigl|\hat{u}_n^{(j_3)}\bigr|
\end{align*}
on $B_n$, one has
\begin{align*}
    \bigl|\tilde{H}'''_{3,n}\bigr|{\bf{1}}_{B_n}
    &\leq\frac{1}{\sqrt{n}}\sum_{j_1=1}^q\sum_{j_2=1}^q\sum_{j_3=1}^q\biggl(\frac{1}{n}\sup_{\theta\in\Theta}\Bigr|
    \partial_{\theta^{(j_1)}}\partial_{\theta^{(j_2)}}\partial_{\theta^{(j_3)}}
    {\bf{H}}_{n}(\mathbb{Z}_n,\theta)\Bigl|\biggr)
    \bigl|\hat{u}_n^{(j_1)}\bigr|\bigl|\hat{u}_n^{(j_2)}\bigr|\bigl|\hat{u}_n^{(j_3)}\bigr|.
\end{align*}
By Proposition \ref{moment} and Lemma \ref{H3lemma}, we have
\begin{align*}
    &\quad\ \Bigl|{\bf{E}}_{\mathbb{X}_n}\Bigl[{\bf{E}}_{\mathbb{Z}_n}\Bigl[\tilde{H}'''_{3,n}\Bigr]{\bf{1}}_{B_n}\Bigr]\Bigr|\\
    &\leq \frac{1}{\sqrt{n}}\sum_{j_1=1}^q\sum_{j_2=1}^q\sum_{j_3=1}^q\sup_{n\in\mathbb{N}}{\bf{E}}_{\mathbb{Z}_n}
    \biggl[\frac{1}{n}\sup_{\theta\in\Theta}\Bigr|
    \partial_{\theta^{(j_1)}}\partial_{\theta^{(j_2)}}\partial_{\theta^{(j_3)}}
    {\bf{H}}_{n}(\mathbb{Z}_n,\theta)\Bigl|\biggr]\\
    &\qquad\qquad
    \times\sup_{n\in\mathbb{N}}{\bf{E}}_{\mathbb{X}_n}
    \Bigl[\bigl|\hat{u}_n^{(j_1)}\bigr|^{2}\Bigr]^{1/2}\sup_{n\in\mathbb{N}}{\bf{E}}_{\mathbb{X}_n}
    \Bigl[\bigl|\hat{u}_n^{(j_2)}\bigr|^{4}\Bigr]^{1/4}\sup_{n\in\mathbb{N}}{\bf{E}}_{\mathbb{X}_n}
    \Bigl[\bigl|\hat{u}_n^{(j_3)}\bigr|^{4}\Bigr]^{1/4}\longrightarrow 0
\end{align*}
as $n\longrightarrow\infty$, so that (\ref{H3dd}) holds. Finally, we show
\begin{align}
    {\bf{E}}_{\mathbb{X}_n}\Bigl[{\bf{E}}_{\mathbb{Z}_n}\Bigl[{\bf{H}}_{n}(\mathbb{Z}_n,\hat{\theta}_{n})-{\bf{H}}_{n}(\mathbb{Z}_n,\theta_0)\Bigr]{\bf{1}}_{B_n^c}\Bigr]\longrightarrow 0 \label{TaylorH2}
\end{align}
as $n\longrightarrow\infty$. Since
\begin{align*}
    &\quad\ \Bigl|{\bf{E}}_{\mathbb{X}_n}\Bigl[{\bf{E}}_{\mathbb{Z}_n}
    \Bigl[{\bf{H}}_{n}(\mathbb{Z}_n,\hat{\theta}_{n})-{\bf{H}}_{n}(\mathbb{Z}_n,\theta_0)\Bigr]{\bf{1}}_{B_n^c}\Bigr]\Bigr|\\
    &\leq n{\bf{E}}_{\mathbb{X}_n}\Biggl[{\bf{E}}_{\mathbb{Z}_n}\Biggl[\biggl|\frac{1}{n}{\bf{H}}_{n}(\mathbb{Z}_n,\hat{\theta}_{n})-\frac{1}{n}{\bf{H}}_{n}(\mathbb{Z}_n,\theta_0)\biggr|\Biggr]{\bf{1}}_{B_n^c}\Biggr]\\
    &\leq n{\bf{E}}_{\mathbb{X}_n}\Biggl[{\bf{E}}_{\mathbb{Z}_n}\Biggl[\biggl|\frac{1}{n}{\bf{H}}_{n}(\mathbb{Z}_n,\hat{\theta}_{n})\biggr|\Biggr]{\bf{1}}_{B_n^c}\Biggr]+n{\bf{E}}_{\mathbb{X}_n}\Biggl[{\bf{E}}_{\mathbb{Z}_n}\Biggl[\biggl|\frac{1}{n}{\bf{H}}_{n}(\mathbb{Z}_n,\theta_0)\biggr|\Biggr]{\bf{1}}_{B_n^c}\Biggr]\\
    &\leq 2n{\bf{E}}_{\mathbb{Z}_n}\Biggl[\sup_{\theta\in\Theta}\biggl|\frac{1}{n}{\bf{H}}_{n}(\mathbb{Z}_n,\theta)\biggr|\Biggr]{\bf{E}}_{\mathbb{X}_n}\Bigl[{\bf{1}}_{B_n^c}\Bigr]\\
    &\leq 2n\Biggl(\sup_{n\in\mathbb{N}}{\bf{E}}_{\mathbb{Z}_n}\Biggl[\sup_{\theta\in\Theta}\biggl|\frac{1}{n}{\bf{H}}_{n}(\mathbb{Z}_n,\theta)\biggr|\Biggr]\Biggr){\bf{P}}\bigl(B_n^c\bigr),
\end{align*}
it follows from Lemma \ref{nHlemma} and (\ref{Bnc}) that
\begin{align*}
    \Bigl|{\bf{E}}_{\mathbb{X}_n}\Bigl[{\bf{E}}_{\mathbb{Z}_n}
    \Bigl[{\bf{H}}_{n}(\mathbb{Z}_n,\hat{\theta}_{n})-{\bf{H}}_{n}(\mathbb{Z}_n,\theta_0)\Bigr]{\bf{1}}_{B_n^c}\Bigr]\Bigr|\leq\frac{C}{n^3},
\end{align*}
which implies (\ref{TaylorH2}). Therefore, by (\ref{H1dd}), (\ref{H2dd}), (\ref{H3dd}) and (\ref{TaylorH2}), we obtain (\ref{H3}). 
\end{proof}
\begin{lemma}\label{optim}
Suppose that {\bf{[A1]}}-{\bf{[A4]}} and {\bf{[B2]}}. Then,
\begin{align*}
    \frac{1}{n}{\bf{H}}_{m,n}(\hat{\theta}_{m,n})\stackrel{p}{\longrightarrow}
    {\bf{H}}_m(\bar{\theta}_{m,0})
\end{align*}
for $m=1,\ldots,M$ as $n\longrightarrow\infty$.
\end{lemma}
\begin{proof}
The results can be shown in a similar way to the proof of Theorem 4 in Kusano and Uchida \cite{Kusano(jump)}.
\end{proof}
\begin{proof}[\textbf{Proof of Proposition \ref{miss}}]
Using Lemma \ref{optim}, we can prove this result in an analogous manner to Theorem 3 in Kusano and Uchida \cite{Kusano(AIC)}.
\end{proof}

\section{Appendix}
\subsection{Proof of Lemma \ref{EX2abslemma}}\label{proofEX2abslemma}
Let 
\begin{align*}
    C_{i,j,0}^n&=\bigl\{J_{i,j}^n=0,\ |\Delta_{i}^n X|\leq Dh_n^{\rho}\bigl\},\quad C_{i,j,1}^n=\bigl\{J_{i,j}^n=1,\ |\Delta_{i}^n X|\leq Dh_n^{\rho}\bigl\},\\
    C_{i,j,2}^n&=\bigl\{J_{i,j}^n\geq 2,\ |\Delta_{i}^n X|\leq Dh_n^{\rho}\bigl\}, \quad 
    D_{i,j,0}^n=\bigl\{J_{i,j}^n=0,\ |\Delta_{i}^n X|> Dh_n^{\rho}\bigl\},\\
    D_{i,j,1}^n&=\bigl\{J_{i,j}^n=1,\ |\Delta_{i}^n X|> Dh_n^{\rho}\bigl\},\quad
    D_{i,j,2}^n=\bigl\{J_{i,j}^n\geq 2,\ |\Delta_{i}^n X|> Dh_n^{\rho}\bigl\}
\end{align*}
for $i=1,\ldots,n$ and $j=1,2,3,4$, where $J_{i,j}^n=p_j((t_{i-1}^n,t_i^n]\times E_j)$. Define
\begin{align*}
    C_{i,k_1,k_2,k_3,k_4}^n=C_{i,1,k_1}^n\cap C_{i,2,k_2}^n\cap C_{i,3,k_3}^n \cap C_{i,4,k_4}^n
\end{align*}
and
\begin{align*}
    D_{i,k_1,k_2,k_3,k_4}^n=D_{i,1,k_1}^n\cap D_{i,2,k_2}^n\cap D_{i,3,k_3}^n \cap D_{i,4,k_4}^n
\end{align*}
for $i=1,\ldots,n$ and $k_1,k_2,k_3,k_4=0,1,2$. We note that
\begin{align*}
    \bigl\{|\Delta_{i}^n X\bigr|\leq Dh_n^{\rho}\bigr\}&=
    \bigcup_{k_1,k_2,k_3,k_4=0,1,2}C_{i,k_1,k_2,k_3,k_4}^n
\end{align*}
and
\begin{align*}
    \bigl\{|\Delta_{i}^n X\bigr|> Dh_n^{\rho}\bigr\}&=
    \bigcup_{k_1,k_2,k_3,k_4=0,1,2}
    D_{i,k_1,k_2,k_3,k_4}^n
\end{align*}
for $i=1,\ldots,n$. Set
\begin{align*}
    K_1&=\Bigl\{(k_1,k_2,k_3,k_4)\in\{0,1,2\}^4;\ 
    \mbox{One element of}\ k_1, k_2, k_3, \mbox{and}\  k_4\ \mbox{is}\ 1,\ \mbox{and the others are 0}\Bigr\},\\
    K_2&=\Bigl\{(k_1,k_2,k_3,k_4)\in\{0,1,2\}^4;\ 
    \mbox{Two elements of}\ k_1, k_2, k_3, \mbox{and}\  k_4\ \mbox{are}\ 1,\ \mbox{and the others are 0}\Bigr\},\\
    K_3&=\Bigl\{(k_1,k_2,k_3,k_4)\in\{0,1,2\}^4;\ 
    \mbox{Three elements of}\ k_1, k_2, k_3, \mbox{and}\  k_4\ \mbox{are}\ 1,\ \mbox{and the other is 0} \Bigr\}
\end{align*}
and
\begin{align*}
    K_4&=\Bigl\{(k_1,k_2,k_3,k_4)\in\{0,1,2\}^4;\ 
    \mbox{At least one of}\ k_1, k_2, k_3, \mbox{and}\  k_4\  \mbox{is}\ 2\Bigr\}.
\end{align*}
\begin{lemma}[Lemma 2 in Kusano and Uchida \cite{Kusano(jump)}]\label{Clemma}
Suppose that $[{\bf{A1}}]$-$[{\bf{A4}}]$ hold. Then, for a sufficiently large $n$, 
\begin{align*}
    &{\bf{P}}\Bigl(C^{n}_{i,0,0,0,0}\big|\mathcal{F}_{i-1}^n\Bigr)=\bar{R}_{i-1}(h_n,\xi,\delta,\varepsilon,\zeta),\\
    &{\bf{P}}\Bigl(C^{n}_{i,1,1,1,1}\big|\mathcal{F}_{i-1}^n\Bigr)\leq\lambda_{1,0}\lambda_{2,0}\lambda_{3,0}\lambda_{4,0} h_n^4,\\
    &{\bf{P}}\Bigl(C^{n}_{i,k_1,k_2,k_3,k_4}\big|\mathcal{F}_{i-1}^n\Bigr)=R_{i-1}(h_n^{\rho+1},\xi,\delta,\varepsilon,\zeta)\quad (k_1,k_2,k_3,k_4)\in K_1,\\
    &{\bf{P}}\Bigl(C^{n}_{i,k_1,k_2,k_3,k_4}\big|\mathcal{F}_{i-1}^n\Bigr)\leq \Bigl(\max_{j=1,2,3,4}\lambda_{j,0}\Bigr)^2 h_n^2\quad (k_1,k_2,k_3,k_4)\in K_2,\\
    & {\bf{P}}\Bigl(C^{n}_{i,k_1,k_2,k_3,k_4}\big|\mathcal{F}_{i-1}^n\Bigr)\leq 
    \Bigl(\max_{j=1,2,3,4}\lambda_{j,0}\Bigr)^3 h_n^3\quad (k_1,k_2,k_3,k_4)\in K_3
\end{align*}
and
\begin{align*}
    {\bf{P}}\Bigl(C^{n}_{i,k_1,k_2,k_3,k_4}\big|\mathcal{F}_{i-1}^n\Bigr)\leq \Bigl(\max_{j=1,2,3,4}\lambda_{j,0}^2\Bigr)h_n^2
    \quad (k_1,k_2,k_3,k_4)\in K_4\qquad 
\end{align*}
for $i=1,\ldots,n$.
\end{lemma}
\begin{lemma}[Lemma 3 in Kusano and Uchida \cite{Kusano(jump)}]\label{Dlemma}
Suppose that $[{\bf{A1}}]$-$[{\bf{A4}}]$ hold. Then, for all $p\geq 1$ and a sufficiently large $n$,
\begin{align*}
    &{\bf{P}}\Bigl(D^{n}_{i,0,0,0,0}\big|\mathcal{F}_{i-1}^n\Bigr)=R_{i-1}(h_n^p,\xi,\delta,\varepsilon,\zeta),\\
    &{\bf{P}}\Bigl(D^{n}_{i,1,1,1,1}\big|\mathcal{F}_{i-1}^n\Bigr)\leq\lambda_{1,0}\lambda_{2,0}\lambda_{3,0}\lambda_{4,0} h_n^4,\\
    &{\bf{P}}\Bigl(D^{n}_{i,1,0,0,0}\big|\mathcal{F}_{i-1}^n\Bigr)=\lambda_{1,0}h_n\bar{R}_{i-1}(h_n^{\rho},\xi,\delta,\varepsilon,\zeta),\\
    &{\bf{P}}\Bigl(D^{n}_{i,0,1,0,0}\big|\mathcal{F}_{i-1}^n\Bigr)=\lambda_{2,0}h_n\bar{R}_{i-1}(h_n^{\rho},\xi,\delta,\varepsilon,\zeta),\\
    &{\bf{P}}\Bigl(D^{n}_{i,0,0,1,0}\big|\mathcal{F}_{i-1}^n\Bigr)=\lambda_{3,0}h_n\bar{R}_{i-1}(h_n^{\rho},\xi,\delta,\varepsilon,\zeta),\\
    &{\bf{P}}\Bigl(D^{n}_{i,0,0,0,1}\big|\mathcal{F}_{i-1}^n\Bigr)=\lambda_{4,0}h_n\bar{R}_{i-1}(h_n^{\rho},\xi,\delta,\varepsilon,\zeta),\\
    &{\bf{P}}\Bigl(D^{n}_{i,k_1,k_2,k_3,k_4}\big|\mathcal{F}_{i-1}^n\Bigr)\leq \Bigl(\max_{j=1,2,3,4}\lambda_{j,0}\Bigr)^2 h_n^2\quad (k_1,k_2,k_3,k_4)\in K_2,\\
    & {\bf{P}}\Bigl(D^{n}_{i,k_1,k_2,k_3,k_4}\big|\mathcal{F}_{i-1}^n\Bigr)\leq 
    \Bigl(\max_{j=1,2,3,4}\lambda_{j,0}\Bigr)^3 h_n^3\quad (k_1,k_2,k_3,k_4)\in K_3  
\end{align*}
and
\begin{align*}
    &{\bf{P}}\Bigl(D^{n}_{i,k_1,k_2,k_3,k_4}\big|\mathcal{F}_{i-1}^n\Bigr)\leq \Bigl(\max_{j=1,2,3,4}\lambda_{j,0}^2\Bigr)h_n^2\quad (k_1,k_2,k_3,k_4)\in K_4 
\end{align*}
for $i=1,\ldots,n$.
\end{lemma}
\begin{proof}[\textbf{Proof of Lemma \ref{EX2abslemma}}]
The stochastic processes $\{\xi^c_{0,t}\}_{t\geq 0}$, $\{\delta^c_{0,t}\}_{t\geq 0}$, $\{\varepsilon^c_{0,t}\}_{t\geq 0}$ and $\{\zeta^c_{0,t}\}_{t\geq 0}$ are defined by 
\begin{align*}
    d\xi^c_{0,t}&=a_{1}(\xi^c_{0,t})dt+{\bf{S}}_{1,0}d W_{1,t},\quad \xi_{0,0}^c=x_{1,0},\\
    d\delta^c_{0,t}&=a_{2}(\delta^c_{0,t})dt+{\bf{S}}_{2,0}d W_{2,t},\quad \delta_{0,0}^c=x_{2,0},\\
    d\varepsilon^c_{0,t}&=a_{3}(\varepsilon^c_{0,t})dt+{\bf{S}}_{3,0}d W_{3,t},\quad \varepsilon_{0,0}^c=x_{3,0},\\
    d\zeta^c_{0,t}&=a_{4}(\zeta^c_{0,t})dt+{\bf{S}}_{4,0}d W_{4,t},\quad \zeta_{0,0}^c=x_{4,0}
\end{align*}
independent of $J_{i,j}^n$. Set 
\begin{align*}
    X^c_{1,0,t}={\bf{\Lambda}}_{1,0}\xi^c_{0,t}+\delta^c_{0,t}
\end{align*}
and
\begin{align*}
    X^c_{2,0,t}&={\bf{\Lambda}}_{2,0}{\bf{\Psi}}_0^{-1}{\bf{\Gamma}}_0\xi^c_{0,t}+{\bf{\Lambda}}_{2,0}{\bf{\Psi}}_0^{-1}\zeta^c_{0,t}
    +\varepsilon^c_{0,t}.
\end{align*}
By using Lemma 20 in Kusano and Uchida \cite{Kusano(2023)}, it is proven that
\begin{align*}
\begin{split}
    &\quad\ {\bf{E}}\biggl[\bigl|\Delta_{i}^n X^c\bigr|^2\big|\mathcal{F}_{i-1}^n\biggr]=R_{i-1}(h_n,\xi,\delta,\varepsilon,\zeta),
\end{split}
\end{align*}
where $X_t^c=(X_{1,0,t}^{c\top},X_{2,0,t}^{c\top})^{\top}$, which yields
\begin{align*}
    {\bf{E}}\Bigl[\bigl|(\Delta_{i}^n X^c)^{(j_1)}\bigr|\bigl|(\Delta_{i}^n X^c)^{(j_2)}\bigr|\big|\mathcal{F}_{i-1}^n\biggr]=R_{i-1}(h_n,\xi,\delta,\varepsilon,\zeta)
\end{align*}
for $j_1,j_2=1,\ldots,p$. As it holds from the independence of $\{\xi_{0,t}\}_{t\geq 0}$, $\{\delta_{0,t}\}_{t\geq 0}$, $\{\varepsilon_{0,t}\}_{t\geq 0}$ and $\{\zeta_{0,t}\}_{t\geq 0}$ that
\begin{align*}
    &\quad\ {\bf{P}}\Bigl(J_{i,1}^n=0, J_{i,2}^n=0, J_{i,3}^n=0, J_{i,4}^n=0
    \big|\mathcal{F}_{i-1}^n\Bigr)\\
    &={\bf{P}}\Bigl(J_{i,1}^n=0 \big|\mathcal{F}_{i-1}^n\Bigr){\bf{P}}\Bigl(J_{i,2}^n=0 \big|\mathcal{F}_{i-1}^n\Bigr){\bf{P}}\Bigl(J_{i,3}^n=0 \big|\mathcal{F}_{i-1}^n\Bigr){\bf{P}}\Bigl(J_{i,4}^n=0 \big|\mathcal{F}_{i-1}^n\Bigr)\\
    &=e^{-(\lambda_{1,0}+\lambda_{2,0}+\lambda_{3,0}+\lambda_{4,0})h_n},
\end{align*}
we see
\begin{align}
\begin{split}
    &\quad\ {\bf{E}}\biggl[\bigl|(\Delta_{i}^n X)^{(j_1)}\bigr|\bigl|(\Delta_{i}^n X)^{(j_2)}\bigr|{\bf{1}}_{\{J_{i,1}^n=0, J_{i,2}^n=0, J_{i,3}^n=0, J_{i,4}^n=0\}}\big|\mathcal{F}_{i-1}^n\biggr]\\
    &={\bf{E}}\biggl[\bigl|(\Delta_{i}^n X^c)^{(j_1)}\bigr|\bigl|(\Delta_{i}^n X^c)^{(j_2)}\bigr|{\bf{1}}_{\{J_{i,1}^n=0, J_{i,2}^n=0, J_{i,3}^n=0, J_{i,4}^n=0\}}\big|\mathcal{F}_{i-1}^n\biggr]\\
    &={\bf{E}}\Bigl[\bigl|(\Delta_{i}^n X^c)^{(j_1)}\bigr|\bigl|(\Delta_{i}^n X^c)^{(j_2)}\bigr|\big|\mathcal{F}_{i-1}^n\biggr]{\bf{P}}\Bigl(J_{i,1}^n=0, J_{i,2}^n=0, J_{i,3}^n=0, J_{i,4}^n=0
    \big|\mathcal{F}_{i-1}^n\Bigr)\\
    &=R_{i-1}(h_n,\xi,\delta,\varepsilon,\zeta).
\end{split}\label{EJ1}
\end{align}
Since it is shown that
\begin{align*}
    {\bf{E}}\Bigl[|\Delta_i^n X|^L\big|\mathcal{F}_{i-1}^n\Bigr]=R_{i-1}(h_n,\xi,\delta,\varepsilon,\zeta)
\end{align*}
for all $L\geq 2$ in an analogous manner to Proposition 3.1 in Shimizu and Yoshida \cite{Shimizu(2006)}, by the Cauchy-Schwartz inequality and Lemma \ref{Dlemma}, one gets
\begin{align}
\begin{split}
    &\quad\ {\bf{E}}\biggl[\bigl|(\Delta_{i}^n X)^{(j_1)}\bigr|\bigl|(\Delta_{i}^n X)^{(j_2)}\bigr|{\bf{1}}_{D^{n}_{i,0,0,0,0}}\big|\mathcal{F}_{i-1}^n\biggr]\\
    &\leq {\bf{E}}\biggl[\bigl|(\Delta_{i}^n X)^{(j_1)}\bigr|^2\bigl|(\Delta_{i}^n X)^{(j_2)}\bigr|^2\big|\mathcal{F}_{i-1}^n\biggr]^{1/2}{\bf{P}}\Bigl(D^{n}_{i,0,0,0,0}\big|\mathcal{F}_{i-1}^n\Bigr)^{1/2}\\
    &\leq {\bf{E}}\Bigl[|\Delta_i^n X|^4\big|\mathcal{F}_{i-1}^n\Bigr]^{1/2}{\bf{P}}\Bigl(D^{n}_{i,0,0,0,0}\big|\mathcal{F}_{i-1}^n\Bigr)^{1/2}\leq R_{i-1}(h_n^{p/2+1/2},\xi,\delta,\varepsilon,\zeta)
\end{split}\label{EJ2}
\end{align}
for all $p\geq 1$. By noting that
\begin{align*}
    \bigl\{J_{i,1}^n=0, J_{i,2}^n=0, J_{i,3}^n=0, J_{i,4}^n=0\bigr\}=C^{n}_{i,0,0,0,0}\cup
    D^{n}_{i,0,0,0,0}
\end{align*}
and $C^{n}_{i,0,0,0,0}\cap D^{n}_{i,0,0,0,0}=\phi$, we have
\begin{align*}
   {\bf{1}}_{\{J_{i,1}^n=0, J_{i,2}^n=0, J_{i,3}^n=0, J_{i,4}^n=0\}}
    &={\bf{1}}_{C^{n}_{i,0,0,0,0}}+{\bf{1}}_{D^{n}_{i,0,0,0,0}}.
\end{align*}
Hence, it follows from (\ref{EJ1}) and (\ref{EJ2}) that
\begin{align}
\begin{split}
    &\quad\ {\bf{E}}\biggl[\bigl|(\Delta_{i}^n X)^{(j_1)}\bigr|\bigl|(\Delta_{i}^n X)^{(j_2)}\bigr|{\bf{1}}_{C^{n}_{i,0,0,0,0}}\big|\mathcal{F}_{i-1}^n\biggr]\\
    &={\bf{E}}\biggl[\bigl|(\Delta_{i}^n X)^{(j_1)}\bigr|\bigl|(\Delta_{i}^n X)^{(j_2)}\bigr|{\bf{1}}_{\{J_{i,1}^n=0, J_{i,2}^n=0, J_{i,3}^n=0, J_{i,4}^n=0\}}\big|\mathcal{F}_{i-1}^n\biggr]\\
    &\qquad\qquad-{\bf{E}}\biggl[\bigl|(\Delta_{i}^n X)^{(j_1)}\bigr|\bigl|(\Delta_{i}^n X)^{(j_2)}\bigr|{\bf{1}}_{D^{n}_{i,0,0,0,0}}\big|\mathcal{F}_{i-1}^n\biggr]=R_{i-1}(h_n,\xi,\delta,\varepsilon,\zeta).
\end{split}\label{EXX-1}
\end{align}
Furthermore, we see from Lemma \ref{Clemma} that
\begin{align*}
    {\bf{E}}\Bigl[\bigl|(\Delta_{i}^n X)^{(j_1)}\bigr|\bigl|(\Delta_{i}^n X)^{(j_2)}\bigr|{\bf{1}}_{C^{n}_{i,1,1,1,1}}\big|\mathcal{F}^n_{i-1}\Bigr]
    &\leq {\bf{E}}\Bigl[|\Delta_i^n X|^2{\bf{1}}_{C^{n}_{i,1,1,1,1}}\big|\mathcal{F}^n_{i-1}\Bigr]\\
    &\leq D^2h_n^{2\rho}{\bf{P}}\Bigl(C^{n}_{i,1,1,1,1}\big|\mathcal{F}_{i-1}^n\Bigr)\\
    &\leq D^2\lambda_{1,0}\lambda_{2,0}
    \lambda_{3,0}\lambda_{4,0}h_n^{2\rho+4},
\end{align*}
which implies
\begin{align}
    {\bf{E}}\Bigl[\bigl|(\Delta_{i}^n X)^{(j_1)}\bigr|\bigl|(\Delta_{i}^n X)^{(j_2)}\bigr|{\bf{1}}_{C^{n}_{i,1,1,1,1}}\big|\mathcal{F}^n_{i-1}\Bigr]=R_{i-1}(h_n^{2\rho+4},\xi,\delta,\varepsilon,\zeta). \label{EXX-2}
\end{align}
In an analogous manner, for a sufficiently large $n$, it holds from Lemma \ref{Clemma} that 
\begin{align}
    {\bf{E}}\Bigl[\bigl|(\Delta_{i}^n X)^{(j_1)}\bigr|\bigl|(\Delta_{i}^n X)^{(j_2)}\bigr|{\bf{1}}_{C^{n}_{i,k_1,k_2,k_3,k_4}}\big|\mathcal{F}^n_{i-1}\Bigr]=R_{i-1}(h_n^{3\rho+1},\xi,\delta,\varepsilon,\zeta) \label{EXX-3}
\end{align}
for $(k_1,k_2,k_3,k_4)\in K_1$, 
\begin{align}
    {\bf{E}}\Bigl[\bigl|(\Delta_{i}^n X)^{(j_1)}\bigr|\bigl|(\Delta_{i}^n X)^{(j_2)}\bigr|{\bf{1}}_{C^{n}_{i,k_1,k_2,k_3,k_4}}\big|\mathcal{F}^n_{i-1}\Bigr]=R_{i-1}(h_n^{2\rho+2},\xi,\delta,\varepsilon,\zeta) \label{EXX-4}
\end{align}
for $(k_1,k_2,k_3,k_4)\in K_2$, 
\begin{align}
    {\bf{E}}\Bigl[\bigl|(\Delta_{i}^n X)^{(j_1)}\bigr|\bigl|(\Delta_{i}^n X)^{(j_2)}\bigr|{\bf{1}}_{C^{n}_{i,k_1,k_2,k_3,k_4}}\big|\mathcal{F}^n_{i-1}\Bigr]=R_{i-1}(h_n^{2\rho+3},\xi,\delta,\varepsilon,\zeta) \label{EXX-5}
\end{align}
for $(k_1,k_2,k_3,k_4)\in K_3$, and
\begin{align}
    {\bf{E}}\Bigl[\bigl|(\Delta_{i}^n X)^{(j_1)}\bigr|\bigl|(\Delta_{i}^n X)^{(j_2)}\bigr|{\bf{1}}_{C^{n}_{i,k_1,k_2,k_3,k_4}}\big|\mathcal{F}^n_{i-1}\Bigr]=R_{i-1}(h_n^{2\rho+2},\xi,\delta,\varepsilon,\zeta) \label{EXX-6}
\end{align}
for $(k_1,k_2,k_3,k_4)\in K_4$. Therefore, from (\ref{EXX-1})-(\ref{EXX-6}), one has
\begin{align*}
     &\quad\ {\bf{E}}\Bigl[\bigl|(\Delta_{i}^n X)^{(j_1)}\bigr|\bigl|(\Delta_{i}^n X)^{(j_2)}\bigr|{\bf{1}}_{\{|\Delta_i^n X|\leq Dh_n^{\rho}\}}\big|\mathcal{F}_{i-1}^n\biggr]\\
    &={\bf{E}}\biggl[\bigl|(\Delta_{i}^n X)^{(j_1)}\bigr|\bigl|(\Delta_{i}^n X)^{(j_2)}\bigr|{\bf{1}}_{C^{n}_{i,0,0,0,0}}\big|\mathcal{F}^n_{i-1}\Bigr]\\
    &\quad+{\bf{E}}\biggl[\bigl|(\Delta_{i}^n X)^{(j_1)}\bigr|\bigl|(\Delta_{i}^n X)^{(j_2)}\bigr|{\bf{1}}_{C^{n}_{i,1,1,1,1}}\big|\mathcal{F}^n_{i-1}\biggr]\\
    &\quad+\sum_{(k_1,k_2,k_3,k_4)\in K_1} {\bf{E}}\Bigl[\bigl|(\Delta_{i}^n X)^{(j_1)}\bigr|\bigl|(\Delta_{i}^n X)^{(j_2)}\bigr|
    {\bf{1}}_{C^{n}_{i,k_1,k_2,k_3,k_4}}\big|\mathcal{F}^n_{i-1}\Bigr]\\
    &\quad+\sum_{(k_1,k_2,k_3,k_4)\in K_2} {\bf{E}}\Bigl[
    \bigl|(\Delta_{i}^n X)^{(j_1)}\bigr|\bigl|(\Delta_{i}^n X)^{(j_2)}\bigr|
    {\bf{1}}_{C^{n}_{i,k_1,k_2,k_3,k_4}}\big|\mathcal{F}^n_{i-1}\Bigr]\\
    &\quad+\sum_{(k_1,k_2,k_3,k_4)\in K_3} {\bf{E}}\Bigl[
    \bigl|(\Delta_{i}^n X)^{(j_1)}\bigr|\bigl|(\Delta_{i}^n X)^{(j_2)}\bigr|
    {\bf{1}}_{C^{n}_{i,k_1,k_2,k_3,k_4}}\big|\mathcal{F}^n_{i-1}\Bigr]\\
    &\quad+\sum_{(k_1,k_2,k_3,k_4)\in K_4} {\bf{E}}\Bigl[
    \bigl|(\Delta_{i}^n X)^{(j_1)}\bigr|\bigl|(\Delta_{i}^n X)^{(j_2)}\bigr|
    {\bf{1}}_{C^{n}_{i,k_1,k_2,k_3,k_4}}\big|\mathcal{F}^n_{i-1}\Bigr]=R_{i-1}(h_n,\xi,\delta,\varepsilon,\zeta)
\end{align*}
for a sufficiently large $n$. This completes the proof.
\end{proof}
\subsection{The check of identifiability (\ref{model2iden})}\label{iden}
\ \vspace{2mm}\\
Assume that 
\begin{align*}
    {\bf{\Sigma}}_{2}(\theta_{2})={\bf{\Sigma}}_{2}(\theta_{2,0}).
\end{align*}
In an analogous manner to the proof of Appendix in Kusano and Uchida \cite{Kusano(BIC)}, we obtain
\begin{align*}
    \theta^{(1)}_2=\theta_{2,0}^{(1)},\quad 
    \theta^{(2)}_2=\theta_{2,0}^{(2)},\quad
    \theta^{(3)}_2=\theta_{2,0}^{(3)}
\end{align*}
and
\begin{align}
    \theta^{(13)}_2=\theta_{2,0}^{(13)},\quad 
    \theta^{(14)}_2=\theta_{2,0}^{(14)},\quad
    \theta^{(15)}_2=\theta_{2,0}^{(15)},\quad
    \theta^{(16)}_2=\theta_{2,0}^{(16)},\quad
    \theta^{(17)}_2=\theta_{2,0}^{(17)}.\label{13-17}
\end{align}
The (1,1) and (1,5)-th elements of 
\begin{align}
    {\bf{\Sigma}}^{12}_{2}(\theta_{2})={\bf{\Sigma}}^{12}_{2}(\theta_{2,0}) \label{sigma12}
\end{align}
are
\begin{align*}
    (\theta_2^{(11)}-\theta_{2,0}^{(11)})\theta_{2,0}^{(13)}=0,\quad
    (\theta_2^{(12)}-\theta_{2,0}^{(12)})\theta_{2,0}^{(13)}=0,
\end{align*}
and $\theta^{(13)}_{2,0}$ is not zero, which gives
\begin{align}
    \theta_2^{(11)}=\theta_{2,0}^{(11)},\quad
    \theta_2^{(12)}=\theta_{2,0}^{(12)}. \label{11,12}
\end{align}
Since the (1,2), (1,3), (1,4), (1,7) and (1,8)-th elements of (\ref{sigma12}) are
\begin{align*}
    &(\theta^{(4)}_{2}-\theta^{(4)}_{2,0})\theta_{2,0}^{(11)}\theta_{2,0}^{(13)}=0,\\
    &(\theta^{(5)}_{2}-\theta^{(5)}_{2,0})\theta_{2,0}^{(11)}\theta_{2,0}^{(13)}=0,\\
    &(\theta^{(6)}_{2}-\theta^{(6)}_{2,0})\theta_{2,0}^{(11)}\theta_{2,0}^{(13)}=0,\\
    &(\theta^{(9)}_{2}-\theta^{(9)}_{2,0})\theta_{2,0}^{(12)}\theta_{2,0}^{(13)}=0
\end{align*}
and
\begin{align*}
    &(\theta^{(10)}_{2}-\theta^{(10)}_{2,0})\theta_{2,0}^{(12)}\theta_{2,0}^{(13)}=0,
\end{align*}
by using the fact that $\theta_{2,0}^{(11)}$, $\theta_{2,0}^{(12)}$ and $\theta^{(13)}_{2,0}$ are not zero, we have
\begin{align}
    \theta_2^{(4)}=\theta_{2,0}^{(4)},\quad
    \theta_2^{(5)}=\theta_{2,0}^{(5)},\quad
    \theta_2^{(6)}=\theta_{2,0}^{(6)},\quad
    \theta_2^{(9)}=\theta_{2,0}^{(9)},\quad
    \theta_2^{(10)}=\theta_{2,0}^{(10)}.\label{4-10}
\end{align}
The (1,2) and (5,7)-th elements of 
\begin{align}
    {\bf{\Sigma}}^{22}_{2}(\theta_{2})={\bf{\Sigma}}^{22}_{2}(\theta_{2,0}) \label{sigma22}
\end{align}
are
\begin{align*}
    (\theta_2^{(26)}-\theta_{2,0}^{(26)})\theta_{2,0}^{(4)}=0
\end{align*}
and
\begin{align*}
    (\theta_2^{(27)}-\theta_{2,0}^{(27)})\theta_{2,0}^{(9)}=0,
\end{align*}
and $\theta^{(4)}_{2,0}$ and $\theta^{(9)}_{2,0}$ are not zero, which implies
\begin{align}
    \theta_2^{(26)}=\theta_{2,0}^{(26)},\quad
    \theta_2^{(27)}=\theta_{2,0}^{(27)}.\label{26,27}
\end{align}
Note that (\ref{13-17}), (\ref{11,12}) and (\ref{26,27}) lead to
\begin{align*}
    {\bf{\Gamma}}_2^{\theta}{\bf{\Sigma}}^{\theta}_{\xi\xi,2}{\bf{\Gamma}}_2^{\theta\top}+{\bf{\Sigma}}^{\theta}_{\zeta\zeta,2}={\bf{\Gamma}}_{2,0}^{\theta}{\bf{\Sigma}}^{\theta}_{\xi\xi,2,0}{\bf{\Gamma}}_{2,0}^{\theta\top}+{\bf{\Sigma}}^{\theta}_{\zeta\zeta,2,0}.
\end{align*}
The (1,6) and (5,6)-th elements of (\ref{sigma22}) are
\begin{align*}
    \begin{pmatrix}
        1 & 0 
    \end{pmatrix}({\bf{\Gamma}}_2^{\theta}{\bf{\Sigma}}^{\theta}_{\xi\xi,2}{\bf{\Gamma}}_2^{\theta\top}+{\bf{\Sigma}}^{\theta}_{\zeta\zeta,2})\begin{pmatrix}
        \theta_2^{(7)}\\
        \theta_2^{(8)}
    \end{pmatrix}=\begin{pmatrix}
        1 & 0 
    \end{pmatrix}({\bf{\Gamma}}_{2,0}^{\theta}{\bf{\Sigma}}^{\theta}_{\xi\xi,2,0}{\bf{\Gamma}}_{2,0}^{\theta\top}+{\bf{\Sigma}}^{\theta}_{\zeta\zeta,2,0})\begin{pmatrix}
        \theta_{2,0}^{(7)}\\
        \theta_{2,0}^{(8)}
    \end{pmatrix}
\end{align*}
and 
\begin{align*}
    \begin{pmatrix}
        0 & 1
    \end{pmatrix}({\bf{\Gamma}}_2^{\theta}{\bf{\Sigma}}^{\theta}_{\xi\xi,2}{\bf{\Gamma}}_2^{\theta\top}+{\bf{\Sigma}}^{\theta}_{\zeta\zeta,2})\begin{pmatrix}
        \theta_2^{(7)}\\
        \theta_2^{(8)}
    \end{pmatrix}=\begin{pmatrix}
        1 & 0 
    \end{pmatrix}({\bf{\Gamma}}_{2,0}^{\theta}{\bf{\Sigma}}^{\theta}_{\xi\xi,2,0}{\bf{\Gamma}}_{2,0}^{\theta\top}+{\bf{\Sigma}}^{\theta}_{\zeta\zeta,2,0})\begin{pmatrix}
        \theta_{2,0}^{(7)}\\
        \theta_{2,0}^{(8)}
    \end{pmatrix},
\end{align*}
so that
\begin{align*}
    ({\bf{\Gamma}}_{2,0}^{\theta}{\bf{\Sigma}}^{\theta}_{\xi\xi,2,0}{\bf{\Gamma}}_{2,0}^{\theta\top}+{\bf{\Sigma}}^{\theta}_{\zeta\zeta,2,0})\begin{pmatrix}
        \theta_2^{(7)}-\theta_{2,0}^{(7)}\\
        \theta_2^{(8)}-\theta_{2,0}^{(8)}
    \end{pmatrix}=0.
\end{align*}
Since ${\bf{\Sigma}}^{\theta}_{\zeta\zeta,2,0}$ is positive definite, we have
\begin{align*}
    {\bf{\Gamma}}_{2,0}^{\theta}{\bf{\Sigma}}^{\theta}_{\xi\xi,2,0}{\bf{\Gamma}}_{2,0}^{\theta\top}+{\bf{\Sigma}}^{\theta}_{\zeta\zeta,2,0}>0,
\end{align*}
which yields
\begin{align}
    \theta_2^{(7)}=\theta_{2,0}^{(7)},\quad \theta_2^{(8)}=\theta_{2,0}^{(8)}. \label{7,8}
\end{align}
As it follows from (\ref{13-17}), (\ref{11,12}), (\ref{4-10}), (\ref{26,27}) and (\ref{7,8}) that 
\begin{align*}
    {\bf{\Lambda}}_2^{\theta}\bigl({\bf{\Gamma}}_2^{\theta}{\bf{\Sigma}}^{\theta}_{\xi\xi,2}{\bf{\Gamma}}_2^{\theta\top}+{\bf{\Sigma}}^{\theta}_{\zeta\zeta,2}\bigr){\bf{\Lambda}}_2^{\theta\top}={\bf{\Lambda}}_{2,0}^{\theta}\bigl({\bf{\Gamma}}_{2,0}^{\theta}{\bf{\Sigma}}^{\theta}_{\xi\xi,2,0}{\bf{\Gamma}}_{2,0}^{\theta\top}+{\bf{\Sigma}}^{\theta}_{\zeta\zeta,2,0}\bigr){\bf{\Lambda}}_{2,0}^{\theta\top},
\end{align*}
(\ref{sigma22}) implies
\begin{align*}
    \theta_2^{(i)}=\theta_{2,0}^{(i)}
\end{align*}
for $i=18,\ldots,25$, which completes the proof.
\end{document}